\documentclass[a4paper,11pt]{article}
\usepackage{amsmath,amsthm,amssymb}
\usepackage{graphicx}
\usepackage[all]{xy}
\title{Norm computation and analytic continuation of \\ vector valued holomorphic discrete series representations%
\thanks{This work was supported by Grant-in-Aid for JSPS Fellows (25$\cdot$7147).}}
\author{Ryosuke Nakahama\thanks{Email: nakahama@ms.u-tokyo.ac.jp} \\
\textit{Graduate School of Mathematical Sciences, the University of Tokyo,} \\
\textit{3-8-1 Komaba Meguro-ku Tokyo 153-8914, Japan}}
\date{\today}
\makeatletter
\@addtoreset{equation}{section}
\makeatother

\newtheorem{theorem}{Theorem}[section]
\newtheorem*{theorem*}{Theorem}

\newtheorem{proposition}[theorem]{Proposition}
\newtheorem{corollary}[theorem]{Corollary}
\newtheorem*{corollary*}{Corollary}
\newtheorem{lemma}[theorem]{Lemma}

\newtheorem{remark}[theorem]{Remark}
\newtheorem{conjecture}[theorem]{Conjecture}
\newcommand{\ds}{\displaystyle}

\newcommand{\bk}{\mathbf{k}}
\newcommand{\bl}{\mathbf{l}}
\newcommand{\bm}{\mathbf{m}}
\newcommand{\bn}{\mathbf{n}}
\newcommand{\bs}{\mathbf{s}}
\newcommand{\bzero}{\mathbf{0}}
\newcommand{\bone}{\mathbf{1}}
\newcommand{\BN}{\mathbb{N}}
\newcommand{\BZ}{\mathbb{Z}}
\newcommand{\BR}{\mathbb{R}}
\newcommand{\BC}{\mathbb{C}}
\newcommand{\BH}{\mathbb{H}}
\newcommand{\BO}{\mathbb{O}}
\newcommand{\fa}{\mathfrak{a}}
\newcommand{\fb}{\mathfrak{b}}

\newcommand{\fg}{\mathfrak{g}}
\newcommand{\fh}{\mathfrak{h}}
\newcommand{\fk}{\mathfrak{k}}
\newcommand{\fl}{\mathfrak{l}}
\newcommand{\fm}{\mathfrak{m}}
\newcommand{\fn}{\mathfrak{n}}
\newcommand{\fp}{\mathfrak{p}}

\newcommand{\cH}{\mathcal{H}}
\newcommand{\cO}{\mathcal{O}}
\newcommand{\cP}{\mathcal{P}}
\newcommand{\cV}{\mathcal{V}}
\newcommand{\rA}{\mathrm{A}}
\newcommand{\rT}{\mathrm{T}}
\newcommand{\rank}{\operatorname{rank}}
\newcommand{\End}{\operatorname{End}}
\newcommand{\Tr}{\operatorname{Tr}}
\newcommand{\tr}{\operatorname{tr}}
\newcommand{\Det}{\operatorname{Det}}

\newcommand{\sgn}{\operatorname{sgn}}

\renewcommand{\Re}{\operatorname{Re}}

\begin{document}
\maketitle

\begin{abstract}
In this paper we compute explicitly the norm of the vector-valued holomorphic discrete series representations, 
when its $K$-type is ``almost multiplicity-free''. 
As an application, we discuss the properties of highest weight modules, such as 
unitarizability, reducibility and composition series. 
\bigskip

\end{abstract}

\section{Introduction}
The purpose of this paper is to compute explicitly the norm of the vector-valued holomorphic discrete series representations, 
and to study the properties of the highest weight modules, such as unitarizabily, reducibility and composition series. 

Let $G$ be a simple Lie group, such that its maximal compact subgroup $K$ has a non-discrete center. 
Then it is known that there exist a linear subspace $\fp^+\subset\fg^\BC$ and a bounded domain $D\subset \fp^+$ such that 
the symmetric space $G/K$ is diffeomorphic to $D$. Therefore $G/K$ becomes a complex manifold. 
Let $(\tau,V)$ be a finite-dimensional holomorphic representation of $K^\BC$, 
and $\chi^{-\lambda}$ be a suitable character of the universal covering group $\tilde{K}^\BC$. 
Then we can consider the representation of the universal covering group $\tilde{G}$ on the space of holomorphic sections 
of the equivariant vector bundle on $G/K$ with fiber $V\otimes\chi^{-\lambda}$, 
\[ \tilde{G}\curvearrowright \Gamma_\cO(G/K,\tilde{G}\times_{\tilde{K}}(V\otimes\chi^{-\lambda})). \]
Since $D\simeq G/K$ is contractible, this space is isomorphic to the space of $V$-valued holomorphic functions on $D$, 
\[ \Gamma_\cO(G/K,\tilde{G}\times_{\tilde{K}}(V\otimes\chi^{-\lambda}))\simeq \cO(D,V). \]
Then the infinitesimal action of the Lie subalgebra $\fp^+\subset\fg^\BC$ on $\cO(D,V)$ is given by 
1st order differential operators with constant coefficients, and thus it annihilates constant functions in $\cO(D,V)$. 
Such representations are called the highest weight representations. 
Also, if $\lambda\in\BR$ is sufficiently large, then this representation preserves an inner product 
which is given by an explicit integral on $D$. Such representations are called the holomorphic discrete series representations. 

For example, let $G:=Sp(r,\BR)$, realized explicitly as 
\[ Sp(r,\BR)=\left\{g\in GL(2r,\BC):
g\begin{pmatrix}0&I_r\\-I_r&0\end{pmatrix}{}^t\hspace{-1pt}g=\begin{pmatrix}0&I_r\\-I_r&0\end{pmatrix},\;
g\begin{pmatrix}0&I_r\\I_r&0\end{pmatrix}=\begin{pmatrix}0&I_r\\I_r&0\end{pmatrix}\bar{g}\right\}. \]
Then $G/K=Sp(r,\BR)/U(r)$ is diffeomorphic to 
\[ D:=\{ w\in\mathrm{Sym}(r,\BC):I_r-ww^*\text{ is positive definite.}\}. \]
Let $(\tau,V)$ be a representation of $K^\BC=GL(r,\BC)$. Then the universal covering group 
$\tilde{G}=\widetilde{Sp}(r,\BR)$ acts on $\cO(D,V)$ by 
\[ \tau_\lambda\left(\begin{pmatrix}a&b\\c&d\end{pmatrix}^{-1}\right)f(w)
=\det(cw+d)^{-\lambda}\tau\left({}^t\hspace{-1pt}(cw+d)\right)f\left((aw+b)(cw+d)^{-1}\right). \]
We note that $\det(cw+d)^{-\lambda}$ is not well-defined as a function on $G\times D$, but is well-defined as a function on 
the universal covering space $\tilde{G}\times D$. 
If $\Re\lambda$ is sufficiently large, then this preserves the sesquilinear form 
\[ \langle f,h\rangle_{\lambda,\tau}
:=\frac{c_\lambda}{\pi^{r(r+1)/2}}\int_D\left(\tau((I-ww^*)^{-1})f(w),h(w)\right)_\tau \det(I-ww^*)^{\lambda-(r+1)}dw, \]
that is, $\langle \tau_{\lambda}(g)f,\tau_{\bar{\lambda}}(g)h\rangle_{\lambda,\tau}=\langle f,h\rangle_{\lambda,\tau}$ holds 
for any $f,h\in \cO(D,V)$ with finite norms, and for any $g\in\tilde{G}$. 
Therefore $\tau_\lambda$ gives a holomorphic discrete series representation of $\tilde{G}$ if $\lambda\in\BR$ and the above norm converges 
for some nonzero function in $\cO(D,V)$. In this case the corresponding Hilbert space $\cH_\lambda(D,V)\subset \cO(D,V)$ 
has the reproducing kernel 
\[ K_{\lambda,\tau}(z,w):=\det(I_r-zw^*)^{-\lambda}\tau(I_r-zw^*)\in \cO(D\times\overline{D},\mathrm{End}(V)), \]
if we choose the normalizing constant $c_\lambda$ suitably. When $r=1$, then we have $G=SU(1,1)$ and $D=\{w\in\BC:|w|<1\}$, 
and the action $\tau_\lambda$ of $\widetilde{SU}(1,1)$ on $\cO(D)$ reduces to the simplest example 
\[ \tau_\lambda\left(\begin{pmatrix}a&b\\c&d\end{pmatrix}^{-1}\right)f(w)
=(cw+d)^{-\lambda}f\left(\frac{aw+b}{cw+d}\right), \]
with the invariant inner product and the reproducing kernel 
\begin{gather}
\langle f,h\rangle_{\lambda}
=\frac{\lambda-1}{\pi}\int_{|w|<1}f(w)\overline{h(w)}(I-|w|^2)^{\lambda-2}dw, \label{normsu11} \\
K_{\lambda}(z,w)=(1-z\bar{w})^{-\lambda} \in\cO(D\times\overline{D}). \label{kersu11}
\end{gather}

We return to the general case. The question of when the highest weight representations are unitarizable 
is studied by e.g. Berezin \cite{B}, Clerc \cite{C}, Vergne-Rossi \cite{VR}, and Wallach \cite{W}, 
and completely classified by Enright-Howe-Wallach \cite{EHW} and Jakobsen \cite{J} by different methods. 
In \cite{EHW} and \cite{J} they used purely algebraic methods. 

On the other hand, the analytical proof, the proof using explicit norm computation, was only partially successful. 
When the fiber $(\tau,V)$ is trivial, this is studied by e.g. Hua \cite{H}, Upmeier \cite{U}, and \O rsted \cite{O}, 
and completely done by Faraut-Kor\'anyi \cite{FK0}. However, vector-valued cases are not computed yet except for a few cases, 
e.g. the case when $(\tau,V)$ is a defining representation of $K^\BC=GL(s,\BC)$ (\O rsted-Zhang \cite{OZ1}, \cite{OZ2}), 
and the case when $G$ is of real rank 1 (Hwang-Liu-Zhang \cite{HLZ}). 

Now we explain how the explicit norm computation gives informations on unitarizability and reducibility in the simplest example. 
Let $G=SU(1,1)$. Then the $\tilde{G}$-invariant inner product (\ref{normsu11}) converges for any polynomial $f,h\in\cP(\BC)$ if $\Re\lambda>1$, 
but does not converge for any non-zero polynomial $f,h\in\cP(\BC)$ if $\Re\lambda\le 1$. 
Suppose $f,h$ has a Taylor expansion $f(w)=\sum_m a_mw^m$, $h(w)=\sum_mb_mw^m$. 
Then for $\Re\lambda>1$, we can compute $\langle f,h\rangle_\lambda$ explicitly as 
\[ \langle f,h\rangle_\lambda=\sum_{m=0}^\infty\frac{m!}{(\lambda)_m}a_m\overline{b_m}, \]
where $(\lambda)_m:=\lambda(\lambda+1)\cdots(\lambda+m-1)$. This expression is available even if $\Re\lambda\le 1$, and is also 
$(\fg,\tilde{K})$-invariant. As a result, the reproducing kernel $K_\lambda(z,w)$ in (\ref{kersu11}) is expanded as 
\[ K_{\lambda}(z,w)=(1-z\bar{w})^{-\lambda}=\sum_{m=0}^\infty\frac{(\lambda)_m}{m!}z^m\bar{w}^m. \]
This expression is also available when $\Re\lambda\le 1$. This kernel function is positive definite if $\lambda\ge 0$, 
and thus $(\tau_\lambda,\cO(D))$ is unitarizable if $\lambda\ge 0$. Here, when $\lambda=0$, the corresponding Hilbert space 
consists of only 0th order polynomials, and is of 1-dimensional. 
Also, for $\lambda=-l\in \BZ_{\le 0}$, the sesquilinear forms 
\begin{gather}
\langle f,h\rangle_{-l}=\sum_{m=0}^l \frac{m!}{(-l)_m}a_m\overline{b_m},\label{subsu11} \\
\lim_{\lambda\to -l}(\lambda+l)\langle f,h\rangle_\lambda
=\frac{1}{(-l)_l}\sum_{m=l+1}^\infty\frac{m!}{(1)_{m-l-1}}a_m\overline{b_m} \label{quotientsu11}
\end{gather}
are well-defined and $(\fg,K)$-invariant on $\cP_{\le l}(\BC)$, the space of polynomials of order at most $l$, 
and on $\cP(\BC)/\cP_{\le l}(\BC)$ respectively. Moreover (\ref{quotientsu11}) is definite. 
Therefore $\cP_{\le l}(\BC)$ gives a $(\fg,K)$-submodule, and 
$\cP(\BC)/\cP_{\le l}(\BC)$ gives a infinitesimally unitary $(\fg,K)$-module. 

To compute the norm for general $G$, we use the $K$-type decomposition of $\cO(D,V)_K=\cP(\fp^+,V)$ instead of the Taylor expansion, 
fix a $K$-invariant norm $\Vert\cdot\Vert_{F,\tau}$ on $\cP(\fp^+,V)$ independent of $\lambda$ (see (\ref{Fischer})), and compare 
$\Vert\cdot\Vert_{\lambda,\tau}$ and $\Vert\cdot\Vert_{F,\tau}$ on each $K$-type. Let 
\[ \cO(D,V)_K=\cP(\fp^+,V)=\bigoplus_i W_i \]
be a $K$-type decomposition such that each $W_i$ is orthogonal to the others with respect to $\langle\cdot,\cdot\rangle_{F,\tau}$. 
Then since $\Vert\cdot\Vert_{\lambda,\tau}$ and $\Vert\cdot\Vert_{F,\tau}$ are both $K$-invariant, 
the ratio of two norms are constant on $W_i$. We denote this ratio by $R_i(\lambda)$. 
Moreover, if $W_i\perp W_j$ with respect to $\langle\cdot,\cdot\rangle_{F,\tau}$ implies 
$W_i\perp W_j$ with respect to $\langle\cdot,\cdot\rangle_{\lambda,\tau}$ 
(for example, if $\cP(\fp^+,V)$ is $K$-multiplicity free), then we have 
\[ \Vert f\Vert_{\lambda,\tau}^2=\sum_iR_i(\lambda)\Vert f_i\Vert_{F,\tau}^2 \qquad (f\in\cO(\fp^+,V)) \]
where $f_i$ is the orthogonal projection of $f$ onto $W_i$, and the reproducing kernel $K_{\lambda,\tau}(z,w)$ is expanded as 
\[ K_{\lambda,\tau}(z,w)=\sum_iR_i(\lambda)^{-1}K_i(z,w), \]
where $K_i(z,w)$ is the reproducing kernel of $W_i$ with respect to $\langle\cdot,\cdot\rangle_{F,\tau}$. 
Similarly to the $SU(1,1)$ case, if we compute $R_i(\lambda)$ explicitly, 
then we can determine completely when the representation is unitarizable, or reducible, and 
can get some informations on composition series. 

Since the above argument is available only if $W_i\perp W_j$ with respect to $\langle\cdot,\cdot\rangle_{F,\tau}$ implies 
$W_i\perp W_j$ with respect to $\langle\cdot,\cdot\rangle_{\lambda,\tau}$, 
we specialize our interest to $(G,V)$'s in the following table. \\
\hspace{-30pt}
\begin{minipage}{\linewidth}
\vspace{5pt}
\begin{tabular}{|c|c|c|c|}\hline
$G$ & $K$ & $V$ & Where \\ \hline
$Sp(r,\BR)$ & $U(r)$ & $\bigwedge^k(\BC^r)^\vee$ \quad ($0\le k\le r-1$) & Thm \ref{sprr} \\ \hline
$SU(q,s)$ & $S(U(q)\times U(s))$ & $\BC\boxtimes V'$ \quad ($V'$: any irrep of $U(s)$) & 
\shortstack{Thm \ref{suqs} ($q\ge s$)\\ Thm \ref{suqsnontube} ($q<s$)} \\ \hline
$SO^*(2s)$ & $U(s)$ & \shortstack{$S^k(\BC^s)^\vee$\\ $S^k(\BC^s)\otimes\det^{-k/2}$} \quad ($k\in \BZ_{\ge 0}$) &
\shortstack{Thm \ref{sostareven} ($s$ even)\\ Thm \ref{sostarodd1}, \ref{sostarodd2} ($s$ odd)} \\ \hline
$Spin_0(2,n)$ & \shortstack{$(Spin(2)\times$ \\ $Spin(n))/\BZ_2$} 
& \shortstack{$\BC_{-k}\boxtimes V_{(k,\ldots,k,\pm k)}$ \quad ($k\in\frac{1}{2}\BZ_{\ge 0}$,\, $n$ even)
\\$\BC_{-k}\boxtimes V_{(k,\ldots,k)}$ \quad ($k\in\{0,\frac{1}{2}\}$,\, $n$ odd)} & Thm \ref{spin2n} \\ \hline
$E_{6(-14)}$ & $SO(2)\times Spin(10)$ & $\BC_{-k/2}\boxtimes\cH^k(\BR^{10})$ \quad ($k\in \BZ_{\ge 0}$) & 
Prop \ref{e6(-14)}, Conj \ref{e6(-14)conj} \\ \hline
$E_{7(-25)}$ & $SO(2)\times E_6$ & $\BC$ & Already done in \cite{FK} \\ \hline
\end{tabular}
\vspace{5pt}
\end{minipage}\\
In the above cases, except for $G=SU(q,s)$ case, $\cP(\fp^+,V)$ is multiplicity-free under $K$, 
which is proved by direct computation of $K$-type decomposition. 
We can also prove multiplicity-freeness a priori by using \cite[Theorem 2]{K}. 
In $G=SU(q,s)$ case, $\cP(\fp^+,V)$ is not multiplicity-free in general, but each $K$-isotypic component sits in a single polynomial space, 
and thus the arguments explained above is still available. 

When $G$ is of tube type or $G=SU(q,s)$ with $q\ge s$, which we deal with in Section 4, 
we can compute the norm in a uniform way, by generalizing the technique used by Faraut-Kor\'anyi \cite{FK}. 
For these cases, the fibers $V$ in the above table satisfy the condition that 
they remain irreducible even if restricted to some subgroup $K_L$ of $K$, 
and this condition allows us to compute the norm explicitly. 
The same condition also appears in e.g. \cite{C}, \cite{HN}. 
In these papers they got some necessary condition on the unitarizability of highest weight representations, 
by considering when the reproducing kernel on the tube domain becomes a Laplace transform of some measure. 
Under the assumption that $V|_{K_L}$ is irreducible, the necessary and sufficient condition is also computable, 
and therefore this assumption seems to be natural. 

However, when $G$ is of non-tube type, there is no such uniform way to compute the norm at this time, 
and we do this by purely case-by-case analysis. For example, we use an embedding of $G$ into a larger group, 
or use an embedding of some smaller subgroup into $G$. We deal with such cases in Section 5. 

We enumerate the main results of this paper. 
\begin{theorem}[Theorem \ref{sprr}]
When $G=Sp(r,\BR)$, and $(\tau,V)=(\tau_{\varepsilon_1+\cdots+\varepsilon_k}^\vee,V_{\varepsilon_1+\cdots+\varepsilon_k}^\vee)$, 
$\Vert\cdot\Vert_{\lambda,\tau}^2$ converges if $\Re\lambda>r$, 
the $K$-type decomposition of $\cO(D,V)_K$ is given by 
\[ \cP(\fp^+)\otimes V_{\varepsilon_1+\cdots+\varepsilon_k}^\vee
=\bigoplus_{\bm\in\BZ^r_{++}}\bigoplus_{\substack{\bk\in\{0,1\}^r,\,|\bk|=k\\ \bm+\bk\in\BZ^r_+}}V_{2\bm+\bk}^\vee, \]
and for $f\in V_{2\bm+\bk}^\vee$, the ratio of norms is given by 
\begin{align*}
\frac{\Vert f\Vert_{\lambda,\tau_{\varepsilon_1+\cdots+\varepsilon_k}^\vee}^2}
{\Vert f\Vert_{F,\tau_{\varepsilon_1+\cdots+\varepsilon_k}^\vee}^2}
&=\frac{\prod_{j=1}^k\left(\lambda-\frac{1}{2}(j-1)\right)}{\prod_{j=1}^r\left(\lambda-\frac{1}{2}(j-1)\right)_{m_j+k_j}}\\
&=\frac{1}{\prod_{j=1}^k\left(\lambda-\frac{1}{2}(j-1)+1\right)_{m_j+k_j-1}
\prod_{j=k+1}^r\left(\lambda-\frac{1}{2}(j-1)\right)_{m_j+k_j}}. 
\end{align*}
\end{theorem}

\begin{theorem}[Theorem \ref{suqs}, \ref{suqsnontube}]
When $G=SU(q,s)$, and $(\tau,V)=(\bone^{(q)}\boxtimes\tau_\bk^{(s)},\BC\otimes V_\bk^{(s)})$ $(\bk\in\BZ_{++}^s)$, 
$\Vert\cdot\Vert_{\lambda,\tau}^2$ converges if $\Re\lambda+k_s>q+s-1$, 
the $K$-type decomposition of $\cO(D,V)_K$ is given by 
\[ \cP(\fp^+)\otimes \left(\BC\boxtimes V_\bk^{(s)}\right)
=\bigoplus_{\bm\in\BZ^s_{++}}\bigoplus_{\bn\in\bm+\mathrm{wt}(\bk)}c^\bn_{\bk,\bm}V_\bm^{(q)\vee}\boxtimes V_\bn^{(s)}, \]
and for $f\in V_\bm^{(q)\vee}\boxtimes V_\bn^{(s)}$, the ratio of norms is given by 
\[ \frac{\Vert f\Vert_{\lambda,\bone^{(q)}\boxtimes\tau_\bk^{(s)}}^2}{\Vert f\Vert_{F,\bone^{(q)}\boxtimes\tau_\bk^{(s)}}^2}
=\frac{\prod_{j=1}^s(\lambda-(j-1))_{k_j}}{\prod_{j=1}^s(\lambda-(j-1))_{n_j}}
=\frac{1}{\prod_{j=1}^s(\lambda-(j-1)+k_j)_{n_j-k_j}}. \]
\end{theorem}

\begin{theorem}[Theorem \ref{sostareven}]
When $G=SO^*(4r)$, and $(\tau,V)=(\tau_{(k,0,\ldots,0)}^\vee,V_{(k,0,\ldots,0)}^\vee)$, 
$\Vert\cdot\Vert_{\lambda,\tau}^2$ converges if $\Re\lambda>4r-3$, 
the $K$-type decomposition of $\cO(D,V)_K$ is given by 
\[ \cP(\fp^+)\otimes V_{(k,0,\ldots,0)}^\vee
=\bigoplus_{\bm\in\BZ^r_{++}}\bigoplus_{\substack{\bk\in(\BZ_{\ge 0})^r,\;|\bk|=k\\ 0\le k_j\le m_{j-1}-m_j}} 
V_{(m_1+k_1,m_1,m_2+k_2,m_2,\ldots,m_r+k_r,m_r)}^\vee, \]
and for $f\in V_{(m_1+k_1,m_1,m_2+k_2,m_2,\ldots,m_r+k_r,m_r)}^\vee$, the ratio of norms is given by 
\[ \frac{\Vert f\Vert_{\lambda,\tau_{(k,0,\ldots,0)}^\vee}^2}{\Vert f\Vert_{F,\tau_{(k,0,\ldots,0)}^\vee}^2}
=\frac{(\lambda)_k}{\prod_{j=1}^r(\lambda-2(j-1))_{m_j+k_j}}
=\frac{1}{(\lambda+k)_{m_1+k_1-k}\prod_{j=2}^r(\lambda-2(j-1))_{m_j+k_j}}. \]

When $G=SO^*(4r)$, and $(\tau,V)=(\tau_{(k/2,\ldots,k/2,-k/2)}^\vee,V_{(k/2,\ldots,k/2,-k/2)}^\vee)$, 
$\Vert\cdot\Vert_{\lambda,\tau}^2$ converges if $\Re\lambda>4r-3$, 
the $K$-type decomposition of $\cO(D,V)_K$ is given by 
\[ \cP(\fp^+)\otimes V_{\left(\frac{k}{2},\ldots,\frac{k}{2},-\frac{k}{2}\right)}^\vee
=\bigoplus_{\bm\in\BZ^r_{++}}\bigoplus_{\substack{\bk\in(\BZ_{\ge 0})^r,\;|\bk|=k\\ 0\le k_j\le m_j-m_{j+1}}}
V_{(m_1,m_1-k_1,m_2,m_2-k_2,\ldots,m_r,m_r-k_r)+\left(\frac{k}{2},\ldots,\frac{k}{2}\right)}^\vee, \]
and for $f\in V_{(m_1,m_1-k_1,m_2,m_2-k_2,\ldots,m_r,m_r-k_r)+\left(\frac{k}{2},\ldots,\frac{k}{2}\right)}^\vee$, 
the ratio of norms is given by 
\begin{align*}
\frac{\Vert f\Vert_{\lambda,\tau_{(k/2,\ldots,k/2,-k/2)}^\vee}^2}{\Vert f\Vert_{F,\tau_{(k/2,\ldots,k/2,-k/2)}^\vee}^2}
&=\frac{\prod_{j=1}^{r-1}(\lambda-2(j-1))_k}{\prod_{j=1}^r(\lambda-2(j-1))_{m_j-k_j+k}}\\
&=\frac{1}{\prod_{j=1}^{r-1}(\lambda+k-2(j-1))_{m_j-k_j}(\lambda-2(r-1))_{m_r-k_r+k}}. 
\end{align*}
\end{theorem}

\begin{theorem}[Theorem \ref{sostarodd1}, \ref{sostarodd2}]
When $G=SO^*(4r+2)$ and $(\tau,V)=(\tau_{(k,0,\ldots,0)}^{\vee},V_{(k,0,\ldots,0)}^{\vee})$, 
$\Vert\cdot\Vert_{\lambda,\tau}^2$ converges if $\Re\lambda>4r-1$, 
the $K$-type decomposition of $\cO(D,V)_K$ is given by 
\[ \cP(\fp^+)\otimes V_{(k,0,\ldots,0)}^{\vee}
=\bigoplus_{\bm\in\BZ^r_{++}}\bigoplus_{\substack{\bk\in(\BZ_{\ge 0})^{r+1};|\bk|=k\\ 0\le k_j\le m_{j-1}-m_j}}
V_{(m_1+k_1,m_1,m_2+k_2,m_2,\ldots,m_r+k_r,m_r,k_{r+1})}^{\vee}, \]
and for $f\in V_{(m_1+k_1,m_1,m_2+k_2,m_2,\ldots,m_r+k_r,m_r,k_{r+1})}^{\vee}$, the ratio of norms is given by 
\begin{align*}
\frac{\Vert f\Vert_{\lambda,\tau_{(k,0,\ldots,0)}^{\vee}}^2}{\Vert f\Vert_{F,\tau_{(k,0,\ldots,0)}^{\vee}}^2}
=&\frac{(\lambda)_k}{\prod_{j=1}^r(\lambda-2(j-1))_{m_j+k_j}(\lambda-2r)_{k_{r+1}}}\\
=&\frac{1}{(\lambda+k)_{m_1+k_1-k}\prod_{j=2}^r(\lambda-2(j-1))_{m_j+k_j}(\lambda-2r)_{k_{r+1}}}. 
\end{align*}

When $G=SO^*(4r+2)$ and $(\tau,V)=(\tau_{(k/2,\ldots,k/2,-k/2)}^{\vee},V_{(k/2,\ldots,k/2,-k/2)}^{\vee})$, 
$\Vert\cdot\Vert_{\lambda,\tau}^2$ converges if $\Re\lambda>4r-1$, 
the $K$-type decomposition of $\cO(D,V)_K$ is given by 
\[ \cP(\fp^+)\otimes V_{\left(\frac{k}{2},\ldots,\frac{k}{2},-\frac{k}{2}\right)}^{\vee}
=\!\!\bigoplus_{\bm\in\BZ^r_{++}}\bigoplus_{\substack{\bk\in(\BZ_{\ge 0})^{r+1};|\bk|=k\\ 0\le k_j\le m_j-m_{j+1}\\ 0\le k_r\le m_r}}
\!\!\!\!V_{(m_1,m_1-k_1,m_2,m_2-k_2,\ldots,m_r,m_r-k_r,-k_{r+1})+\left(\frac{k}{2},\ldots,\frac{k}{2}\right)}^{\vee}, \]
and for $f\in V_{(m_1,m_1-k_1,m_2,m_2-k_2,\ldots,m_r,m_r-k_r,-k_{r+1})+\left(\frac{k}{2},\ldots,\frac{k}{2}\right)}^{\vee}$, 
the ratio of norms is given by 
\begin{align*}
\frac{\Vert f\Vert_{\lambda,\tau_{(k/2,\ldots,k/2,-k/2)}^{\vee}}^2}
{\Vert f\Vert_{F,\tau_{(k/2,\ldots,k/2,-k/2)}^{\vee}}^2}
=&\frac{\prod_{j=1}^r\left(\lambda-2(j-1)\right)_k}{\prod_{j=1}^r\left(\lambda-2(j-1)\right)_{m_j-k_j+k}\left(\lambda-2r+1\right)_{k-k_{r+1}}}\\
=&\frac{1}{\prod_{j=1}^r\left(\lambda+k-2(j-1)\right)_{m_j-k_j}\left(\lambda-2r+1\right)_{k-k_{r+1}}}. 
\end{align*}
\end{theorem}

\begin{theorem}[Theorem \ref{spin2n}]
When $G=Spin_0(2,n)$ and 
\[ (\tau,V)=\left\{\begin{array}{lll} 
(\chi^{-k}\boxtimes \tau_{(k,\ldots,k,\pm k)},\BC_{-k}\otimes V_{(k,\ldots,k,\pm k)})
&\left(k\in\frac{1}{2}\BZ_{\ge 0}\right)&(n:\text{even}),\\
(\chi^{-k}\boxtimes \tau_{(k,\ldots,k)},\BC_{-k}\otimes V_{(k,\ldots,k)})
&\left(k=0,\frac{1}{2}\right)&(n:\text{odd}), \end{array}\right. \]
$\Vert\cdot\Vert_{\lambda,\tau}^2$ converges if $\Re\lambda>n-1$, 
the $K$-type decomposition of $\cO(D,V)_K$ is given by 
\begin{align*}
\cP(\fp^+)\otimes V=\begin{cases}
\ds \bigoplus_{\bm\in\BZ^2_{++}}\bigoplus_{\substack{-k\le l\le k\\ m_1-m_2+l\ge k}} 
\BC_{-(m_1+m_2+k)}\boxtimes V_{(m_1-m_2+l,k,\ldots,k,\pm l)} &(n:\text{even}), \\ 
\ds \bigoplus_{\bm\in\BZ^2_{++}}\bigoplus_{\substack{-k\le l\le k\\ m_1-m_2+l\ge k}} 
\BC_{-(m_1+m_2+k)}\boxtimes V_{(m_1-m_2+l,k,\ldots,k,|l|)} &(n:\text{odd}), \end{cases}
\end{align*}
and for $f\in \BC_{-(m_1+m_2+k)}\boxtimes V_{(m_1-m_2+l,k,\ldots,k,\pm l)}$ or 
$\BC_{-(m_1+m_2+k)}\boxtimes V_{(m_1-m_2+l,k,\ldots,k,|l|)}$, the ratio of norms is given by 
\[ \frac{\Vert f\Vert_{\lambda,\tau}^2}{\Vert f\Vert_{F,\tau}^2}
=\frac{(\lambda)_{2k}}{(\lambda)_{m_1+k+l}\left(\lambda-\frac{n-2}{2}\right)_{m_2+k-l}}
=\frac{1}{(\lambda+2k)_{m_1-k+l}\left(\lambda-\frac{n-2}{2}\right)_{m_2+k-l}}. \]
\end{theorem}

We also state the conjecture on $E_{6(-14)}$ in Section \ref{secte6}. From these theorems 
we can get informations on unitarizability, reducibility and composition series. 

This paper is organized as follows. In Section 2 we prepare some notations and review some facts on 
Lie algebras of Hermitian type and Jordan triple systems. 
In Section 3 we state and prove the theorems (Theorem \ref{keythm}, Corollary \ref{tubecor}) which plays a key role in this paper. 
In Section 4 and 5 we compute the norm explicitly. 
In Section 4 we deal with the cases that the norm is computable directly from the theorem in Section 3, 
and in Section 5 we deal with the cases that need more techniques. 
In Section 6 we apply the results on norm computation to the problems on unitarizabily, reducibility and composition series.

\section{Preliminaries}
\subsection{Root decomposition}\label{root}
Let $\fg=\fk\oplus\fp$ be a simple Hermitian Lie algebra, that is, 
the maximal compact part $\fk$ has a 1-dimensional center. 
We take an element $z$ from the center of $\fk$ such that the eigenvalues of $\mathrm{ad}(z)$ are 
$+\sqrt{-1}$, $0$, $-\sqrt{-1}$, and let 
\[ \fg^\BC=\fp^+\oplus\fk^\BC\oplus\fp^- \]
be the corresponding eigenspace decomposition. We denote the Cartan involution of $\fg^\BC$ 
(the anti-holomorphic extension of the Cartan involution on $\fg$) by $\vartheta$. 
Then $\fp^+$ has a Hermitian Jordan triple system structure with the product 
\[ (x,y,z)\longmapsto \{x,y,z\}:=-\frac{1}{2}[[x,\vartheta y],z], \qquad x,y,z\in\fp^+. \]
We take a maximal abelian subalgebra $\fh\subset\fk$. 
Then $\fh^\BC$ becomes simaltaneously a Cartan subalgebra of both $\fk^\BC$ and $\fg^\BC$. 
Let $\Delta=\Delta(\fg^\BC,\fh^\BC)$ be the root system. We denote by $\Delta_{\fp^\pm}$, $\Delta_{\fk^\BC}$ 
the all roots $\alpha$ such that the corresponding root space $\fg_\alpha^\BC$ is contained in $\fp^\pm$, $\fk^\BC$ respectively. 
Also, we take a positive root system $\Delta_+=\Delta_+(\fg^\BC,\fh^\BC)$ such that $\Delta_{\fp^+}\subset\Delta_+$, 
and we denote $\Delta_{\fk^\BC,+}:=\Delta_{\fk^\BC}\cap\Delta_+$. We set $n:=\dim\fp^+$, $r:=\rank_\BR\fg$. 

We take the set of strongly orthogonal roots $\{\gamma_1,\ldots,\gamma_r\}\subset\Delta_{\fp^+}$ such that 
\begin{enumerate}
\item $\gamma_1$ is the highest root in $\Delta_{\fp^+}$, 
\item $\gamma_k$ is the root in $\Delta_{\fp^+}$ which is highest among the roots strongly orthogonal to each $\gamma_j$ with $1\le j\le k-1$, 
\end{enumerate}
and for each $j$, we take $e_j\in\fg_{\gamma_j}^\BC$ such that $-[[e_j,\vartheta e_j],e_j]=2e_j$. 
Then $\fa:=\bigoplus_{j=1}^r\BR(e_j-\vartheta e_j)\subset\fp$ is a maximal abelian subalgebra in $\fp$, 
and $\{e_1,\ldots,e_r\}$ is a Jordan frame on $\fp^+$. We set $e:=\sum_{j=1}^re_j\in\fp^+$ (a maximal tripotent), 
and $h:=-[e,\vartheta e]\in\sqrt{-1}\fh$. Then $ad(h)$ has eigenvalues $2,1,0,-1,-2$. We set 
\begin{align*}
\fp^\pm_\rT&:=\{x\in\fp^\pm:[h,x]=\pm 2x\}\subset \fp^\pm,\\
\fk^\BC_\rT&:=[\fp^+_\rT,\fp^-_\rT]\subset \fk^\BC,\\
\fg^\BC_\rT&:=\fp^+_\rT\oplus\fk^\BC_\rT\oplus\fp^-_\rT,\\
\fg_\rT&:=\fg^\BC_\rT\cap\fg. 
\end{align*}
Then, $\fp^+_\rT$ becomes a complex simple Jordan algebra with the product 
\begin{equation}\label{Jordanstr}
x\cdot y:=\{ x,e,y\}=-\frac{1}{2}[[x,\vartheta e],y], 
\end{equation}
and $\fg_\rT$ becomes a Lie algebra of tube type.  

We define the Cayley transform $c:\fg^\BC\to\fg^\BC$ by $c:=Ad(e^{\frac{\pi i}{4}(e-\vartheta e)})$, 
and set ${}^c\fg:=c(\fg)$, ${}^c\fg_\rT:=c(\fg_\rT)$. Then ${}^c\fg_\rT\subset\fg^\BC_\rT$ is fixed by the involution 
$\sigma\vartheta:=Ad(e^{\frac{\pi}{2}(e+\vartheta e)})\circ\vartheta$. By direct computation we have 
\begin{gather*}
\sigma\vartheta|_{\fp^+_\rT}=\frac{1}{2}ad(e)^2\circ \vartheta:\fp^+_\rT\longrightarrow\fp^+_\rT,\\
\sigma\vartheta|_{\fk^\BC_\rT}=(\mathrm{id}_{\fk^\BC}+ad(e)ad(\vartheta e))\circ\vartheta:\fk^\BC_\rT\longrightarrow\fk^\BC_\rT,\\
\sigma\vartheta|_{\fp^-_\rT}=\frac{1}{2}ad(\vartheta e)^2\circ \vartheta:\fp^-_\rT\longrightarrow\fp^-_\rT. 
\end{gather*}
That is, $\sigma\vartheta$ preserves the grading. Therefore we denote 
\[ {}^c\fg_\rT=\fn^+\oplus\fl\oplus\fn^-\subset\fp^+_\rT\oplus\fk^\BC_\rT\oplus\fp^-_\rT=\fg^\BC_\rT. \]
Then the real form $\fn^+$ of $\fp^+_\rT$ becomes a Euclidean simple Jordan algebra. 

We set $\fa_{\fl}:=c(\fa)=\sqrt{-1}\fh\cap\fl=\bigoplus_{j=1}^r\BR h_j$, where $h_j:=-[e_j,\vartheta e_j]$. 
Then the restricted root system $\Sigma=\Sigma({}^c\fg,\fa_\fl)$ is given by 
\[ \Sigma=\begin{cases}\ds \left\{\left.\frac{1}{2}(\gamma_j-\gamma_k)\right|_{\fa_\fl}:
\!\!\begin{array}{c}1\le j, k\le r,\\ j\ne k\end{array}\!\!\right\}
\cup\left\{\left.\pm\frac{1}{2}(\gamma_j+\gamma_k)\right|_{\fa_\fl}:1\le j\le k\le r\right\}&(\fg=\fg_\rT),\\
\ds (\text{as above})\cup\left\{\left.\pm\frac{1}{2}\gamma_j\right|_{\fa_\fl}:1\le j\le r\right\}&(\fg\ne\fg_\rT). 
\end{cases} \]
We define the positive restricted roots $\Sigma_+$ by 
\[ \Sigma_+=\begin{cases}\ds \left\{\left.\frac{1}{2}(\gamma_j-\gamma_k)\right|_{\fa_\fl}:1\le j<k\le r\right\}
\cup\left\{\left.\frac{1}{2}(\gamma_j+\gamma_k)\right|_{\fa_\fl}:1\le j\le k\le r\right\}&(\fg=\fg_\rT),\\
\ds (\text{as above})\cup\left\{\left.\frac{1}{2}\gamma_j\right|_{\fa_\fl}:1\le j\le r\right\}&(\fg\ne\fg_\rT). 
\end{cases} \]
Then $\Sigma_+$ and $\Delta_+$ are compatible, that is, $\alpha\in\Delta_+$ implies 
$\alpha|_{\fa_\fl}\in\Sigma_+\cup\{0\}$. We set 
\begin{align*}
\fl_{jk}&:=\left\{X\in {}^c\fg_\rT:ad(H)X=\frac{1}{2}(\gamma_j-\gamma_k)(H)X \text{ for any }H\in\fa_\fl\right\}
&&(1\le j,k\le r,\ j\ne k),\\
\fm_\fl&:=\left\{X\in {}^c\fg^\vartheta_\rT:ad(H)X=0 \text{ for any }H\in\fa_\fl\right\},\\
\fn^\pm_{jk}&:=\left\{X\in {}^c\fg_\rT:ad(H)X=\pm\frac{1}{2}(\gamma_j+\gamma_k)(H)X \text{ for any }H\in\fa_\fl\right\}
&&(1\le j\le k\le r),\\
\fp^\pm_{jk}&:=(\fn_{jk}^\pm)^\BC &&(1\le j\le k\le r),\\
\fp^\pm_{0j}&:=\left\{X\in\fp^\pm:ad(H)X=\pm\frac{1}{2}\gamma_j(H)X \text{ for any }H\in\fa_\fl\right\}
&&(1\le j\le r),
\end{align*}
and 
\begin{gather*}
\fk_\fl:=\fl^{\vartheta}=\{X\in\fl:\vartheta X=Ad(e^{\frac{\pi}{2}(e+\vartheta e)})X=X\},\\
\fn_\fl^-:=\bigoplus_{1\le k<j\le r}\fl_{jk}.
\end{gather*}
Then we have 
\begin{gather*}
\fl=\fa_\fl\oplus\fm\oplus\bigoplus_{j\ne k}\fl_{jk}=\fk_\fl\oplus\fa_\fl\oplus\fn_\fl^-,\\
\fn^\pm=\bigoplus_{1\le j\le k\le r}\fn^\pm_{jk},\qquad \fp^\pm_\rT=\bigoplus_{1\le j\le k\le r}\fp^\pm_{jk},
\qquad \fp^\pm=\bigoplus_{\substack{0\le j\le k\le r\\ (j,k)\ne (0,0)}}\fp^\pm_{jk}.
\end{gather*}
The decomposition $\fn^+=\bigoplus_{j\le k}\fn^+_{jk}$, or $\fp^+=\bigoplus_{j\le k}\fp^+_{jk}$, 
coincides with the Peirce decomposition of the Jordan algebra $\fn^+$, or the Jordan triple system $\fp^+$, 
with respect to the Jordan frame $\{e_1,\ldots,e_r\}$. 
We set $d:=\dim_\BC\fp^+_{12}$, $b:=\dim_\BC\fp^+_{01}$, and $n_\rT:=\dim_\BC\fp^+_\rT$. 
Then $n=r+\frac{1}{2}r(r-1)d+br$ and $n_\rT=r+\frac{1}{2}r(r-1)d$ holds. 
Also we set $p:=2+(r-1)d+b$. 

Throughout this paper, let $G^\BC$ be a connected complex Lie group with Lie algebra $\fg^\BC$, 
and let $G, {}^cG_\rT, K, K^\BC, K^\BC_\rT$ be the connected Lie subgroups with Lie algebras 
$\fg, {}^c\fg_\rT, \fk, \fk^\BC, \fk^\BC_\rT$ respectively. 
Also we set $L:=K^\BC\cap {}^cG_\rT$, $K_L:=K\cap L$ (possibly non-connected, with Lie algebras $\fl,\fk_\fl$), 
let $A_L, N_L^-$ be the connected Lie subgroups of $L$ with Lie algebras $\fa_\fl, \fn_\fl^-$ respectively, 
and let $M_L$ the centralizer of $\fa_\fl$ in $K_L$. 

We write 
\begin{align*}
\bar{x}&:=\sigma\vartheta x=\frac{1}{2}ad(e)^2(\vartheta x)& &(x\in\fp^+_\rT), \\
l^*&:=-\vartheta l& &(l\in\fk^\BC), \\
{}^t\hspace{-1pt}l&:=-\sigma l=-(\mathrm{id}_{\fk^\BC}+ad(e)ad(\vartheta e))(l) & &(l\in\fk^\BC_\rT), \\
\bar{l}&:=\sigma\vartheta l=(\mathrm{id}_{\fk^\BC}+ad(e)ad(\vartheta e)) & &(l\in\fk^\BC_\rT). 
\end{align*}
Then these are (anti-)involutions on $\fp^+_\rT$, $\fk^\BC$ and $\fk^\BC_\rT$, 
which preserves $\fn^+$, $\fk$, $(\fk_\fl)^\BC$ and $\fl$ respectively. 
Also, we denote by the same symbols ${}^*$, ${}^t$ and $\bar{\phantom{l}}$ the corresponding (anti-)involutions on $K^\BC$ and $K^\BC_\rT$. 
Also, for $x\in\fp^+$ and $l\in K^\BC$ or $\fk^\BC$, we abbreviate $Ad(l)x$ or $ad(l)x$ as $lx$.

\subsection{Some operations and polynomials on Jordan algebras}
As in the previous subsection, $\fp^+$ has a Jordan triple system structure, and $\fp^+_\rT,\fn^+$ has a Jordan algebra structure. 
For $x,y\in\fp^+$, we define $x\Box y$, $B(x,y)\in\End_\BC(\fp^+)$ by, for $z\in\fp^+$, 
\begin{align*}
(x\Box y)z&:=\{x,y,z\}=-\frac{1}{2}ad([x,\vartheta y])z,\\
B(x,y)z&:=x-2\{x,y,z\}+\{x,\{y,z,y\},x\}
=\left(I_{\fp^+}+ad([x,\vartheta y])+\frac{1}{4}ad(x)^2ad(\vartheta y)^2\right)z. 
\end{align*}
These depends holomorphically on $x$, and anti-holomorphically on $y$. 
Also, for $x\in\fp^+_\rT$, we define $L(x),P(x)\in\End_\BC(\fp^+_\rT)$ by, for $y\in\fp^+_\rT$, 
\begin{align*}
L(x)y&:=xy=-\frac{1}{2}ad([x,\vartheta e])y,\\
P(x)y&:=2x(xy)-(x^2)y=\frac{1}{4}ad(x)^2ad(\vartheta e)^2y. 
\end{align*}
Then for $x,y\in\fp^+$ and $l\in K^\BC$, 
\begin{align*}
lx\Box (l^*)^{-1}y&=l(x\Box y)l^{-1},\\
B(lx,(l^*)^{-1}y)&=lB(x,y)l^{-1}
\end{align*}
holds, and for $x\in\fp^+_\rT$, $l\in K^\BC_\rT$, 
\begin{align*}
P(lx)&=lP(x){}^t\hspace{-1pt}l,\\
\left.B(x,\overline{x})\right|_{\fp^+_\rT}&=P(e-x^2) 
\end{align*}
holds. We define an inner product $(\cdot|\cdot)$ on $\fp^+$ by 
\[ (x|y):=\frac{2}{p}\Tr(x\Box y:\fp^+\to\fp^+). \]
Then for $l\in K^\BC$, $(lx|y)=(x|l^*y)$ holds. 
This inner product is proportional to the restriction of the Killing form on $\fg^\BC$ to $\fp^+\times\fp^-$, 
under the identification of $\fp^+$ and $\fp^-$ through $\vartheta$. 
Also, let $\tr(x)$, $\det(x)$ be the trace and determinant polynomials of the Jordan algebra $\fp^+_\rT$, 
and let $h(x,y)$ be the generic norm of the Jordan triple system $\fp^+$. 
Then these polynomials are expressed by 
\begin{align*}
\frac{n_\rT}{r}\tr(x)&=\Tr(L(x):\fp^+_\rT\to\fp^+_\rT),\\
(\det(x))^{2n_\rT/r}&=\Det(P(x):\fp^+_\rT\to\fp^+_\rT),\\
(h(x,y))^{p}&=\Det(B(x,y):\fp^+\to\fp^+). 
\end{align*}
$\tr(x)$ is a linear form satisfying $\tr(x)=(x|e)$, and $\det(x)$, $h(x,y)$ are polynomials of degree $r$ 
with respect to each variable. These polynomials satisfy 
\begin{align*}
\det(lx)&=\det(le)\det(x)&&(l\in K^\BC_\rT,\; x\in\fp^+_\rT),\\
h(lx,(l^*)^{-1}y)&=h(x,y)&&(l\in K^\BC,\; x,y\in\fp^+),\\
h(x,\overline{x})&=\det(e-x^2)&&(x\in \fp^+_\rT). 
\end{align*}
From now we abbreviate $B(x,x)=B(x)$, $h(x,x)=h(x)$, and $(x|x)=|x|^2$ for $x\in\fp^+$. 
Then $B(x)$ is self-adjoint on $\fp^+$, and therefore $h(x)$ is real-valued. 
Also we set 
\begin{gather*}
\Omega:=\{x^2\in\fn^+:x\in\fn^+,\ \det(x)\ne 0\},\\
D:=(\text{connected componet of }\{w\in\fp^+:h(w)>0\}\text{ which contains }0). 
\end{gather*}
Then $L$ acts on $\Omega$ by linear transformation, and $G$ acts on $D\subset\fp^+$ via Borel embedding, 
which we will review later. Moreover we have 
\[ \Omega\simeq L/K_L,\qquad D\simeq G/K. \]
For $x\in\Omega$, $P(x)$ is positive definite on $\fn^+$, and there exists a unique element 
$l\in\exp(\fl^{-\vartheta})\subset L$ such that $P(x)=Ad(l)|_{\fn^+}$. We denote such $l\in L$ by the same $P(x)$. 
Similarly, for $z,w\in D$, $B(z,w)$ is invertible on $\fp^+$, and there exists an element $l\in K^\BC$ 
such that $B(z,w)=Ad(l)|_{\fp^+}$. So we define the holomorphic map $B:D\times\overline{D}\to K^\BC$ 
(with the same symbol $B$) such that $Ad(B(z,w))|_{\fp^+}=B(z,w)$ and $B(0,0)=\mbox{\boldmath{$1$}}$. 
Clearly $P(x)$ and $B(z,w)$ are also well-defined as elements of the universal covering groups 
$\tilde{L}$, $\tilde{K}^\BC$. 

Now we recall the Peirce decomposition 
\[ \fp^+=\bigoplus_{\substack{0\le j\le k\le r\\ (j,k)\ne (0,0)}}\fp^+_{jk}. \]
We set 
\[ \fp^+_{(l)}:=\bigoplus_{1\le j\le k\le l}\fp^+_{jk} \]
for $l=1,2,\ldots,r$. Then each $\fp^+_{(l)}$ is again a unital Jordan algebra. 
For each $l$, let ${\det}_{(l)}$ be the determinant polynomial of $\fp^+_{(l)}$, $P_l:\fp^+\to\fp^+_{(l)}$ be the 
orthogonal projection, and we set 
\[ \Delta_l(x):={\det}_{(l)}(P_l(x)). \]
For $l=r$ we also write 
\[ \Delta(x)=\Delta_r(x)=\det(x). \]
Using these, for $\bs=(s_1,\ldots,s_r)\in\BC^r$, we set 
\[ \Delta_\bs(x):=\Delta_1(x)^{s_1-s_2}\Delta_2(x)^{s_2-s_3}\cdots\Delta_{r-1}(x)^{s_{r-1}-s_r}\Delta_r(x)^{s_r}. \]
If $\bm\in\BZ^r$ and $m_1\ge m_2\ge\cdots\ge m_r\ge 0$, then $\Delta_\bm$ is a polynomial of degree $m_1+\cdots+m_r$. 
We denote this condition by $\BZ^r_{++}$:
\[ \BZ^r_{++}:=\{\bm=(m_1,\ldots,m_r)\in\BZ^r:m_1\ge \cdots\ge m_r\ge 0\}. \]
For later use, we prepare another set $\BZ^r_+$:
\[ \BZ^r_+:=\{\bm=(m_1,\ldots,m_r)\in\BZ^r:m_1\ge \cdots\ge m_r\}. \]

Now for $q\in (M_LA_LN_L^-)^\BC$, since $q$ preserves each $\fp^+_{(l)}$, we have 
\[ \Delta_\bs(qx)=\Delta_\bs(qe)\Delta_\bs(x). \]
That is, for any $\bm$, $\Delta_\bm$ is a lowest weight vector with lowest weight $-m_1\gamma_1-\cdots-m_r\gamma_r$ 
under the representation 
\[ L\longrightarrow\End(\cP(\fp^+)),\qquad l\longmapsto(f(x)\longmapsto f(l^{-1}x)) \]
where $\cP(\fp^+)$ denotes the space of all holomorphic polynials on $\fp^+$. In fact, we have 
\begin{theorem}[Hua-Kostant-Schmid, {\cite[Part III, Theorem V.2.1]{FKKLR}}]\label{HKS}
\[ \cP(\fp^+)=\bigoplus_{\bm\in\BZ^r_{++}}\cP_\bm(\fp^+) \]
where $\cP_\bm(\fp^+)$ is the irreducible representation of $K^\BC$ with lowest weight 
$-m_1\gamma_1-\cdots-m_r\gamma_r$. 
\end{theorem}
We quote another theorem here. 
\begin{theorem}[{\cite[Theorem XII.2.2]{FK}}]\label{sph}
The irreducible representation $V$ of $L$ has a $K_L$-fixed vector if and only if the lowest weight $-\lambda$ 
is of the form $-\lambda=-m_1\gamma_1-\cdots-m_r\gamma_r$ with $(m_1,\ldots,m_r)\in\BZ_+^r$. 
\end{theorem}

For $l=0,1,\ldots,r$ we set 
\begin{equation}\label{orbit}
\cO_l:=Ad(K^\BC)(e_1+\cdots+e_l)\subset \fp^+. 
\end{equation}
Then $K^\BC$ acts on each $\cO_l$ transitively, and we have the orbit decomposition 
\[ \fp^+=\cO_0\cup\cO_1\cup\cdots\cup\cO_r. \]
For each orbit $\cO_l$, its closure $\overline{\cO_l}$ is given by 
\[ \overline{\cO_l}=\cO_0\cup\cO_1\cup\cdots\cup\cO_l. \]
Also, since the polynomial $\Delta_{l+1}(x)$ vanishes on $\overline{\cO_l}$, 
the polynomial space on $\overline{\cO_l}$ decomposes under $K^\BC$ as 
\begin{equation}\label{orbitpoly}
\cP(\overline{\cO_l})=\bigoplus_{\substack{\bm\in\BZ^r_{++}\\ m_{l+1}=m_{l+2}=\cdots=0}}\cP_\bm(\fp^+). 
\end{equation}
Each orbit $\cO_l$ has the dimension 
\begin{equation}\label{orbitdim}
\dim_\BC\cO_l=l+\frac{1}{2}l(2r-l-1)d+lb 
\end{equation}
since the tangent space of $\cO_l$ at $e_1+\cdots+e_l$ is given by 
\[ T_{e_1+\cdots+e_l}\cO_l=\bigoplus_{\substack{0\le j\le k\le r\\ j\le l,\,(j,k)\ne (0,0)}}\fp^+_{jk}. \]

Now we recall the generalized Gamma function, which was introduced by Gindikin \cite{G}. 
For $\bs\in\BC^n$ this is defined as 
\[ \Gamma_\Omega(\bs):=\int_{\Omega}e^{-\tr(x)}\Delta_\bs(x)\Delta(x)^{-\frac{n_\rT}{r}}dx. \]
This integral converges if $\Re s_j>(j-1)\frac{d}{2}$, and we have the following equality 
\[ \Gamma_\Omega(\bs)=(2\pi)^{\frac{n_\rT-r}{2}}\prod_{j=1}^r\Gamma\left(s_j-(j-1)\frac{d}{2}\right) \]
(\cite[Corollary VII.1.3]{FK}), and this is meromorphically extended on $\BC^n$. Also we denote 
\[ (\bs)_{\bm}:=\frac{\Gamma_\Omega(\bs+\bm)}{\Gamma_\Omega(\bs)}
=\prod_{j=1}^r\left(s_j-(j-1)\frac{d}{2}\right)_{m_j}. \]
For $\bs=(\lambda,\ldots,\lambda)$, we abbreviate $(\lambda,\ldots,\lambda)=:\lambda$. For example, we denote 
\[ \Gamma_\Omega((\lambda,\ldots,\lambda))=\Gamma_\Omega(\lambda),\qquad 
((\lambda,\ldots,\lambda))_\bm=\frac{\Gamma_\Omega(\lambda+\bm)}{\Gamma_\Omega(\lambda)}=(\lambda)_\bm. \]

\section{Norm computation: General theory}
\subsection{Holomorphic discrete series representation}\label{HDS}
In this subsection we recall the explicit realization of the holomorphic series representation 
of the universal covering group $\tilde{G}$. First we recall the Borel embedding. 
\[ \xymatrix{ G/K \ar[r] \ar@{-->}[d]^{\mbox{\rotatebox{90}{$\sim$}}} & G^\BC/K^\BC P^- \\
 D \ar@{^{(}->}[r] & \fp^+ \ar[u]_{\exp} } \]
We consider maps $\pi^+:G\times D\to D\subset\fp^+$, $\kappa:G\times D\to K^\BC$, 
$\pi^-:G\times D\to\fp^-$ such that 
\[ g\exp(w)=\exp(\pi^+(g,w))\kappa(g,w)\exp(\pi^-(g,w))\qquad(g\in G,w\in D). \]
Then $\pi^+$ gives the action of $G$ on $D$, so we abbreviate $\pi^+(g,w)=:gw$. 
On $K\subset G$ this coincides with the adjoint action. 
Also, $\kappa$ satisfies the cocycle condition 
\[ \kappa(gh,w)=\kappa(g,hw)\kappa(h,w)\qquad(g,h\in G,\; w\in D), \]
and for $k\in K$, $\kappa(k,w)=k$ holds. 
$Ad(\kappa(g,w))|_{\fp^+}\in\mathrm{End}(\fp^+)$ coincides with the tangent map of $w\mapsto gw=\pi^+(g,w)$ at $w\in\fp^+$. 
We naturally lift $\kappa$ to the universal covering group, 
and we denote this map by the same symbol $\kappa:\tilde{G}\times D\to \tilde{K}^\BC$. 

Let $(\tau,V)$ be a finite dimensional irreducible complex representation of $K^\BC$, 
and we fix a $K$-invariant inner product $(\cdot,\cdot)_\tau$ on $V$. 
Also, let $\chi^{\lambda}$ be the character of $\tilde{K}^\BC$ such that $\chi(k)^\lambda=\Det(Ad(k)|_{\fp^+})^{\lambda/p}$. 
We consider the space of holomorphic sections 
\[ \Gamma_\cO(G/K, \tilde{G}\times_{\tilde{K}}(V\otimes\chi^{-\lambda})). \]
Then since $G/K\simeq D$ is contractible, this is isomorphic to $\cO(D,V)$, 
the space of $V$-valued holomorphic functions. Under this identification, 
the natural action $\tau_\lambda$ of $\tilde{G}$ on $\cO(D,V)$ is written as 
\[ \tau_\lambda(g)f(w)=\chi(\kappa(g^{-1},w))^{\lambda}\tau(\kappa(g^{-1},w))^{-1}f(g^{-1}w)
\qquad(g\in\tilde{G}, w\in D, f\in\cO(D,V)). \]
Its differential representation is given by, for $u+l-\vartheta v\in\fp^+\oplus\fk^\BC\oplus\fp^-=\fg^\BC$, 
\begin{align*}
d\tau_\lambda(u+l-\vartheta v)f(w)=&-\lambda d\chi(l+[w,\vartheta v])f(w)+d\tau(l+[w,\vartheta v])f(w) \\
&+\left.\frac{d}{dt}\right|_{t=0}f\left(w-t\left(u+ad(l)w-\frac{1}{2}ad(w)^2\vartheta v\right)\right). 
\end{align*}
Then since $\kappa(g,w)B(w)\kappa(g,w)^*=B(gw)$ holds for any $g\in\tilde{G}$, $w\in D$ (see \cite[Lemma 2.11]{L}), 
this action preserves the following \textit{weighted Bergman inner product} 
\begin{equation}\label{Bergmaninner}
\langle f,g\rangle_{\lambda,\tau}
:=\frac{c_\lambda}{\pi^n}\int_D\left(\tau(B(w)^{-1})f(w),g(w)\vphantom{w^2}\right)_\tau h(w)^{\lambda-p}dw 
\qquad (f,g\in\cO(D,V)), 
\end{equation}
where $c_\lambda$ is a constant defined such that $\Vert v\Vert_{\lambda,\tau}=|v|_\tau$ holds for any constant functions 
$z\mapsto v\in V$ (i.e. for any element of the minimal $K$-type). 
Let $\cH_\lambda(D,V)\subset\cO(D,V)$ be the unitary subrepresentation of $\tilde{G}$ under $\tau_\lambda$. 
Then $\cH_\lambda(D,V)$ is non-zero if $\lambda\in\BR$ is sufficiently large so that the above inner product converges. 
On the other hand, we cannot know a priori whether $\cH_\lambda(D,V)$ is zero or non-zero if $\lambda$ is small. 
In any case, if $\cH_\lambda(D,V)$ is non-zero, the reproducing kernel is proportional to $K_{\Re \lambda,\tau}(z,w)$, where 
\[ K_{\lambda,\tau}(z,w):=h(z,w)^{-\lambda}\tau(B(z,w))\in\cO(D\times\overline{D},\operatorname{End}(V)). \]
This is because the reproducing kernel $K(z,w)$ is characterized by 
\[ \chi(\kappa(g,z))^{\lambda}\tau(\kappa(g,z))^{-1}K(gz,gw)\tau(\kappa(g,w))^{*-1}\overline{\chi(\kappa(g,w))^{\lambda}}=K(z,w), \]
and such $K(z,w)$ is unique up to constant multiple, 
since $\tilde{G}$ acts transitively on the totally real submanifold $\mathrm{diag}(D)\subset D\times\overline{D}$, 
which allows the value at origin $K(0,0)$ to determine the whole $K(z,w)$, 
and $K(0,0)\in\mathrm{End}(V)$ is proportional to identity since this commutes with $\tilde{K}$-action. 
When $\lambda\in\BR$ is sufficiently large, then the reproducing kernel corresponding to the inner product (\ref{Bergmaninner}) 
is precisely $K_{\lambda,\tau}(z,w)$ by the normalization assumption.

\subsection{Key theorem}\label{key}
The norm $\Vert\cdot\Vert_{\lambda,\tau}$ in the previous subsection is $\tilde{G}$-invariant, 
and therefore $\tilde{K}$-invariant. From now on we observe how the norm varies as the parameter $\lambda$ varies 
on each $K$-type. In order to compare, we consider another $K$-invariant norm which is independent of $\lambda$. 

We recall the \textit{Fischer inner product} $\langle\cdot,\cdot\rangle_{F,\tau}$ on $\cP(\fp^+,V)$, 
the space of $V$-valued holomorphic polynomials on $\fp^+$. 
\begin{equation}\label{Fischer}
\langle f,g\rangle_{F,\tau}:=\frac{1}{\pi^n}\int_{\fp^+}(f(w),g(w))_\tau e^{-|w|^2}dw 
\qquad(f,g\in\cP(\fp^+,V)). 
\end{equation}
This inner product is invariant under the following representation $(\hat{\tau},\cP(\fp^+,V))$: 
\[ \left(\hat{\tau}(k)f\right)(w):=\tau(k)f(k^{-1}w)\qquad(k\in K^\BC,\; f\in\cP(\fp^+,V),\; w\in\fp^+), \]
that is, $\langle \hat{\tau}(k)f,g\rangle_{F,\tau}=\langle f,\hat{\tau}(k^*)g\rangle_{F,\tau}$ holds. 
Let $W\subset\cP(\fp^+,V)=\cO(D,V)_K$ be a $K^\BC$-irreducible subspace. Then since both $\Vert\cdot\Vert_{F,\tau}$ and 
$\Vert\cdot\Vert_{\lambda,\tau}$ are $K$-invariant, the ratio of these two norms are constant on $W$. 
Therefore we aim to compute this ratio of two norms. 

In order to state the key theorem, we prepare some notations. Let 
\[ (\tau,V)|_{K^\BC_\rT}=\bigoplus_i(\tau_i,V_i)\]
be the decomposition of the $K^\BC$-module $(\tau,V)$ 
into $K^\BC_\rT$-irreducible submodules, 
and for each $i$ we denote by $(\bar{\tau}_i,\overline{V_i})$ the complex conjugate representation of $V_i$ 
with respect to the real form $L\subset K^\BC_\rT$, that is, there exists a conjugate linear isomorphism $\bar{\cdot}:V_i\to\overline{V_i}$, 
and $\bar{\tau}_i$ is given by $\bar{\tau}_i(l)\bar{v}=\overline{\tau_i(\bar{l})v}$. Let 
\[ \mathrm{rest}:\cP(\fp^+,V)\to\cP(\fp^+_\rT,V)=\bigoplus_i\cP(\fp^+_\rT,V_i) \]
be the restriction map, and for each $i$ we take $K^\BC_\rT$-submodules $W_{ij}\subset\cP(\fp^+_\rT,V_i)$ such that 
\[ \mathrm{rest}(W)\subset\bigoplus_i\bigoplus_j W_{ij} \]
holds. 
\begin{theorem}\label{keythm}
Let $(\tau,V)|_{K^\BC_\rT}=\bigoplus_i(\tau_i,V_i)$, and suppose each $(\tau_i,V_i)$ has a restricted lowest weight 
$\left.-\left(\frac{k_{i,1}}{2}\gamma_1+\cdots+\frac{k_{i,r}}{2}\gamma_r\right)\right|_{\fa_\fl}$. 
Let $W\subset\cP(\fp^+,V)$ be a $K^\BC$-irreducible subspace, with 
$\mathrm{rest}(W)\subset\bigoplus_i\bigoplus_j W_{ij}\subset\bigoplus_i\cP(\fp^+_\rT,V_i)$ as above. We assume 
\begin{itemize}
\item[(A1)] $(\tau_i,V_i)|_{K_L}$ still remains irreducible for each $i$. 
\item[(A2)] For each $i,j$, all the $K_L$-spherical irreducible subspaces 
in $W_{ij}\otimes\overline{V_i}$ have the same lowest weight $-\left(n_{ij,1}\gamma_1+\cdots+n_{ij,r}\gamma_r\right)$. 
\end{itemize}
Then the integral $\Vert f\Vert_{\lambda,\tau}^2$ converges for any $f\in W$ if $\Re(\lambda)+k_{i,r}>p-1$ for all $i$. 
Moreover, there exist non-negative integers $a_{ij}$ such that, for any $f\in W$, 
\[ \frac{\Vert f\Vert_{\lambda,\tau}^2}{\Vert f\Vert_{F,\tau}^2}=\frac{c_\lambda}{\sum_{ij}a_{ij}}
\sum_{ij}a_{ij}\frac{\Gamma_\Omega\left(\lambda+\bk_i-\frac{n}{r}\right)}{\Gamma_\Omega(\lambda+\bn_{ij})}, \]
where 
\[ c_\lambda^{-1}=\frac{1}{\dim V}\sum_i(\dim V_i)
\frac{\Gamma_\Omega\left(\lambda+\bk_i-\frac{n}{r}\right)}{\Gamma_\Omega(\lambda+\bk_i)}. \]
\end{theorem}

In the rest of this section we prove this theorem. 
We set $\Vert f\Vert_{\lambda,\tau}^2/\Vert f\Vert_{F,\tau}^2=:R_W(\lambda)$ for $f\in W$, 
and compute this ratio $R_W(\lambda)$. 

Let $K_W(z,w)\in\cP(\fp^+\times\overline{\fp^+},\mathrm{End}(V))$ be the reproducing kernel of $W$ with respect to 
$\langle\cdot,\cdot\rangle_{F,\tau}$, that is, for an orthonormal basis $\{f_i\}$ of $W$ with respect to $\langle\cdot,\cdot\rangle_{F,\tau}$, 
\[ K_W(z,w)v:=\sum_i(v,f_i(w))_\tau f_i(z) \qquad (v\in V), \]
which does not depend on the choice of $\{f_i\}$. Then the ratio $R_W(\lambda)$ is computed as 
\begin{align*}
R_W(\lambda)&=\frac{\ds c_\lambda\sum_i\int_D\left(\tau(B(w)^{-1})f_i,f_i\vphantom{w^2}\right)_{\tau}h(w)^{\lambda-p}dw}
{\ds \sum_i\int_{\fp^+}\left(f_i,f_i\right)_{\tau}e^{-|w|^2}dw}\\
&=\frac{\ds c_\lambda\int_D\Tr_V\left(\tau(B(w)^{-1})K_W(w,w)\vphantom{w^2}\right)h(w)^{\lambda-p}dw}
{\ds \int_{\fp^+}\Tr_V(K_W(w,w))e^{-|w|^2}dw}, 
\end{align*}
and if the numerator converges, then $\Vert f_i\Vert_{\lambda,\tau}^2$ converges for any $i$, and so does $\Vert f\Vert_{\lambda,\tau}^2$ 
for any $f\in W$. 
To proceed the computation, we use the following lemma. 
\begin{lemma}\label{intformula}
For any integrable, or non-negative-valued measurable function $f$ on $\fp^+$, we have 
\[ \frac{1}{\pi^n}\int_{\fp^+}f(w)dw
=\frac{1}{\Gamma_\Omega\left(\frac{n}{r}\right)}\int_\Omega\int_Kf(kx^{\frac{1}{2}})\Delta(x)^bdkdx, \]
where $x^{\frac{1}{2}}$ is the square root with respect to the Jordan algebra structure (\ref{Jordanstr}) on $\Omega\subset\fn^+$. 
\end{lemma}
\begin{proof}
For tube type case ($b=0$) see \cite[Proposition X.3.4]{FK}. Even for $b\ne 0$ case we can prove this similarly. 
\end{proof}
Since the integrand of $R_W(\lambda)$ is non-negative-valued, by this lemma, this is equal to 
\[ R_W(\lambda)=\frac{\ds c_\lambda\int_{\Omega\cap(e-\Omega)}\int_K\Tr_V\left(\tau(B(kx^\frac{1}{2})^{-1})
K_W(kx^\frac{1}{2},kx^\frac{1}{2})\right)h(kx^\frac{1}{2})^{\lambda-p}\Delta(x)^bdkdx}
{\ds \int_\Omega\int_K\Tr_V\left(K_W(kx^\frac{1}{2},kx^\frac{1}{2})\right)e^{-|kx^\frac{1}{2}|^2}\Delta(x)^bdkdx}. \]
Since the reproducing kernel satisfies 
\[ K_W(kz,k^{*-1}w)=\tau(k)K_W(z,w)\tau(k^{-1}) \qquad(z,w\in\fp^+,\; k\in K^\BC), \]
we have, 
\begin{align*}
K_W(kx^\frac{1}{2},kx^\frac{1}{2})&=\tau(k)K_W(P(x^{-\frac{1}{4}})x,P(x^{\frac{1}{4}})e)\tau(k^{-1}) \\
&=\tau(k)\tau(P(x^{-\frac{1}{4}}))K_W(x,e)\tau(P(x^{\frac{1}{4}}))\tau(k^{-1}) \qquad(x\in\Omega,k\in K). 
\end{align*}
Therefore we have 
\[ \Tr_V\left(K_W(kx^\frac{1}{2},kx^\frac{1}{2})\right)=\Tr_V(K_W(x,e)). \]
Also, since $k^{-1}B(kx^\frac{1}{2})^{-1}k=B(x^\frac{1}{2})^{-1}=P(e-x)^{-1}$ and 
$P(e-x)^{-1}$ commutes with $P(x^\frac{1}{4})$, we have 
\[ \Tr_V\left(\tau(B(kx^\frac{1}{2})^{-1})K_W(kx^\frac{1}{2},kx^\frac{1}{2})\right)
=\Tr_V\left(\tau(P(e-x)^{-1})K_W(x,e)\right). \]
By these and $h(kx^\frac{1}{2})=\Delta(e-x)$, $|kx^\frac{1}{2}|^2=\tr(x)$, we have 
\[ R_W(\lambda)=\frac{\ds c_\lambda\int_{\Omega\cap(e-\Omega)}\Tr_V\left(\tau(P(e-x)^{-1})
K_W(x,e)\vphantom{x^2}\right)\Delta(e-x)^{\lambda-p}\Delta(x)^bdx}
{\ds \int_\Omega\Tr_V(K_W(x,e))e^{-\tr (x)}\Delta(x)^bdx}. \]
By the assumption, we can rewrite $K_W(z,w)$ by using $K_{W_{ij}}(z,w)$, the reproducing kernels of $W_{ij}$, when $z,w\in\fp^+_\rT$: 
\[ K_W(z,w)=\sum_{ij}\tilde{a}_{ij}K_{W_{ij}}(z,w)\in\cP(\fp^+_\rT\times\overline{\fp^+_\rT},\mathrm{End}(V)) \qquad (z,w\in\fp^+_\rT), \]
using some non-negative numbers $\tilde{a}_{ij}$. Therefore we have 
\[ R_W(\lambda)=\frac{\ds c_\lambda\sum_{ij}\tilde{a}_{ij}\int_{\Omega\cap(e-\Omega)}\Tr_{V_i}\left(\tau_i(P(e-x)^{-1})
K_{W_{ij}}(x,e)\vphantom{x^2}\right)\Delta(e-x)^{\lambda-p}\Delta(x)^bdx}
{\ds \sum_{ij}\tilde{a}_{ij}\int_\Omega\Tr_{V_i}(K_{W_{ij}}(x,e))e^{-\tr (x)}\Delta(x)^bdx}. \]
Now we set 
\begin{gather*}
B_{ij}(\lambda):=\int_{\Omega\cap(e-\Omega)}\Tr_{V_i}\left(\tau_i(P(e-x)^{-1})
K_{W_{ij}}(x,e)\vphantom{x^2}\right)\Delta(e-x)^{\lambda-p}\Delta(x)^bdx, \\
\Gamma_{ij}:=\int_\Omega\Tr_{V_i}(K_{W_{ij}}(x,e))e^{-\tr (x)}\Delta(x)^bdx 
\end{gather*}
so that $R_W(\lambda)=c_\lambda\left(\sum_{ij}\tilde{a}_{ij}B_{ij}(\lambda)\right)\Big/\left(\sum_{ij}\tilde{a}_{ij}\Gamma_{ij}\right)$. 
Now, we regard $K_{W_{ij}}(x,e)\in\cP(\fp^+_\rT,\mathrm{End}(V_i))$ as a function of $x$. 
We define the action $\tilde{\tau}_i$ of $K^\BC_T$ on $\cP(\fp^+_\rT,\mathrm{End}(V_i))$ by 
\[ (\tilde{\tau}_i(k)F)(x):=\tau_i(k)F(k^{-1}x)\tau_i({}^t\hspace{-1pt}k)\qquad
(k\in K^\BC_\rT, F\in\cP(\fp^+_\rT,\mathrm{End}(V_i)), x\in\fp^+_\rT). \]
Then $K_{W_{ij}}(x,e)$ is $K_L$-invariant under $\tilde{\tau}_i$. 
Now we identify 
\[ (\tilde{\tau}_i,\cP(\fp^+_\rT,\mathrm{End}(V_i)))\simeq
(\hat{\tau}|_{K^\BC_\rT}\otimes\bar{\tau}_i,\cP(\fp^+_\rT,V_i)\otimes\overline{V_i}). \]
Then under this identification $K_{W_{ij}}(x,e)$ sits in $W_{ij}\otimes\overline{V_i}$, and therefore by (A2) 
this sits in the space with lowest weight $-(n_1\gamma_1+\cdots+n_r\gamma_r)$. 
That is, there exists a function $F_{ij}\in\cP(\fp^+_\rT,\mathrm{End}(V_i))$ such that 
\begin{gather*}
(\tilde{\tau}_i(q)F_{ij})(x)=\Delta_{\bn_{ij}}(q^{-1}e)F_{ij}(x) \qquad(q\in A_LN_L^-,x\in\fp^+_\rT),\\
\int_{K_L}(\tilde{\tau}(k)F_{ij})(x)dk=K_{W_{ij}}(x,e). 
\end{gather*}
We note that $\int_{K_L}(\tilde{\tau}(k)F_{ij})(x)dk$ is non-zero for any non-zero $N_L^-$-fixed vector $F_{ij}$, 
since we have $(F_{ij},K_{W_{ij}}(\cdot,e))_\tau\ne 0$, which is proved by using the Iwasawa decomposition $L=K_LA_LN_L^-$. 

From now, we compute $B_{ij}(\lambda)$ formally, allowing variable changes. 
By using $F_{ij}$, we rewrite $B_{ij}(\lambda)$ and $\Gamma_{ij}$. 
\begin{gather*}
B_{ij}(\lambda)=\int_{\Omega\cap(e-\Omega)}\Tr_{V_i}\left(\tau_i(P(e-x)^{-1})
F_{ij}(x)\vphantom{x^2}\right)\Delta(e-x)^{\lambda-p}\Delta(x)^bdx, \\
\Gamma_{ij}:=\int_\Omega\Tr_{V_i}(F_{ij}(x))e^{-\tr (x)}\Delta(x)^bdx. 
\end{gather*}
We set 
\begin{equation}\label{Ilambda}
I(y):=\int_{\Omega\cap(y-\Omega)}\Tr_{V_i}\left(\tau_i(P(y-x)^{-1})
F_{ij}(x)\vphantom{x^2}\right)\Delta(y-x)^{\lambda-p}\Delta(x)^bdx
\end{equation}
so that $I(e)=B_{ij}(\lambda)$. We take $q\in A_LN_L^-$ such that $y=qe$, and set $x=qz$. Then 
\begin{align*}
I(y)&=\int_{\Omega\cap(e-\Omega)}\Tr_{V_i}\left(\tau_i(P(q.(e-z))^{-1})
F_{ij}(qz)\right)\Delta(q.(e-z))^{\lambda-p}\Delta(qz)^b\Delta(qe)^{\frac{n_\rT}{r}}dz \\
&=\int_{\Omega\cap(e-\Omega)}\Tr_{V_i}\left(\tau_i({}^t\hspace{-1pt}q^{-1})\tau_i(P(e-z)^{-1})\tau_i(q^{-1})
F_{ij}(qz)\right)\Delta(e-z)^{\lambda-p}\Delta(z)^b\Delta(qe)^{\lambda-p+b+\frac{n_\rT}{r}}dz \\
&=\int_{\Omega\cap(e-\Omega)}\Tr_{V_i}\left(\tau_i(P(e-z)^{-1})F_{ij}(z)\right)\Delta_{\bn_{ij}}(qe)
\Delta(e-z)^{\lambda-p}\Delta(z)^b\Delta(qe)^{\lambda-\frac{n_\rT}{r}}dz \\
&=I(e)\Delta_{\bn_{ij}}(y)\Delta(y)^{\lambda-\frac{n_\rT}{r}}=B_{ij}(\lambda)\Delta_{\lambda+\bn_{ij}}(y)\Delta(y)^{-\frac{n_\rT}{r}}. 
\end{align*}
Now we calculate $\int_\Omega I(y)e^{-\tr(y)}dy$ by two ways. 
\begin{gather*}
\int_\Omega I(y)e^{-\tr(y)}dy
=B_{ij}(\lambda)\int_\Omega e^{-\tr(y)}\Delta_{\lambda+\bn_{ij}}(y)\Delta(y)^{-\frac{n_\rT}{r}}dy
=B_{ij}(\lambda)\Gamma_\Omega(\lambda+\bn_{ij}), \\
\begin{split}
\int_\Omega I(y)e^{-\tr(y)}dy&=\iint_{x\in\Omega,y-x\in\Omega}e^{-\tr(y)}\Tr_{V_i}\left(\tau_i(P(y-x)^{-1})
F_{ij}(x)\vphantom{x^2}\right)\Delta(y-x)^{\lambda-p}\Delta(x)^bdxdy \\
&=\iint_{x\in\Omega,z\in\Omega}e^{-\tr(x+z)}\Tr_{V_i}\left(\tau_i(P(z)^{-1})
F_{ij}(x)\vphantom{x^2}\right)\Delta(z)^{\lambda-p}\Delta(x)^bdxdz \\
&=\Tr_{V_i}\left(\int_\Omega e^{-\tr(z)}\tau_i(P(z)^{-1})\Delta(z)^{\lambda-p}dz\int_\Omega e^{-\tr(x)}F_{ij}(x)\Delta(x)^bdx\right). 
\end{split}
\end{gather*}
Therefore, formally 
\[ B_{ij}(\lambda)\Gamma_\Omega(\lambda+\bn_{ij})
=\Tr_{V_i}\left(\int_\Omega e^{-\tr(z)}\tau_i(P(z)^{-1})\Delta(z)^{\lambda-p}dz\int_\Omega e^{-\tr(x)}F_{ij}(x)\Delta(x)^bdx\right) \]
holds. By Fubini's theorem, variable changes are verified 
and the above equality exactly holds if 
\[ \iint_{x\in\Omega,z\in\Omega}e^{-\tr(x+z)}\left|\Tr_{V_i}\left(\tau_i(P(z)^{-1})
F_{ij}(x)\vphantom{x^2}\right)\right|\Delta(z)^{\Re(\lambda)-p}\Delta(x)^bdxdz<\infty \]
is verified, and since all norms on the finite-dimensional vector space $\mathrm{End}(V_i)$ are equivalent, 
this holds if 
\begin{gather}
\int_\Omega e^{-\tr(z)}\left|\tau_i(P(z)^{-1})\right|_{\tau_i,\mathrm{op}}\Delta(z)^{\Re(\lambda)-p}dz<\infty,\label{inttau}\\
\int_\Omega e^{-\tr(x)}\left|F_{ij}(x)\right|_{\tau_i,\mathrm{op}}\Delta(x)^bdx<\infty \label{intF}
\end{gather}
hold, where $|\cdot|_{\tau_i,\mathrm{op}}$ denotes the operator norm. Since 
\[ \left|F_{ij}(x)\right|_{\tau_i,\mathrm{op}}=\max_{u,v\in V_i\setminus\{0\}}\frac{|(F_{ij}(x)u,v)_{\tau_i}|}{|u|_{\tau_i}|v|_{\tau_i}} \]
holds and $(F_{ij}(x)u,v)_\tau$ is a polynomial on $\Omega$ for any $u,v\in V_i$, (\ref{intF}) exactly holds. 
Also, since $\tau_i(P(z)^{-1})$ is self-adjoint and positive definite for $z\in\Omega$, we have 
\[ \left|\tau_i(P(z)^{-1})\right|_{\tau_i,\mathrm{op}}
=\max_{u\in V_i\setminus\{0\}}\frac{\left|(\tau_i(P(z)^{-1})u,u)_{\tau_i}\right|}{|u|_{\tau_i}^2}, \]
and elements $v\in V_i$ such that 
\begin{equation}\label{inttau2}
\int_\Omega e^{-\tr(z)}\left|(\tau_i(P(z)^{-1})v,v)_{\tau_i}\right|\Delta(z)^{\Re(\lambda)-p}dz<\infty
\end{equation}
forms a $K_L$-invariant vector subspace, by the triangle inequality and the $K_L$-invariance of the integral. 
By assumption (A1), such vector subspace is either $V_i$ or $\{0\}$. 
Thus (\ref{inttau}) holds if and only if (\ref{inttau2}) holds for some non-zero $v\in V_i$. 
Moreover, again by assumption (A1), the integral 
\begin{equation}\label{gammaprimedef}
\Gamma'_i(\lambda):=\int_\Omega e^{-\tr(z)}\tau_i(P(z)^{-1})\Delta(z)^{\lambda-p}dz
\end{equation}
is proportional to the identity operator $I_{V_i}$ if (\ref{inttau2}) holds, 
since this $\Gamma'_i(\lambda)$ commutes with $K_L$-action. 
Now we prove (\ref{inttau2}) for $v\in V_i$ lowest weight vector, assuming $\Re(\lambda)+k_{i,r}>p-1$. 
Since the restricted lowest weight of $V_i$ is 
$\left.-\frac{k_{i,1}}{2}\gamma_1-\cdots-\frac{k_{i,r}}{2}\gamma_r\right|_{\fa_\fl}$, 
for $q\in A_LN_L^-$ we have 
\[ (\tau_i(P(qe)^{-1})v,v)_{\tau_i}=(\tau_i({}^t\hspace{-1pt}q^{-1}q^{-1})v,v)_{\tau_i}=|\tau_i(q^{-1})v|_{\tau_i}^2
=\Delta_{-\frac{\bk_i}{2}}(q^{-1}e)^2|v|_{\tau_i}^2=\Delta_{\bk_i}(qe)|v|_{\tau_i}^2, \]
and this is positive valued. Therefore we have 
\begin{align}\label{gammaprime}
(\Gamma'_i(\lambda)v,v)_{\tau_i}&=\int_\Omega e^{-\tr(z)}(\tau_i(P(z)^{-1})v,v)_{\tau_i}\Delta(z)^{\lambda-p}dz \notag \\
&=\int_\Omega e^{-\tr(z)}\Delta_{\bk_i}(z)\Delta(z)^{\lambda-\frac{n}{r}-\frac{n_\rT}{r}}dz|v|_{\tau_i}^2 \notag \\
&=\Gamma_\Omega\left(\lambda+\bk_i-\frac{n}{r}\right)|v|_{\tau_i}^2 
\end{align}
if $\Re(\lambda)+k_{i,r}>p-1$. That is, (\ref{inttau}) is verified, and 
$\Gamma'_i(\lambda)=\Gamma_\Omega\left(\lambda+\bk_i-\frac{n}{r}\right)I_{V_i}$ holds. Therefore, 
\[ B_{ij}(\lambda)=\frac{\Gamma_\Omega\left(\lambda+\bk_i-\frac{n}{r}\right)}{\Gamma_\Omega(\lambda+\bn_{ij})}
\Tr_V\left(\int_\Omega e^{-\tr(x)}\Delta(x)^bF_{ij}(x)dx\right)
=\frac{\Gamma_\Omega\left(\lambda+\bk_i-\frac{n}{r}\right)}{\Gamma_\Omega(\lambda+\bn_{ij})}\Gamma_i, \]
exactly holds, and 
\[ R_W(\lambda)=\frac{c_\lambda}{\sum_{ij}\tilde{a}_{ij}\Gamma_{ij}}
\sum_{ij}\tilde{a}_{ij}\frac{\Gamma_\Omega\left(\lambda+\bk_i-\frac{n}{r}\right)}{\Gamma_\Omega(\lambda+\bn_{ij})}\Gamma_{ij}. \]
By putting $\tilde{a}_{ij}\Gamma_{ij}=:a_{ij}$, we get the desired formula. 

When $W=V$, clearly we have $\mathrm{rest}(V)=\oplus_iV_i$, and $K_V(z,w)=I_V$, $K_{V_i}(z,w)=I_{V_i}$. Thus, the coefficients 
\begin{align*}
a_i=\Gamma_i&=\int_\Omega\Tr_{V_i}(K_{V_i}(x,e))e^{-\tr (x)}\Delta(x)^bdx\\
&=\int_\Omega\Tr_{V_i}(I_{V_i})e^{-\tr (x)}\Delta(x)^bdx=(\dim V_i)\Gamma_\Omega\left(\frac{n}{r}\right). 
\end{align*}
Also, by assumption (A1), $K_L$-spherical vectors in 
$(\tilde{\tau},\mathrm{End}(V_i))\simeq (\tau_i\otimes\overline{\tau_i},V_i\otimes\overline{V_i})$ 
is proportional to $I_{V_i}$, that is, $\dim\mathrm{End}(V_i)^{K_L}=1$. 
Therefore, assumption (A2) is automatically satisfied, with $\bn_i=\bk_i$. 
Since $c_\lambda$ is determined such that $R_{V,\lambda}=1$, we have 
\begin{align*}
c_\lambda^{-1}&=\frac{1}{\sum_i(\dim V_i)\Gamma_\Omega\left(\frac{n}{r}\right)}
\sum_i(\dim V_i)\Gamma_\Omega\left(\frac{n}{r}\right)\frac{\Gamma_\Omega\left(\lambda+\bk_i-\frac{n}{r}\right)}
{\Gamma_\Omega(\lambda+\bk_i)}\\
&=\frac{1}{\dim V}\sum_i(\dim V_i)\frac{\Gamma_\Omega\left(\lambda+\bk_i-\frac{n}{r}\right)}{\Gamma_\Omega(\lambda+\bk_i)}, 
\end{align*}
and this completes the proof. \qed

\begin{remark}
The integral $\Gamma'_{i,\lambda}$ in (\ref{gammaprimedef}) is essentially the same as the ``Gamma function'' in 
\cite[Definition 3.1]{GK}, \cite[Section 4]{HN} on $\mathrm{End}(V_i)$, 
or the integral with the measure $R_\mu$ in \cite[Theorem 3.4]{C}, 
and the property of $\Gamma'_{i,\lambda}$ or the finiteness of (\ref{inttau}) have been already proved. 
However, since the notation is different, the author wrote the proof for completeness. 
\end{remark}

If $(\tau,V)|_{\fk_\rT^\BC}$ is still irreducible and 
$\mathrm{rest}(W)\subset \cP(\fp^+_\rT,V)$ consists of one irreducible $K_\rT^\BC$-module, 
then Theorem \ref{keythm} becomes easier. 
\begin{corollary}\label{tubecor}
Suppose $(\tau,V)|_{K^\BC_\rT}$ has a restricted lowest weight 
$\left.-\left(\frac{k_{1}}{2}\gamma_1+\cdots+\frac{k_{r}}{2}\gamma_r\right)\right|_{\fa_\fl}$. 
Let $W\subset\cP(\fp^+,V)$ be a $K^\BC$-irreducible subspace. We assume 
\begin{itemize}
\item[(A0)] $\mathrm{rest}(W)\subset \cP(\fp^+_\rT,V)$ is irreducible as a $K_\rT^\BC$-module. 
\item[(A1')] $(\tau,V)|_{K_L}$ still remains irreducible. 
\item[(A2')] All the $K_L$-spherical irreducible subspaces 
in $\mathrm{rest}(W)\otimes\overline{V}$ have the same lowest weight $-\left(n_{1}\gamma_1+\cdots+n_{r}\gamma_r\right)$. 
\end{itemize}
Then the integral $\Vert f\Vert_{\lambda,\tau}^2$ converges for any $f\in W$ if $\Re(\lambda)+k_{r}>p-1$. 
Moreover, we have 
\[ c_\lambda=\frac{\Gamma_\Omega(\lambda+\bk)}{\Gamma_\Omega\left(\lambda+\bk-\frac{n}{r}\right)}, \]
and for any $f\in W$, we have 
\[ \frac{\Vert f\Vert_{\lambda,\tau}^2}{\Vert f\Vert_{F,\tau}^2}=\frac{\Gamma_\Omega(\lambda+\bk)}{\Gamma_\Omega(\lambda+\bn)}
=\frac{(\lambda)_\bk}{(\lambda)_\bn}=\frac{1}{(\lambda+\bk)_{\bn-\bk}}. \]
\end{corollary}
The assumption (A0) is automatically satisfied if 
\begin{itemize}
\item $G=G_\rT$ i.e. $G$ is of tube type, or 
\item $G=SU(q,r)$ ($q\le r$), and $V=\BC\boxtimes V'$ as a $K=S(U(q)\times U(r))$-module. 
\end{itemize}
In Section \ref{secttube}, we deal with these cases explicitly, and in Section \ref{sectnontube}, we deal with the cases such that 
Corollary \ref{tubecor} is not applicable. 
To remove the ambiguity of the action of the center, 
we assume $k_{i,r}\ge 0$ for any $i$, and $k_{i,r}=0$ for some $i$.

\section{Norm computation: Tube type case}\label{secttube}
\subsection{Explicit roots}\label{explicitroots}
Before starting the computation of norms, we fix the notation about roots of classical Lie algebras of Hermitian type. 

Let $\fg=\fk\oplus\fp$ be a classical simple Lie algebra of Hermitian type, i.e. one of 
$\mathfrak{sp}(r,\BR)$, $\mathrm{su}(q,s)$, $\mathfrak{so}^*(2s)$, or $\mathfrak{so}(2,n)$. 
We fix a Cartan subalgebra $\fh\subset\fk$. Then $\fh$ automatically becomes a Cartan subalgebra of $\fg$. 
We take a basis 
\begin{align*}
\{ t_1,t_2,\ldots,t_r\}&\subset\sqrt{-1}\fh & &(\fg=\mathfrak{sp}(r,\BR)),\\
\{ t_1,t_2,\ldots,t_{q+s}\}&\subset(\sqrt{-1}\fh)\oplus\BR & &(\fg=\mathfrak{su}(q,s)),\\
\{ t_1,t_2,\ldots,t_s\}&\subset\sqrt{-1}\fh & &(\fg=\mathfrak{so}^*(2s)),\\
\{ t_0,t_1,\ldots,t_{\lfloor n/2\rfloor}\}&\subset\sqrt{-1}\fh & &(\fg=\mathfrak{so}(2,n)), 
\end{align*}
with the dual basis $\{\varepsilon_j\}$, such that the simple systems 
$\Pi_{\fg^\BC}$, $\Pi_{\fk^\BC}$ of positive roots $\Delta_+(\fg^\BC,\fh^\BC)$, $\Delta_+(\fk^\BC,\fh^\BC)$ are given by 
\begin{align*}
\Pi_{\fk^\BC}&=\begin{cases}
\{\varepsilon_j-\varepsilon_{j+1}:j=1,\ldots,r-1\} & (\fg=\mathfrak{sp}(r,\BR)),\\
\{\varepsilon_j-\varepsilon_{j+1}:j=1,\ldots,q-1\} & \\
\hspace{22pt}\cup\{\varepsilon_{j+1}-\varepsilon_j:j=q+1,\ldots,q+s-1\} & (\fg=\mathfrak{su}(q,s)),\\
\{\varepsilon_j-\varepsilon_{j+1}:j=1,\ldots,s-1\} & (\fg=\mathfrak{so}^*(2s)),\\
\{\varepsilon_j-\varepsilon_{j+1}:j=1,\ldots,s-1\}\cup\{\varepsilon_{s-1}+\varepsilon_s\} & (\fg=\mathfrak{so}(2,2s)),\\
\{\varepsilon_j-\varepsilon_{j+1}:j=1,\ldots,s-1\}\cup\{\varepsilon_s\} & (\fg=\mathfrak{so}(2,2s+1)), \end{cases} \\
\Pi_{\fg^\BC}&=\Pi_{\fk^\BC}\cup\begin{cases}
\{2\varepsilon_r\} & (\fg=\mathfrak{sp}(r,\BR)),\\
\{\varepsilon_q-\varepsilon_{q+s}\} & (\fg=\mathfrak{su}(q,s)),\\
\{\varepsilon_{s-1}+\varepsilon_s\} & (\fg=\mathfrak{so}^*(2s)),\\
\{\varepsilon_0-\varepsilon_1\} & (\fg=\mathfrak{so}(2,n)). \end{cases}
\end{align*}
Then the central character $d\chi$ of $\fk^\BC$ is given by 
\[ d\chi=\begin{cases}
\varepsilon_1+\cdots+\varepsilon_r & (\fg=\mathfrak{sp}(r,\BR)),\\
\varepsilon_1+\cdots+\varepsilon_q=-(\varepsilon_{q+1}+\cdots+\varepsilon_{q+s}) & (\fg=\mathfrak{su}(q,s)),\\
\frac{1}{2}(\varepsilon_1+\cdots+\varepsilon_s) & (\fg=\mathfrak{so}^*(2s)),\\
\varepsilon_0 & (\fg=\mathfrak{so}(2,n)), \end{cases} \]
and the maximal set of strongly orthogonal roots $\{\gamma_1,\ldots,\gamma_{\rank_\BR\fg}\}$ is given by 
\begin{align*}
\gamma_j&=2\varepsilon_j & &(j=1,\ldots,r) & &(\fg=\mathfrak{sp}(r,\BR)),\\
\gamma_j&=\varepsilon_j-\varepsilon_{q+j} & &(j=1,\ldots,\min\{q,s\}) & &(\fg=\mathfrak{su}(q,s)),\\
\gamma_j&=\gamma_{2j-1}+\gamma_{2j} & &(j=1,\ldots,\lfloor s/2\rfloor) & &(\fg=\mathfrak{so}^*(2s)),\\
\gamma_1&=\varepsilon_0+\varepsilon_1,\quad \gamma_2=\varepsilon_0-\varepsilon_1 & & & &(\fg=\mathfrak{so}(2,n)). 
\end{align*}
When $\fg=\mathfrak{sp}(r,\BR)$, $\mathfrak{su}(r,r)$, $\mathfrak{so}^*(4r)$ or $\mathfrak{so}(2,n)$, $\fg$ is of tube type, i.e. 
$\fg=\fg_\rT$ holds. On the other hand, when $\mathfrak{su}(q,s)$ ($q\ne s$) or $\fg=\mathfrak{so}^*(4r+2)$, 
$\fg$ is of non-tube type, and we have $\fg_\rT=\mathfrak{su}(r,r)$ ($r:=\min\{q,s\}$), or $\fg_\rT=\mathfrak{so}^*(4r)$ respectively. 
Let $\fh_\rT:=\fh\cap\fg_\rT$. Then we have 
\begin{align*}
\sqrt{-1}\fh_\rT&=\operatorname{span}(\{t_j-t_{j+1}:j=1,\ldots, r-1,q+1,\ldots, q+r-1\} & & \\
&\hspace{220pt}\cup \{t_r-t_{q+r}\})
& &(\fg=\mathfrak{su}(q,s)),\\
\sqrt{-1}\fh_\rT&=\operatorname{span}\{t_1,\ldots,t_{2r}\} & &(\fg=\mathfrak{so}^*(4r+2)). 
\end{align*}
Also, $\fa_\fl\subset\sqrt{-1}\fh_\rT$ is given by 
\[ \fa_\fl=\begin{cases}
\sqrt{-1}\fh & (\fg_\rT=\mathfrak{sp}(r,\BR)),\\
\operatorname{span}\{ t_j-t_{q+j}:j=1,\ldots,r\} & (\fg_\rT=\mathfrak{su}(r,r)),\\
\operatorname{span}\{ t_{2j-1}+t_{2j}:j=1,\ldots,r\} & (\fg_\rT=\mathfrak{so}^*(4r)),\\
\operatorname{span}\{ t_0,t_1\} & (\fg_\rT=\mathfrak{so}(2,n)). \end{cases} \]

In general, we consider $\mathfrak{gl}(s,\BC)$ or $\mathfrak{so}(n,\BC)$, and parametrize their irreducible representations. 
We fix the positive root system of $\mathfrak{gl}(s,\BC)$ such that its simple system is given by 
$\{\varepsilon_j-\varepsilon_{j+1}:j=1,\ldots,s-1\}$, and for $\bm\in\BZ_+^s$, let $(\tau_\bm^{(s)},V_\bm^{(s)})$, 
$(\tau_\bm^{(s)\vee},V_\bm^{(s)\vee})$ be the finite-dimensional irreducible representation of $\mathfrak{gl}(s,\BC)$ 
with highest weight $m_1\varepsilon_1+\cdots+m_s\varepsilon_s$, $-m_s\varepsilon_1-\cdots-m_1\varepsilon_s$ respectively. 
Similarly, we fix the positive root system of $\mathfrak{so}(n,\BC)$ such that its simple system is given by 
\begin{align*}
&\{ \varepsilon_j-\varepsilon_{j+1}:j=1,\ldots,s-1\}\cup\{\varepsilon_{s-1}+\varepsilon_s\} & &(n=2s),\\
&\{ \varepsilon_j-\varepsilon_{j+1}:j=1,\ldots,s-1\}\cup\{\varepsilon_s\} & &(n=2s+1),
\end{align*}
and for $\bm\in\BZ^s\cup\left(\BZ+\frac{1}{2}\right)^s$ with 
\begin{align*}
&m_1\ge m_2\ge\cdots\ge m_{s-1}\ge |m_s| & &(n=2s),\\
&m_1\ge m_2\ge\cdots\ge m_{s-1}\ge m_s\ge 0 & &(n=2s+1),
\end{align*}
let $(\tau_\bm^{[n]},V_\bm^{[n]})$ be the finite-dimensional irreducible representation of $\mathfrak{so}(n,\BC)$ 
with highest weight $m_1\varepsilon_1+\cdots+m_s\varepsilon_s$. 
Then $(\tau_\bm^{(r)\vee},V_\bm^{(r)\vee})$, $(\tau_\bm^{(q)\vee}\boxtimes \tau_\bn^{(s)},V_\bm^{(q)\vee}\otimes V_\bn^{(s)})$, 
$(\tau_\bm^{(s)\vee},V_\bm^{(s)\vee})$ and $(\chi^{m_0}\boxtimes \tau_\bm^{[n]},\BC_{m_0}\otimes V_\bm^{[n]})$ are 
naturally identified with the representation of $\fk^\BC$ 
for $\fg=\mathfrak{sp}(r,\BR)$, $\mathfrak{su}(q,s)$, $\mathfrak{so}^*(2s)$ and $\mathfrak{so}(2,n)$ respectively. 
Their restricted lowest weights are given by 
\begin{align*}
&-\left.\frac{1}{2}(m_1\gamma_1+\cdots+m_r\gamma_r)\right|_{\fa_\fl} 
& &(\fg=\mathfrak{sp}(r,\BR), & &V=V_\bm^{(r)\vee}),\\
&-\left.\frac{1}{2}((m_1-n_1)\gamma_1+\cdots+(m_r-n_r)\gamma_r)\right|_{\fa_\fl}
& &(\fg=\mathfrak{su}(q,s), & &V=V_\bm^{(q)\vee}\boxtimes V_\bn^{(s)}),\\
&-\left.\frac{1}{2}((m_1+m_2)\gamma_1+\cdots+(m_{2r-1}+m_{2r})\gamma_r)\right|_{\fa_\fl}
& &(\fg=\mathfrak{so}^*(2s), & &V=V_\bm^{(s)\vee}),\\
&-\left.\frac{1}{2}((m_0+m_1)\gamma_1+(m_0-m_1)\gamma_2)\right|_{\fa_\fl}
& &(\fg=\mathfrak{so}(2,n), & &V=\BC_{m_0}\boxtimes V_\bm^{[n]}). 
\end{align*}
We will omit the superscript $(s)$ or $[n]$ if there is no confusion. 

Next we determine $(\bar{\tau},\bar{V})$ for each representation $(\tau,V)$ of $\fk_\rT^\BC$. 
As in Section \ref{root}, let $\bar{\cdot}$ be the involution of $\fk_\rT^\BC$ fixing $\fl$. 
Then $\bar{\cdot}$ acts on $\fh_\rT^\BC$ anti-linearly, and fixes $\fa_\fl\oplus(\fm_\fl\cap\fh)$. 
Therefore $\bar{\cdot}|_{\fh_\rT^\BC}$ is characterized by 
\begin{align*}
&\overline{t_j}=t_j & &(\fg_\rT=\mathfrak{sp}(r,\BR)),\\
&\overline{t_j}=-t_{q+j},\; \overline{t_{q+j}}=-t_j & &(\fg_\rT=\mathfrak{su}(r,r)),\\
&\overline{t_{2j-1}}=t_{2j},\; \overline{t_{2j}}=t_{2j-1} & &(\fg_\rT=\mathfrak{so}^*(4r)),\\
&\overline{t_j}=\begin{cases} t_j & (j=0,1)\\ -t_j & (j=2,\ldots,s)\end{cases} & &(\fg_\rT=\mathfrak{so}(2,n),\; s=\lfloor n/2\rfloor). 
\end{align*}
We take an element $w\in N_K(\fh)\subset K$ (the normalizer of $\fh$ in $K$, or the ``Weyl group'' of $\fh$) such that 
\begin{align*}
&Ad(w)t_j=t_j & &(\fg_\rT=\mathfrak{sp}(r,\BR),\mathfrak{su}(r,r)),\\
&Ad(w)t_{2j-1}=t_{2j},\; Ad(w)t_{2j}=t_{2j-1} & &(\fg_\rT=\mathfrak{so}^*(4r)), \\
&Ad(w)t_j=\begin{cases}t_j&(j=0,1,s)\\ -t_j&(j=2,3,\ldots,s-1) \end{cases} 
& &(\fg_\rT=\mathfrak{so}(2,n),\; n\in 4\BN,\; s=\lfloor n/2\rfloor),\\
&Ad(w)t_j=\begin{cases}t_j&(j=0,1)\\ -t_j&(j=2,3,\ldots,s) \end{cases} 
& &(\fg_\rT=\mathfrak{so}(2,n),\; n\notin 4\BN,\; s=\lfloor n/2\rfloor).
\end{align*}
Then we have 
\begin{align*}
&Ad(w)\overline{t_j}=t_j & &(\fg_\rT=\mathfrak{sp}(r,\BR),\mathfrak{so}^*(4r)),\\
&Ad(w)\overline{t_j}=-t_{q+j},\; Ad(w)\overline{t_{q+j}}=-t_j & &(\fg_\rT=\mathfrak{su}(r,r)),\\
&Ad(w)\overline{t_j}=\begin{cases}t_j&(j=0,1,\ldots,s-1)\\ -t_s&(j=s) \end{cases} 
& &(\fg_\rT=\mathfrak{so}(2,n),\; n\in 4\BN,\; s=\lfloor n/2\rfloor),\\
&Ad(w)\overline{t_j}=t_j & &(\fg_\rT=\mathfrak{so}(2,n),\; n\notin 4\BN,\; s=\lfloor n/2\rfloor),
\end{align*}
and thus $Ad(w)\bar{\cdot}|_{\fh_\rT^\BC}$ preserves the positive Weyl chamber. 
This implies $Ad(w)\bar{\cdot}$ preserves the Borel subalgebra $\fb\subset\fk_\rT^\BC$. 
Let $(\tau,V)$ be an irreducible $\fk_\rT$-module with highest weight $\mu\in(\fh_\rT^\BC)^\vee$ 
and we extend $\mu$ on $\fb$ such that it is trivial on the nilradical. 
Let $v\in V$ be the highest weight. Then for $b\in\fb$ we have 
\[ d\bar{\tau}(b)(\overline{\tau(w^{-1}){v}})=\overline{d\tau(\bar{b})\tau(w^{-1})v}=\overline{\tau(w^{-1})d\tau(Ad(w)\bar{b})v}
=\overline{\mu(Ad(w)\bar{b})}\,\overline{\tau(w^{-1})v}. \]
Therefore $(\bar{\tau},\bar{V})$ has the highest weight vector $\overline{\tau(w^{-1})v}$ with highest weight 
$t\mapsto\overline{\mu(Ad(w)\bar{t})}$ ($t\in\fh_\rT^\BC$). 
Thus we conclude 
\begin{align*}
\overline{V_\bm^{(r)\vee}} \simeq& V_\bm^{(r)\vee} & &(\fg_\rT=\mathfrak{sp}(r,\BR)),\\
\overline{V_\bm^{(r)\vee}\boxtimes V_\bn^{(r)}} \simeq& V_\bn^{(r)\vee}\boxtimes V_\bm^{(r)} & &(\fg_\rT=\mathfrak{su}(r,r)),\\
\overline{V_\bm^{(2r)\vee}} \simeq& V_\bm^{(2r)\vee} & &(\fg_\rT=\mathfrak{so}^*(4r)),\\
\overline{\BC_{m_0}\boxtimes V_{(m_1,\ldots,m_{s-1},m_s)}^{[n]}} \simeq& \BC_{m_0}\boxtimes V_{(m_1,\ldots,m_{s-1},-m_s)}^{[n]}
& &(\fg_\rT=\mathfrak{so}(2,n),\; n\in 4\BN,\; s=\lfloor n/2\rfloor),\\ 
\overline{\BC_{m_0}\boxtimes V_{(m_1,\ldots,m_{s-1},m_s)}^{[n]}} \simeq& \BC_{m_0}\boxtimes V_{(m_1,\ldots,m_{s-1},m_s)}^{[n]}
& &(\fg_\rT=\mathfrak{so}(2,n),\; n\notin 4\BN,\; s=\lfloor n/2\rfloor).
\end{align*}
In the following sections, we compute the ratio of norms by using Corollary \ref{tubecor}.

\subsection{$Sp(r,\BR)$}
In this subsection we set $G=Sp(r,\BR)$. This is of tube type, and we have 
\begin{gather*}
K\simeq U(r),\quad \fp^\pm\simeq \mathrm{Sym}(r,\BC),\quad L\simeq GL(r,\BR),\quad K_L\simeq O(r), \\
r=r,\quad n=\frac{1}{2}r(r+1),\quad d=1,\quad p=r+1. 
\end{gather*}
We want to calculate the norm $\Vert\cdot\Vert_{\lambda,\tau}$ of $\cO(D,V)$ in the case 
$V=V_{\varepsilon_1+\cdots+\varepsilon_k}^\vee\simeq\bigwedge^k(\BC^r)^\vee$ ($k=0,1,\ldots,r-1$). 
These $V$ have the restricted lowest weight $-\left.\frac{1}{2}(\gamma_1+\cdots+\gamma_s)\right|_{\fa_\fl}$, 
and remain irreducible even if restricted to $K_L=O(r)$, i.e. satisfy assumption (A1') of corollary \ref{tubecor}. 
Thus the norm $\Vert\cdot\Vert_{\lambda,\tau_{\varepsilon_1+\cdots+\varepsilon_k}^\vee}^2$ converges if $\Re\lambda>r$, 
and the normalizing constant $c_\lambda$ is given by 
\[ c_\lambda=\frac{\Gamma_\Omega(\lambda+\varepsilon_1+\cdots+\varepsilon_k)}
{\Gamma_\Omega\left(\lambda+\varepsilon_1+\cdots+\varepsilon_k-\frac{r+1}{2}\right)}
=\frac{\prod_{j=1}^k\Gamma\left(\lambda-\frac{j-1}{2}+1\right)\prod_{j=k+1}^r\Gamma\left(\lambda-\frac{j-1}{2}\right)}
{\prod_{j=1}^k\Gamma\left(\lambda-\frac{j+r}{2}+1\right)\prod_{j=k+1}^r\Gamma\left(\lambda-\frac{j+r}{2}\right)}. \]
First we compute the $K$-type decomposition of $\cO(D,V)_K=\cP(\fp^+)\otimes V_{\varepsilon_1+\cdots+\varepsilon_k}^\vee$. 
To do this, we quote the following lemma. 
\begin{lemma}[{\cite[$\mathsection$79, Example 3]{Z}}]
\[ V_{\bm}^\vee\otimes V_{\varepsilon_1+\cdots+\varepsilon_k}^\vee
=\bigoplus_{\substack{\bk\in\{0,1\}^r,\,|\bk|=k\\ \bm+\bk\in\BZ^r_+}}V_{\bm+\bk}^\vee. \]
\end{lemma}
By this lemma and Theorem \ref{HKS}, we have 
\begin{align*}
\cP(\fp^+)\otimes V_{\varepsilon_1+\cdots+\varepsilon_k}^\vee
&=\bigoplus_{\bm\in\BZ^r_{++}}V_{2\bm}^\vee\otimes V_{\varepsilon_1+\cdots+\varepsilon_k}^\vee\\
&=\bigoplus_{\bm\in\BZ^r_{++}}\bigoplus_{\substack{\bk\in\{0,1\}^r,\,|\bk|=k\\ \bm+\bk\in\BZ^r_+}}V_{2\bm+\bk}^\vee. 
\end{align*}
Second, for each $K$-type $V_{2\bm+\bk}^\vee$, we compute 
$V_{2\bm+\bk}^\vee\otimes\overline{V_{\varepsilon_1+\cdots+\varepsilon_k}^\vee}
\simeq V_{2\bm+\bk}^\vee\otimes V_{\varepsilon_1+\cdots+\varepsilon_k}^\vee$. 
\[ V_{2\bm+\bk}^\vee\otimes V_{\varepsilon_1+\cdots+\varepsilon_k}^\vee
=\bigoplus_{\substack{\bk'\in\{0,1\}^r,\,|\bk'|=k\\ 2\bm+\bk+\bk'\in\BZ^r_+}}V_{2\bm+\bk+\bk'}^\vee. \]
By Theorem \ref{sph}, $V_{2\bm+\bk+\bk'}^\vee$ is $K_L$-spherical if and only if each component of $2\bm+\bk+\bk'$ 
is even, that is, $\bk=\bk'$. Thus, the only $K_L$-spherical submodule in 
$V_{2\bm+\bk}^\vee\otimes V_{\varepsilon_1+\cdots+\varepsilon_k}^\vee$ is $V_{2\bm+2\bk}^\vee$, 
and $V_{2\bm+\bk}^\vee$ satisfies the assumption (A2') of Corollary \ref{tubecor} with $\bn=\bm+\bk$. 
Therefore by Corollary \ref{tubecor}, for $f\in V_{2\bm+\bk}^\vee$ we have 
\[ \frac{\Vert f\Vert_{\lambda,\tau_{\varepsilon_1+\cdots+\varepsilon_k}^\vee}^2}
{\Vert f\Vert_{F,\tau_{\varepsilon_1+\cdots+\varepsilon_k}^\vee}^2}
=\frac{(\lambda)_{\varepsilon_1+\cdots+\varepsilon_k}}{(\lambda)_{\bm+\bk}}
=\frac{\prod_{j=1}^k\left(\lambda-\frac{1}{2}(j-1)\right)}{\prod_{j=1}^r\left(\lambda-\frac{1}{2}(j-1)\right)_{m_j+k_j}}. \]

We summarize this subsection. 
\begin{theorem}\label{sprr}
When $G=Sp(r,\BR)$, and $(\tau,V)=(\tau_{\varepsilon_1+\cdots+\varepsilon_k}^\vee,V_{\varepsilon_1+\cdots+\varepsilon_k}^\vee)$, 
$\Vert\cdot\Vert_{\lambda,\tau}^2$ converges if $\Re\lambda>r$, the normalizing constant $c_\lambda$ is given by  
\[ c_\lambda=\frac{\prod_{j=1}^k\Gamma\left(\lambda-\frac{j-1}{2}+1\right)\prod_{j=k+1}^r\Gamma\left(\lambda-\frac{j-1}{2}\right)}
{\prod_{j=1}^k\Gamma\left(\lambda-\frac{j+r}{2}+1\right)\prod_{j=k+1}^r\Gamma\left(\lambda-\frac{j+r}{2}\right)}, \]
the $K$-type decomposition of $\cO(D,V)_K$ is given by 
\[ \cP(\fp^+)\otimes V_{\varepsilon_1+\cdots+\varepsilon_k}^\vee
=\bigoplus_{\bm\in\BZ^r_{++}}\bigoplus_{\substack{\bk\in\{0,1\}^r,\,|\bk|=k\\ \bm+\bk\in\BZ^r_+}}V_{2\bm+\bk}^\vee, \]
and for $f\in V_{2\bm+\bk}^\vee$, the ratio of norms is given by 
\begin{align*}
\frac{\Vert f\Vert_{\lambda,\tau_{\varepsilon_1+\cdots+\varepsilon_k}^\vee}^2}
{\Vert f\Vert_{F,\tau_{\varepsilon_1+\cdots+\varepsilon_k}^\vee}^2}
&=\frac{\prod_{j=1}^k\left(\lambda-\frac{1}{2}(j-1)\right)}{\prod_{j=1}^r\left(\lambda-\frac{1}{2}(j-1)\right)_{m_j+k_j}}\\
&=\frac{1}{\prod_{j=1}^k\left(\lambda-\frac{1}{2}(j-1)+1\right)_{m_j+k_j-1}
\prod_{j=k+1}^r\left(\lambda-\frac{1}{2}(j-1)\right)_{m_j+k_j}}. 
\end{align*}
\end{theorem}

\subsection{$SU(q,s)$}\label{sectsuqs}
In this subsection we set $G=SU(q,s)$, with $q\ge s$. Then we have 
\begin{gather*}
K\simeq S(U(q)\times U(s)),\quad \fp^\pm\simeq M(q,s;\BC), \quad G_T\simeq SU(s,s),\quad K_T\simeq S(U(s)\times U(s)), \\ 
L\simeq \{l\in GL(s,\BC):\det l\in \BR^\times\}, \quad 
K_L\simeq \{k\in U(s):\det k=\pm 1\}, \\
r=s,\quad n=qs,\quad d=2,\quad p=q+s. 
\end{gather*}
We want to calculate the norm $\Vert\cdot\Vert_{\lambda,\tau}$ of $\cO(D,V)$ in the case 
$(\tau,V)=(\tau_\bzero^{(q)\vee}\boxtimes\tau_\bk^{(s)},V_\bzero^{(q)\vee}\otimes V_\bk^{(s)})
=(\bone^{(q)}\boxtimes\tau_\bk^{(s)},\BC\otimes V_\bk^{(s)})$ ($\bk\in\BZ_{++}^s$). 
These $V$ have the restricted lowest weight $-\left.\frac{1}{2}(k_1\gamma_1+\cdots+k_s\gamma_s)\right|_{\fa_\fl}$, 
and remain irreducible even if restricted to $K_L=\mathrm{diag}(\{\pm 1\}\times SU(s))$ 
i.e. satisfy assumption (A1') of corollary \ref{tubecor}. 
Thus $\Vert\cdot\Vert_{\lambda,\tau}^2$ converges if $\Re\lambda+k_s>q+s-1$, and the normalizing constant $c_\lambda$ is given by 
\[ c_\lambda=\frac{\Gamma_\Omega(\lambda+\bk)}{\Gamma_\Omega(\lambda+\bk-q)}
=\prod_{j=1}^s(\lambda-(j-1)+k_j-q)_q. \]
First, we compute the $K$-type decomposition of $\cO(D,V)_K=\cP(\fp^+)\otimes \left(\BC\boxtimes V_\bk^{(s)}\right)$. 
By Theorem \ref{HKS} we have 
\begin{align*}
\cP(\fp^+)\otimes \left(\BC\boxtimes V_\bk^{(s)}\right)
&=\bigoplus_{\bm\in\BZ^s_{++}}\left(V_\bm^{(q)\vee}\boxtimes V_\bm^{(s)}\right)\otimes \left(\BC\boxtimes V_\bk^{(s)}\right)\\
&=\bigoplus_{\bm\in\BZ^s_{++}}\bigoplus_{\bn\in\bm+\mathrm{wt}(\bk)}c^\bn_{\bk,\bm}V_\bm^{(q)\vee}\boxtimes V_\bn^{(s)}. 
\end{align*}
where $V_\bm^{(q)\vee}$ is the abbreviation of $V_{(m_1,\ldots,m_s,0,\ldots,0)}^{(q)\vee}$, 
$\mathrm{wt}(\bk)$ is the set of all weights in the $GL(s,\BC)$-module $V_\bk^{(s)}$, 
and $c^\bn_{\bk,\bm}$ are some non-negative integers. 
Second, let $\mathrm{rest}:\cP(\fp^+)\otimes V\to \cP(\fp^+_\rT)\otimes V$ be the restriction map, as in Section \ref{key}. 
Then we have 
\[ \mathrm{rest}\left(V_\bm^{(q)\vee}\boxtimes V_\bn^{(s)}\right)=V_\bm^{(s)\vee}\boxtimes V_\bn^{(s)}, \]
so each $K$-type $V_\bm^{(q)\vee}\boxtimes V_\bn^{(s)}$ satisfies the assumption (A0) in Corollary \ref{tubecor}. 
Third, we compute the tensor product with $\overline{\BC\boxtimes V_\bn^{(s)}}\simeq V_\bn^{(s)}\boxtimes \BC$. 
\[ \left(V_\bm^{(s)\vee}\boxtimes V_\bn^{(s)}\right)\otimes \left(V_\bk^{(s)\vee}\boxtimes \BC\right)
=\bigoplus_{\bn'\in\bm+\mathrm{wt}(\bk)}c^{\bn'}_{\bk,\bm}V_{\bn'}^{(s)\vee}\boxtimes V_\bn^{(s)}. \]
By Theorem \ref{sph}, $V_{\bn'}^{(s)\vee}\boxtimes V_\bn^{(s)}$ is $K_L$-spherical if and only if $\bn'=\bn$, 
so all irreducible $K_L$-spherical submodules in 
$\left(V_\bm^{(s)\vee}\boxtimes V_\bn^{(s)}\right)\otimes \left(V_\bk^{(s)\vee}\boxtimes \BC\right)$ 
are isomorphic to $V_{\bn}^{(s)\vee}\boxtimes V_\bn^{(s)}$, which has the lowest weight $-(n_1\gamma_1+\cdots+n_s\gamma_s)$. 
Therefore each $K$-type satisfies the assumption (A2'), 
and by Corollary \ref{tubecor}, for $f\in V_\bm^{(q)\vee}\boxtimes V_\bn^{(s)}$ we have 
\[ \frac{\Vert f\Vert_{\lambda,\bone^{(q)}\boxtimes\tau_\bk^{(s)}}^2}{\Vert f\Vert_{F,\bone^{(q)}\boxtimes\tau_\bk^{(s)}}^2}=
\frac{(\lambda)_\bk}{(\lambda)_\bn}
=\frac{\prod_{j=1}^s(\lambda-(j-1))_{k_j}}{\prod_{j=1}^s(\lambda-(j-1))_{n_j}}. \]
We summarize this subsection. 
\begin{theorem}\label{suqs}
When $G=SU(q,s)$ $(q\ge s)$, and $(\tau,V)=(\bone^{(q)}\boxtimes\tau_\bk^{(s)},\BC\otimes V_\bk^{(s)})$ $(\bk\in\BZ_{++}^s)$, 
$\Vert\cdot\Vert_{\lambda,\tau}^2$ converges if $\Re\lambda+k_s>q+s-1$, 
the normalizing constant $c_\lambda$ is given by 
\[ c_\lambda=\prod_{j=1}^s(\lambda-(j-1)+k_j-q)_q, \]
the $K$-type decomposition of $\cO(D,V)_K$ is given by 
\[ \cP(\fp^+)\otimes \left(\BC\boxtimes V_\bk^{(s)}\right)
=\bigoplus_{\bm\in\BZ^s_{++}}\bigoplus_{\bn\in\bm+\mathrm{wt}(\bk)}c^\bn_{\bk,\bm}V_\bm^{(q)\vee}\boxtimes V_\bn^{(s)}, \]
and for $f\in V_\bm^{(q)\vee}\boxtimes V_\bn^{(s)}$, the ratio of norms is given by 
\[ \frac{\Vert f\Vert_{\lambda,\bone^{(q)}\boxtimes\tau_\bk^{(s)}}^2}{\Vert f\Vert_{F,\bone^{(q)}\boxtimes\tau_\bk^{(s)}}^2}
=\frac{\prod_{j=1}^s(\lambda-(j-1))_{k_j}}{\prod_{j=1}^s(\lambda-(j-1))_{n_j}}
=\frac{1}{\prod_{j=1}^s(\lambda-(j-1)+k_j)_{n_j-k_j}}. \]
\end{theorem}

\subsection{$SO^*(4r)$}\label{sectsostar}
In this subsection we set $G=SO^*(4r)$. Then we have 
\begin{gather*}
K\simeq U(2r),\quad \fp^\pm\simeq \mathrm{Skew}(2r,\BC),\quad 
L\simeq GL(r,\BH),\quad K_L\simeq Sp(r), \\ 
r=r,\quad n=r(2r-1),\quad d=4,\quad p=2(2r-1). 
\end{gather*}
We want to calculate the norm $\Vert\cdot\Vert_{\lambda,\tau}$ of $\cO(D,V)$ in the case 
$V=V_{(k,0,\ldots,0)}^\vee\simeq S^k(\BC^r)^\vee$, or 
$V=V_{\left(\frac{k}{2},\ldots,\frac{k}{2},-\frac{k}{2}\right)}^\vee\simeq S^k(\BC^r)\otimes\det^{-k/2}$ ($k=0,1,2\ldots$) 
(the latter is not defined as the representation of $U(2r)$ if $k$ is odd, so in this case we consider 
the double covering group $K={\widetilde{U}}^2(r)\subset G=\widetilde{SO^*}^2(4r)\subset Spin(4r,\BC)$). 
These $V$ have the restricted lowest weight $-\left.\frac{k}{2}\gamma_1\right|_{\fa_\fl}$ and 
$-\left.\frac{k}{2}(\gamma_1+\cdots+\gamma_{r-1})\right|_{\fa_\fl}$ respectively. 
Also, these $V$ remain irreducible even if restricted to $K_L=Sp(r)$, i.e. satisfy assumption (A1') of corollary \ref{tubecor}. 

First, we deal with $V=V_{(k,0,\ldots,0)}^\vee$ case. Then $\Vert\cdot\Vert_{\lambda,\tau_{(k,0,\ldots,0)}^\vee}^2$ converges 
if $\Re\lambda>4r-3$, and the normalizing constant $c_\lambda$ is given by 
\[ c_\lambda=\frac{\Gamma_\Omega(\lambda+(k,0,\ldots,0))}
{\Gamma_\Omega(\lambda+(k,0,\ldots,0)-(2r-1))}
=(\lambda+k)_{2r-1}\prod_{j=2}^r(\lambda-2(j-1)-(2r-1))_{2r-1}. \]
To begin with, we compute the $K$-type decomposition of $\cO(D,V)_K=\cP(\fp^+)\otimes V_{(k,0,\ldots,0)}^\vee$. 
To do this, we quote the following lemma. 
\begin{lemma}[{\cite[$\mathsection$79, Example 4]{Z}}]
\[ V_\bm^\vee\otimes V_{(k,0,\ldots,0)}^\vee
=\bigoplus_{\substack{\bk\in(\BZ_{\ge 0})^{2r},\;|\bk|=k\\ 0\le k_j\le m_{j-1}-m_j}}V_{\bm+\bk}^\vee. \]
\end{lemma}
Using this and Theorem \ref{HKS}, we get 
\begin{align*}
\cP(\fp^+)\otimes V_{(k,0,\ldots,0)}^\vee
&=\bigoplus_{\bm\in\BZ^r_{++}}V_{(m_1,m_1,m_2,m_2,\ldots,m_r,m_r)}^\vee\otimes V_{(k,0,\ldots,0)}^\vee\\
&=\bigoplus_{\bm\in\BZ^r_{++}}\bigoplus_{\substack{\bk\in(\BZ_{\ge 0})^r,\;|\bk|=k\\ 0\le k_j\le m_{j-1}-m_j}}
V_{(m_1+k_1,m_1,m_2+k_2,m_2,\ldots,m_r+k_r,m_r)}^\vee. 
\end{align*}
Next, for each $K$-type $V_{(m_1+k_1,m_1,\ldots,m_r+k_r,m_r)}^\vee$, we compute the tensor product with 
$\overline{V_{(k,0,\ldots,0)}^\vee}\simeq V_{(k,0,\ldots,0)}^\vee$. 
\begin{align*}
&V_{(m_1+k_1,m_1,m_2+k_2,m_2,\ldots,m_r+k_r,m_r)}^\vee\otimes V_{(k,0,\ldots,0)}^\vee\\
=&\bigoplus_{\substack{\bl\in(\BZ_{\ge 0})^{2r},\;|\bl|=k\\ 0\le l_{2j-1}\le m_{j-1}-m_j-k_j\\ 0\le l_{2j}\le k_j}}
V_{(m_1+k_1+l_1,m_1+l_2,m_2+k_2+l_3,m_2+l_4,\ldots,m_r+k_r+l_{2r-1},m_r+l_{2r})}^\vee. 
\end{align*}
By Theorem \ref{sph}, $V_{(m_1+k_1+l_1,m_1+l_2,\ldots,m_r+k_r+l_{2r-1},m_r+l_{2r})}^\vee$ is $K_L$-spherical 
if and only if the $(2j-1)$-th component of its lowest weight is equal to the $2j$-th component for each $j$, 
that is, $l_{2j-1}=0$ and $l_{2j}=k_j$. Thus, the only $K_L$-spherical submodule in 
$V_{(m_1+k_1,m_1,\ldots,m_r+k_r,m_r)}^\vee\otimes V_{(k,0,\ldots,0)}^\vee$ is 
$V_{(m_1+k_1,m_1+k_1,\ldots,m_r+k_r,m_r+k_r)}^\vee$, 
and $V_{(m_1+k_1,m_1,\ldots,m_r+k_r,m_r)}^\vee$ satisfies the assumption (A2') of Corollary \ref{tubecor} with $\bn=\bm+\bk$. 
Therefore by Corollary \ref{tubecor}, for $f\in V_{(m_1+k_1,m_1,\ldots,m_r+k_r,m_r)}^\vee$ we have 
\[ \frac{\Vert f\Vert_{\lambda,\tau_{(k,0,\ldots,0)}^\vee}^2}{\Vert f\Vert_{F,\tau_{(k,0,\ldots,0)}^\vee}^2}
=\frac{(\lambda)_{(k,0,\ldots,0)}}{(\lambda)_{\bm+\bk}}
=\frac{(\lambda)_k}{\prod_{j=1}^r(\lambda-2(j-1))_{m_j+k_j}}. \]

Second, we deal with $V=V_{\left(\frac{k}{2},\ldots,\frac{k}{2},-\frac{k}{2}\right)}^\vee$ case. 
Then $\Vert\cdot\Vert_{\lambda,\tau_{(k,0,\ldots,0)}^\vee}^2$ converges 
if $\Re\lambda>4r-3$, and the normalizing constant $c_\lambda$ is given by 
\begin{align*}
c_\lambda=&\frac{\Gamma_\Omega(\lambda+(k,\ldots,k,0))}{\Gamma_\Omega(\lambda+(k,\ldots,k,0)-(2r-1))}\\
=&\prod_{j=1}^{r-1}(\lambda-2(j-1)+k-(2r-1))_{2r-1}(\lambda-2(r-1)-(2r-1))_{2r-1}. 
\end{align*}
Similar to the previous arguments, $K$-type decomposition of 
$\cO(D,V)_K=\cP(\fp^+)\otimes V_{\left(\frac{k}{2},\ldots,\frac{k}{2},-\frac{k}{2}\right)}^\vee$ is given by 
\begin{align*}
\cP(\fp^+)\otimes V_{\left(\frac{k}{2},\ldots,\frac{k}{2},-\frac{k}{2}\right)}^\vee
&=\bigoplus_{\bm\in\BZ^r_{++}}V_{(m_1,m_1,m_2,m_2,\ldots,m_r,m_r)}^\vee\otimes V_{(0,\ldots,0,-k)}^\vee
\otimes V_{\left(\frac{k}{2},\ldots,\frac{k}{2}\right)}^\vee\\
&=\bigoplus_{\bm\in\BZ^r_{++}}\bigoplus_{\substack{\bk\in(\BZ_{\ge 0})^r,\;|\bk|=k\\ 0\le k_j\le m_j-m_{j+1}}}
V_{(m_1,m_1-k_1,m_2,m_2-k_2,\ldots,m_r,m_r-k_r)+\left(\frac{k}{2},\ldots,\frac{k}{2}\right)}^\vee, 
\end{align*}
and for each $K$-type, we can show that the only $K_L$-spherical submodule in 
\[ V_{(m_1,m_1-k_1,\ldots,m_r,m_r-k_r)+\left(\frac{k}{2},\ldots,\frac{k}{2}\right)}^\vee\otimes 
\overline{V_{\left(\frac{k}{2},\ldots,\frac{k}{2},-\frac{k}{2}\right)}^\vee} \]
is $V_{(m_1-k_1,m_1-k_1,\ldots,m_r-k_r,m_r-k_r)+(k,\ldots,k)}^\vee$. 
Thus $V_{(m_1,m_1-k_1,\ldots,m_r,m_r-k_r)+\left(\frac{k}{2},\ldots,\frac{k}{2}\right)}^\vee$ satisfies the assumption (A2') 
of Corollary \ref{tubecor} with $\bn=\bm-\bk+(k,\ldots,k)$. 
Therefore by Corollary \ref{tubecor}, for $f\in V_{(m_1,m_1-k_1,\ldots,m_r,m_r-k_r)+\left(\frac{k}{2},\ldots,\frac{k}{2}\right)}^\vee$ we have 
\[ \frac{\Vert f\Vert_{\lambda,\tau_{(k/2,\ldots,k/2,-k/2)}^\vee}^2}{\Vert f\Vert_{F,\tau_{(k/2,\ldots,k/2,-k/2)}^\vee}^2}
=\frac{(\lambda)_{(k,\ldots,k,0)}}{(\lambda)_{\bm-\bk+k}}
=\frac{\prod_{j=1}^{r-1}(\lambda-2(j-1))_k}{\prod_{j=1}^r(\lambda-2(j-1))_{m_j-k_j+k}}. \]

We summarize this subsection. 
\begin{theorem}\label{sostareven}
When $G=SO^*(4r)$, and $(\tau,V)=(\tau_{(k,0,\ldots,0)}^\vee,V_{(k,0,\ldots,0)}^\vee)$, 
$\Vert\cdot\Vert_{\lambda,\tau}^2$ converges if $\Re\lambda>4r-3$, the normalizing constant $c_\lambda$ is given by 
\[ c_\lambda=(\lambda+k)_{2r-1}\prod_{j=2}^r(\lambda-2(j-1)-(2r-1))_{2r-1}, \]
the $K$-type decomposition of $\cO(D,V)_K$ is given by 
\[ \cP(\fp^+)\otimes V_{(k,0,\ldots,0)}^\vee
=\bigoplus_{\bm\in\BZ^r_{++}}\bigoplus_{\substack{\bk\in(\BZ_{\ge 0})^r,\;|\bk|=k\\ 0\le k_j\le m_{j-1}-m_j}} 
V_{(m_1+k_1,m_1,m_2+k_2,m_2,\ldots,m_r+k_r,m_r)}^\vee, \]
and for $f\in V_{(m_1+k_1,m_1,m_2+k_2,m_2,\ldots,m_r+k_r,m_r)}^\vee$, the ratio of norms is given by 
\[ \frac{\Vert f\Vert_{\lambda,\tau_{(k,0,\ldots,0)}^\vee}^2}{\Vert f\Vert_{F,\tau_{(k,0,\ldots,0)}^\vee}^2}
=\frac{(\lambda)_k}{\prod_{j=1}^r(\lambda-2(j-1))_{m_j+k_j}}
=\frac{1}{(\lambda+k)_{m_1+k_1-k}\prod_{j=2}^r(\lambda-2(j-1))_{m_j+k_j}}. \]

When $G=SO^*(4r)$, and $(\tau,V)=(\tau_{(k/2,\ldots,k/2,-k/2)}^\vee,V_{(k/2,\ldots,k/2,-k/2)}^\vee)$, 
$\Vert\cdot\Vert_{\lambda,\tau}^2$ converges if $\Re\lambda>4r-3$, the normalizing constant $c_\lambda$ is given by 
\[ c_\lambda=\prod_{j=1}^{r-1}(\lambda-2(j-1)+k-(2r-1))_{2r-1}(\lambda-2(r-1)-(2r-1))_{2r-1}, \]
the $K$-type decomposition of $\cO(D,V)_K$ is given by 
\[ \cP(\fp^+)\otimes V_{\left(\frac{k}{2},\ldots,\frac{k}{2},-\frac{k}{2}\right)}^\vee
=\bigoplus_{\bm\in\BZ^r_{++}}\bigoplus_{\substack{\bk\in(\BZ_{\ge 0})^r,\;|\bk|=k\\ 0\le k_j\le m_j-m_{j+1}}}
V_{(m_1,m_1-k_1,m_2,m_2-k_2,\ldots,m_r,m_r-k_r)+\left(\frac{k}{2},\ldots,\frac{k}{2}\right)}^\vee, \]
and for $f\in V_{(m_1,m_1-k_1,m_2,m_2-k_2,\ldots,m_r,m_r-k_r)+\left(\frac{k}{2},\ldots,\frac{k}{2}\right)}^\vee$, 
the ratio of norms is given by 
\begin{align*}
\frac{\Vert f\Vert_{\lambda,\tau_{(k/2,\ldots,k/2,-k/2)}^\vee}^2}{\Vert f\Vert_{F,\tau_{(k/2,\ldots,k/2,-k/2)}^\vee}^2}
&=\frac{\prod_{j=1}^{r-1}(\lambda-2(j-1))_k}{\prod_{j=1}^r(\lambda-2(j-1))_{m_j-k_j+k}}\\
&=\frac{1}{\prod_{j=1}^{r-1}(\lambda+k-2(j-1))_{m_j-k_j}(\lambda-2(r-1))_{m_r-k_r+k}}. 
\end{align*}
\end{theorem}

\subsection{$Spin_0(2,n)$}\label{sectso2n}
In this subsection we set $G=Spin_0(2,n)$, the identity component of the indefinite spin group. 
This is of tube type, and we have 
\begin{gather*}
K\simeq (Spin(2)\times Spin(n))/\{(1,1),(-1,-1)\},\qquad \fp^\pm\simeq \BC^n,\\
r=2,\quad n=n,\quad d=n-2,\quad p=n. 
\end{gather*}
Let $\pi:K^\BC=(Spin(2,\BC)\times Spin(n,\BC))/\{(1,1),(-1,-1)\}\to SO(2,\BC)\times SO(n,\BC)$ be the covering map. 
Then we have 
\begin{gather*}
\pi(L)\simeq SO_0(1,1)\times SO_0(1,n-1)\cup SO_-(1,1)\times SO_-(1,n-1), \\
\pi(K_L)\simeq \{+I_2\}\times SO(n-1)\cup \{-I_2\}\times O_-(n-1), 
\end{gather*}
where $SO_-(p,q),O_-(q)$ are the connected component of $SO(p,q),O(q)$ which does not contain the unit element. 
Each representation of $K^\BC$ is of the form $(\chi^{m_0}\boxtimes \tau_\bm^{[n]},\BC_{m_0}\otimes V_\bm^{[n]})$, 
and sometimes we abbreviate this to $(\tau_{(m_0;\bm)},V_{(m_0;\bm)})$. 

Now we want to calculate the norm $\Vert\cdot\Vert_{\lambda,\tau}$ of $\cO(D,V)$ in the case 
\[ (\tau,V)=\left\{\begin{array}{lll} 
(\chi^{-k}\boxtimes \tau_{(k,\ldots,k,\pm k)},\BC_{-k}\otimes V_{(k,\ldots,k,\pm k)})
&\left(k\in\frac{1}{2}\BZ_{\ge 0}\right)&(n:\text{even}),\\
(\chi^{-k}\boxtimes \tau_{(k,\ldots,k)},\BC_{-k}\otimes V_{(k,\ldots,k)})
&\left(k=0,\frac{1}{2}\right)&(n:\text{odd}). \end{array}\right. \]
These $(\tau,V)$ have the restricted lowest weight $-k\gamma_1$, 
and remain irreducible even if restricted to $K_L$, i.e. satisfy assumption (A1') of corollary \ref{tubecor}. 
Thus $\Vert\cdot\Vert_{\lambda,\tau}^2$ converges if $\Re\lambda>n-1$, and the normalizing constant $c_\lambda$ is given by 
\[ c_\lambda=\frac{\Gamma_\Omega(\lambda+(k,0))}
{\Gamma_\Omega\left(\lambda+(k,0)-\frac{n}{2}\right)}
=\frac{\Gamma\left(\lambda+k\right)\Gamma\left(\lambda-\frac{n-2}{2}\right)}
{\Gamma\left(\lambda+k-\frac{n}{2}\right)\Gamma\left(\lambda-(n-1)\right)}. \]
First we compute the $K$-type decomposition of $\cO(D,V)_K=\cP(\fp^+)\otimes V$. 
To do this, we use the following lemma, which comes from the ``multi-minuscule rule'' \cite[Corollary 2.16]{S}. 
\begin{lemma}\label{tensorso}
\begin{enumerate}
\item Let $m\in\BZ_{\ge 0}$ and $k\in\frac{1}{2}\BZ_{\ge 0}$. 
For two representations $V_{(m,0,\ldots,0)}$ and $V_{(k,\ldots,k,\pm k)}$ of $\mathfrak{so}(2s,\BC)$, 
\[ V_{(m,0,\ldots,0)}\otimes V_{(k,\ldots,k,\pm k)}=\bigoplus_{l=\max\{ -k,k-m\}}^kV_{(m+l,k,\ldots,k,\pm l)} \]
(double sign corresponds) holds. 
\item Let $m\in\BZ_{>0}$. 
For two representations $V_{(m,0,\ldots,0)}$ and $V_{\left(\frac{1}{2},\ldots,\frac{1}{2}\right)}$ of 
$\mathfrak{so}(2s+1,\BC)$, 
\[ V_{(m,0,\ldots,0)}\otimes V_{\left(\frac{1}{2},\ldots,\frac{1}{2}\right)}
=V_{\left(m+\frac{1}{2},\frac{1}{2},\ldots,\frac{1}{2}\right)}
\oplus V_{\left(m-\frac{1}{2},\frac{1}{2},\ldots,\frac{1}{2}\right)} \]
holds. 
\end{enumerate}
\end{lemma}
By Theorem \ref{HKS}, 
\[ \cP(\fp^+)=\bigoplus_{\bm\in\BZ^2_{++}}\BC_{-(m_1+m_2)}\boxtimes V_{(m_1-m_2,0,\ldots,0)} \]
holds, and combining with the above lemma, we have 
\[ \cP(\fp^+)\otimes \left(\BC_{-k}\boxtimes V_{(k,\ldots,k,\pm k)}\right)
=\bigoplus_{\bm\in\BZ^2_{++}}\bigoplus_{\substack{-k\le l\le k\\ m_1-m_2+l\ge k}} 
\BC_{-(m_1+m_2+k)}\boxtimes V_{(m_1-m_2+l,k,\ldots,k,\pm l)} \] 
for $n=2s$ even case, $k\in\frac{1}{2}\BZ_{\ge 0}$, and 
\[ \cP(\fp^+)\otimes \left(\BC_{-k}\boxtimes V_{(k,\ldots,k)}\right)
=\bigoplus_{\bm\in\BZ^2_{++}}\bigoplus_{\substack{-k\le l\le k\\ m_1-m_2+l\ge k}} 
\BC_{-(m_1+m_2+k)}\boxtimes V_{(m_1-m_2+l,k,\ldots,k,|l|)} \] 
for $n=2s+1$ odd case, $k=0,\frac{1}{2}$. 

Second, we seek $K_L$-spherical subspace in the tensor product of each $K$-type and $\bar{V}$. 
To begin with, we deal with $n=2s$ even, $V=V_{(-k;k,\ldots,k,k)}$ case. Suppose 
\[ V_{(-(n_1+n_2):n_1-n_2,0,\ldots,0)}\subset V_{(-(m_1+m_2+k);m_1-m_2+l,k,\ldots,k,l)}\otimes\overline{V_{(-k;k,\ldots,k)}}, \]
where $(n_1,n_2)\in\BZ^2_+$. This implies that $(-(n_1+n_2)+(m_1+m_2+k);(n_1-n_2)-(m_1-m_2+l),-k,\ldots,-k,-l)$ is 
a weight of $\overline{V_{(-k;k,\ldots,k)}}$. 
However, the weight of this form is only $(-k;l,-k,\ldots,-k,-l)$, since 
$\overline{V_{(-k;k,\ldots,k,k)}}$ has the lowest weight $(-k;-k,\ldots,-k,k)$, 
and root vectors $x_{\varepsilon_1-\varepsilon_s},x_{\varepsilon_1+\varepsilon_s}\in\mathfrak{so}(2s)$ commute with each other. 
Therefore we have 
\[ \left\{\begin{array}{l} (n_1+n_2)-(m_1+m_2+k)=k,\\ (n_1-n_2)-(m_1-m_2+l)=l. \end{array}\right. \qquad
\therefore \left\{\begin{array}{l} n_1=m_1+k+l,\\ n_2=m_2+k-l. \end{array}\right. \]
Thus all $K_L$-spherical irreducible submodule in $V_{(-(m_1+m_2+k);m_1-m_2+l,k,\ldots,k,l)}\otimes\overline{V_{(-k;k,\ldots,k)}}$ 
have the same lowest weight $-(n_1\gamma_1+n_2\gamma_2)$ with $(n_1,n_2)=(m_1+k+l,m_2+k-l)$, 
and all $K$-types satisfy the assumption (A2') of Corollary \ref{tubecor}. 
The same argument holds for $V=V_{(-k;k,\ldots,k,-k)}$ case, and also for $n$ odd case, noting that 
only $k=0,\frac{1}{2}$ is allowed, and $n_1,n_2\in\BZ$. 
Therefore by Corollary \ref{tubecor}, for $f\in V_{(-(m_1+m_2+k);m_1-m_2+l,k,\ldots,k,\pm l)}$ or 
$V_{(-(m_1+m_2+k);m_1-m_2+l,k,\ldots,k,|l|)}$, we have 
\[ \frac{\Vert f\Vert_{\lambda,\tau}^2}{\Vert f\Vert_{F,\tau}^2}
=\frac{(\lambda)_{(2k,0)}}{(\lambda)_{(m_1+k+l,m_2+k-l)}}
=\frac{(\lambda)_{2k}}{(\lambda)_{m_1+k+l}\left(\lambda-\frac{n-2}{2}\right)_{m_2+k-l}}. \]

We summarize this subsection. 
\begin{theorem}\label{spin2n}
When $G=Spin_0(2,n)$ and 
\[ (\tau,V)=\left\{\begin{array}{lll} 
(\chi^{-k}\boxtimes \tau_{(k,\ldots,k,\pm k)},\BC_{-k}\otimes V_{(k,\ldots,k,\pm k)})
&\left(k\in\frac{1}{2}\BZ_{\ge 0}\right)&(n:\text{even}),\\
(\chi^{-k}\boxtimes \tau_{(k,\ldots,k)},\BC_{-k}\otimes V_{(k,\ldots,k)})
&\left(k=0,\frac{1}{2}\right)&(n:\text{odd}), \end{array}\right. \]
$\Vert\cdot\Vert_{\lambda,\tau}^2$ converges if $\Re\lambda>n-1$, 
the normalizing constant $c_\lambda$ is given by 
\[ c_\lambda=\frac{\Gamma\left(\lambda+k\right)\Gamma\left(\lambda-\frac{n-2}{2}\right)}
{\Gamma\left(\lambda+k-\frac{n}{2}\right)\Gamma\left(\lambda-(n-1)\right)}, \]
the $K$-type decomposition of $\cO(D,V)_K$ is given by 
\begin{align*}
\cP(\fp^+)\otimes V=\begin{cases}
\ds \bigoplus_{\bm\in\BZ^2_{++}}\bigoplus_{\substack{-k\le l\le k\\ m_1-m_2+l\ge k}} 
\BC_{-(m_1+m_2+k)}\boxtimes V_{(m_1-m_2+l,k,\ldots,k,\pm l)} &(n:\text{even}), \\ 
\ds \bigoplus_{\bm\in\BZ^2_{++}}\bigoplus_{\substack{-k\le l\le k\\ m_1-m_2+l\ge k}} 
\BC_{-(m_1+m_2+k)}\boxtimes V_{(m_1-m_2+l,k,\ldots,k,|l|)} &(n:\text{odd}), \end{cases}
\end{align*}
and for $f\in \BC_{-(m_1+m_2+k)}\boxtimes V_{(m_1-m_2+l,k,\ldots,k,\pm l)}$ or 
$\BC_{-(m_1+m_2+k)}\boxtimes V_{(m_1-m_2+l,k,\ldots,k,|l|)}$, the ratio of norms is given by 
\[ \frac{\Vert f\Vert_{\lambda,\tau}^2}{\Vert f\Vert_{F,\tau}^2}
=\frac{(\lambda)_{2k}}{(\lambda)_{m_1+k+l}\left(\lambda-\frac{n-2}{2}\right)_{m_2+k-l}}
=\frac{1}{(\lambda+2k)_{m_1-k+l}\left(\lambda-\frac{n-2}{2}\right)_{m_2+k-l}}. \]
\end{theorem}

\section{Norm computation: Non-tube type case}\label{sectnontube}
When $G$ is of non-tube type, we cannot compute the norm by just using Theorem \ref{keythm}, 
because it is difficult to determine the constants $a_{ij}$ in Theorem \ref{keythm}. 
Thus we have to use other informations to compute the norm. 
In this section we compute the norm in the case
\begin{itemize}
\item $(G,V)=(SU(q,s),\BC\boxtimes V')$ ($q<s$), by direct computation, 
\item $(G,V)=(SO^*(4r+2),S^k(\BC^{2r+1})^\vee)$, by using the embedding $SO^*(4r+2)\subset SO^*(4r+4)$, 
\item $(G,V)=(SO^*(4r+2),S^k(\BC^{2r+1})\otimes\det^{-k/2})$, by combining Theorem \ref{keythm} and the embedding $SU(1,2r)\subset SO^*(4r+2)$. 
\end{itemize}
Also, for $G=E_{6(-14)}$, we try to compute the norm as best we can, by using Theorem \ref{keythm}.

\subsection{Explicit realization of $G$}
Before starting the computation, we fix the realization of $G=SU(q,s)$, $SO^*(2s)$. 
We realize $SU(q,s)$, $SO^*(2s)$ as 
\begin{align}
SU(q,s):=&\left\{g\in SL(q+s,\BC):g\begin{pmatrix}I_q&0\\0&-I_s\end{pmatrix}g^*=\begin{pmatrix}I_q&0\\0&-I_s\end{pmatrix}\right\}, 
\label{suqsrealize} \\
SO^*(2s):=&\left\{g\in GL(2s,\BC):g\begin{pmatrix}0&I_s\\I_s&0\end{pmatrix}{}^t\hspace{-1pt}g=\begin{pmatrix}0&I_s\\I_s&0\end{pmatrix},\,
g\begin{pmatrix}0&I_s\\-I_s&0\end{pmatrix}=\begin{pmatrix}0&I_s\\-I_s&0\end{pmatrix}\bar{g}\right\}, \label{sostarrealize}
\end{align}
and realize $K^\BC$, $\fp^\pm$ as 
\begin{align*}
K^\BC:=&\left\{\begin{pmatrix}a&0\\0&d\end{pmatrix}:
\begin{array}{ll}(a,d)\in S(GL(q,\BC)\times GL(s,\BC))& (G=SU(q,s))\\
a\in GL(s,\BC),\; d={}^t\hspace{-1pt}a^{-1}& (G=SO^*(2s))\end{array}\right\},\\
\fp^+:=&\left\{\begin{pmatrix}0&b\\0&0\end{pmatrix}:
\begin{array}{ll}b\in M(q,s;\BC)& (G=SU(q,s))\\
b\in \mathrm{Skew}(s,\BC)& (G=SO^*(2s))\end{array}\right\},\\
\fp^-:=&\left\{\begin{pmatrix}0&0\\c&0\end{pmatrix}:
\begin{array}{ll}c\in M(s,q;\BC)& (G=SU(q,s))\\
c\in \mathrm{Skew}(s,\BC)& (G=SO^*(2s))\end{array}\right\}.
\end{align*}
Then under the identification $\fp^+\simeq M(q,s;\BC)$ or $\mathrm{Skew}(2s,\BC)$ by $\begin{pmatrix}0&b\\0&0\end{pmatrix}\mapsto b$, 
we have 
\begin{align}
D=&\{w\in M(q,s;\BC):I_q-ww^*\text{ is positive definite.}\} &(G=SU(q,s)), \label{suqsD}\\ 
D=&\{w\in\mathrm{Skew}(s,\BC):I_s-ww^*\text{ is positive definite.}\} &(G=SO^*(2s)). \label{sostarD}
\end{align}
For a representation $(\tau_1\boxtimes\tau_2,V_1\otimes V_2)$ of 
$K^\BC=S(GL(q,\BC)\times GL(s,\BC))$, the universal covering group $\widetilde{SU}(q,s)$ acts on $\cO(D,V_1\otimes V_2)$ by 
\begin{multline}\label{suqsrepn} 
\tau_\lambda\left(\begin{pmatrix}a&b\\c&d\end{pmatrix}^{-1}\right)f(w)
=\det(cw+d)^{-\lambda}\left(\tau_1\left(a^*+wb^*\right)\boxtimes\tau_2\left((cw+d)^{-1}\right)\right)\\
\times f\left((aw+b)(cw+d)^{-1}\right), 
\end{multline}
and for a representation $(\tau,V)$ of $K^\BC=GL(s,\BC)$, the universal covering group $\widetilde{SO^*}(2s)$ acts on $\cO(D,V)$ by 
\begin{equation}\label{sostarrepn}
\tau_\lambda\left(\begin{pmatrix}a&b\\c&d\end{pmatrix}^{-1}\right)f(w)
=\det(cw+d)^{-\lambda/2}\tau\left({}^t\hspace{-1pt}(cw+d)\right)f\left((aw+b)(cw+d)^{-1}\right), 
\end{equation}
We note that we have the identities, for $w\in M(q,s;\BC)$ and $\begin{pmatrix}a&b\\c&d\end{pmatrix}\in U(q,s)$, 
\[ \det(I_q-ww^*)=\det(I_s-w^*w), \qquad \det(a^*+wb^*)=\det\begin{pmatrix}a&b\\c&d\end{pmatrix}^{-1}\det(cw+d). \]
Therefore, on $SU(q,s)$, $\det(a^*+wb^*)=\det(cw+d)$ holds. 
We also note that $\det(cw+d)^{-\lambda}$ is not well-defined on $G$ for general $\lambda\in\BC$, 
but is well-defined on the universal covering group $\tilde{G}$. 
These representations preserve the inner product 
\begin{align}
\langle f,g\rangle_{\lambda,\tau}=&\frac{c_\lambda}{\pi^{qs}}\int_D
\left(\left(\tau_1\left((I_q-ww^*)^{-1}\right)\boxtimes\tau_2\left(I_s-w^*w\right)\right)f(w),g(w)\right)_{\tau_1\boxtimes\tau_2} \notag \\
&\hspace{198pt}\times\det(I_q-ww^*)^{\lambda-(q+s)}dw, \label{suqsinner}\\
\langle f,g\rangle_{\lambda,\tau}=&\frac{c_\lambda}{\pi^{s(s-1)/2}}\int_D\left(\tau\left((I_s-ww^*)^{-1}\right)f(w),g(w)\right)_\tau 
\det(I_s-ww^*)^{\frac{1}{2}(\lambda-2(s-1))}dw. \label{sostarinner}
\end{align}
respectively. Let $\fh\subset\fg$ be the subspace which consists of all diagonal matrices, and define the linear form $\varepsilon_i$ 
on $\fh^\BC$ by $\varepsilon_i(E_{jj})=\delta_{ij}$. We define the positive system $\Delta_+(\fg^\BC,\fh^\BC)$ as in Section \ref{explicitroots}.

\subsection{$SU(q,s)$}
In this subsection we set $G=SU(q,s)$, with $q<s$, which is realized explicitly as (\ref{suqsrealize}). 
Then we have 
\begin{gather*}
K\simeq S(U(q)\times U(s)),\quad \fp^\pm\simeq M(q,s;\BC), \quad G_T\simeq SU(q,q),\quad K_T\simeq S(U(q)\times U(q)), \\ 
L\simeq \{l\in GL(q,\BC):\det l\in \BR^\times\}, \quad 
K_L\simeq \{k\in U(q):\det k=\pm 1\}, \\
r=q,\quad n=qs,\quad d=2,\quad p=q+s. 
\end{gather*}
We set $(\tau,V)=(\tau_\bzero^{(q)\vee}\boxtimes\tau_\bk^{(s)},V_\bzero^{(q)\vee}\otimes V_\bk^{(s)})
=(\bone^{(q)}\boxtimes\tau_\bk^{(s)},\BC\otimes V_\bk^{(s)})$ ($\bk\in\BZ_{++}^s$). 
In this case, the inner product is given by 
\[ \langle f,g\rangle_{\lambda,\bone^{(q)}\boxtimes\tau_\bk^{(s)}}=\frac{c_\lambda}{\pi^{qs}}\int_D
\left(\left(\tau_\bk^{(s)}\left(I_s-w^*w\right)\right)f(w),g(w)\right)_{\tau_\bk^{(s)}}
\det(I_s-w^*w)^{\lambda-(q+s)}dw. \]
The goal of this subsection is to prove the following theorem. 
\begin{theorem}\label{suqsnontube}
When $G=SU(q,s)$ $(q<s)$ and $(\tau,V)=(\bone^{(q)}\boxtimes\tau_\bk^{(s)},\BC\otimes V_\bk^{(s)})$ $(\bk\in\BZ_{++}^s)$, 
$\Vert\cdot\Vert_{\lambda,\tau}^2$ converges if $\Re\lambda+k_s>q+s-1$, 
the normalizing constant $c_\lambda$ is given by 
\[ c_\lambda=\prod_{j=1}^s(\lambda-(j-1)+k_j-q)_q, \]
the $K$-type decomposition of $\cO(D,V)_K$ is given by 
\[ \cP(\fp^+)\otimes \left(\BC\boxtimes V_\bk^{(s)}\right)
=\bigoplus_{\bm\in\BZ^q_{++}}\bigoplus_{\bn\in\bm+\mathrm{wt}(\bk)}c^\bn_{\bk,\bm}V_\bm^{(q)\vee}\boxtimes V_\bn^{(s)}, \]
and for $f\in V_\bm^{(q)\vee}\boxtimes V_\bn^{(s)}$, the ratio of norms is given by 
\[ \frac{\Vert f\Vert_{\lambda,\bone^{(q)}\boxtimes\tau_\bk^{(s)}}^2}{\Vert f\Vert_{F,\bone^{(q)}\boxtimes\tau_\bk^{(s)}}^2}
=\frac{\prod_{j=1}^s(\lambda-(j-1))_{k_j}}{\prod_{j=1}^s(\lambda-(j-1))_{n_j}}
=\frac{1}{\prod_{j=1}^s(\lambda-(j-1)+k_j)_{n_j-k_j}}. \]
\end{theorem}

Before beginning the proof, we prepare some more notations. For $k\in\BN$, $\bm\in\BC^k$ and for $x\in M(k,\BC)$, we write 
\[ \Delta_\bm(x):=\prod_{l=1}^{k-1}\det\left((x_{ij})_{1\le i,j\le l}\right)^{m_l-m_{l+1}}\det(x)^{m_k}. \]
For $k\in\BN$, let $Q_k\subset GL(k,\BC)$ be the set of upper triangular matrices with positive diagonal entries. 
Then for $l_1,l_2\in Q_k$, $\bm\in\BC^k$, $\Delta_\bm(l_1)\Delta_\bm(l_2)=\Delta_\bm({}^t\hspace{-1pt}l_1l_2)$ holds, and 
for $l_1\in Q_k$, $l_2\in M(k,l;\BC)$, $l_3\in Q_l$ and $\bm\in\BC^k$, $\bn\in\BC^l$, 
$\Delta_\bm(l_1)\Delta_\bn(l_3)=\Delta_{(\bm,\bn)}\begin{pmatrix}l_1&l_2\\0&l_3\end{pmatrix}$ holds. Also we set 
\begin{align*}
(\fp^+_\rT)^\bot&:=M(q,s-q;\BC),\\
\Omega&:=\{x\in\mathrm{Herm}(q,\BC):x\text{ is positive definite.}\},\\
\tilde{\Omega}&:=\{x\in\mathrm{Herm}(s,\BC):x\text{ is positive definite.}\}. 
\end{align*}

Now we start the proof. 
To begin with, we compute the $K$-type decomposition of $\cO(D,V)_K=\cP(\fp^+)\otimes \left(\BC\boxtimes V_\bk^{(s)}\right)$. 
\begin{align*}
\cP(\fp^+)\otimes \left(\BC\boxtimes V_\bk^{(s)}\right)
&=\bigoplus_{\bm\in\BZ^q_{++}}\left(V_\bm^{(q)\vee}\boxtimes V_\bm^{(s)}\right)\otimes \left(\BC\boxtimes V_\bk^{(s)}\right)\\
&=\bigoplus_{\bm\in\BZ^q_{++}}\bigoplus_{\bn\in\bm+\mathrm{wt}(\bk)}c^\bn_{\bk,\bm}V_\bm^{(q)\vee}\boxtimes V_\bn^{(s)}. 
\end{align*}
where $V_\bm^{(s)}$ is the abbreviation of $V_{(m_1,\ldots,m_q,0,\ldots,0)}^{(s)}$, 
$\mathrm{wt}(\bk)$ is the set of all weights in the $GL(s,\BC)$-module $V_\bk^{(s)}$, 
and $c^\bn_{\bk,\bm}$ are some non-negative integers. 
We note that, for $\bn\in\BZ^s_{++}$, there exists $\bm\in\BZ^q_{++}$ such that $c^\bn_{\bk,\bm}\ne 0$ if and only if 
\[ n_j\ge k_j\; (1\le j\le q) \quad \text{and} \quad k_{j-q}\le n_j\le k_j\; (j\ge q+1), \]
which can be proved by using Littlewood-Richardson rule. 

For each $K$-type $V_\bm^{(q)\vee}\boxtimes V_\bn^{(s)}$, 
let $K_{\bm,\bn}(z,w)\in\cP(\fp^+\times\overline{\fp^+},\mathrm{End}(V_\bk^{(s)}))$ be the reproducing kernel of 
the $K_\rT^\BC$-submodule $V_\bm^{(q)\vee}\boxtimes V_{\bn'}^{(q)}\subset V_\bm^{(q)\vee}\boxtimes V_\bn^{(s)}$, 
where $\bn':=(n_1,\ldots,n_q)\in\BZ_{++}^q$. 
Then since $V_\bm^{(q)\vee}\boxtimes V_{\bn'}^{(q)}\subset V_\bm^{(q)\vee}\boxtimes V_\bn^{(s)}$ is the lowest submodule, we have 
\begin{align*}
\tau_\bk^{(s)}\!\begin{pmatrix}l_2&l_3\\0&l_4\end{pmatrix}\!
K_{\bm,\bn}\!\left(l_1z\begin{pmatrix}l_2&l_3\\0&l_4\end{pmatrix}\!,l_1^{*-1}w\begin{pmatrix}l_2^{*-1}&l_5\\0&l_6\end{pmatrix}\right)\!
\tau_\bk^{(s)}\!\begin{pmatrix}l_2^{-1}&0\\l_5^*&l_6^*\end{pmatrix}=\Delta_{\bn''}(l_6^*l_4)K_{\bm,\bn}(z,w)\\ 
\quad(z,w\in M(q,s;\BC),\; l_1,l_2\in GL(q,\BC),\; l_3,l_5\in M(q,s-q;\BC),\; l_4,l_6\in Q_{s-q}), 
\end{align*}
where $\bn'':=(n_{s-q+1},\ldots,n_s)$. Using this $K_{\bm,\bn}(z,w)$, we can rewrite the ratio of norms. 
That is, for $f\in V_\bm^{(q)\vee}\boxtimes V_\bn^{(s)}$, the ratio of norms 
$\Vert f\Vert_{\lambda,\bone^{(q)}\boxtimes\tau_\bk^{(s)}}^2/\Vert f\Vert_{F,\bone^{(q)}\boxtimes\tau_\bk^{(s)}}^2$ is equal to 
\[ R_{\bm,\bn}(\lambda):=\frac{\ds c_\lambda\int_D\Tr_{V_\bk^{(s)}}
\left(\tau_{\bk}^{(s)}(I_s-w^*w)K_{\bm,\bn}(w,w)\right)\det(I_s-w^*w)^{\lambda-(q+s)}dw}
{\ds \int_{\fp^+}\Tr_{V_\bk^{(s)}}(K_{\bm,\bn}(w,w))e^{-\tr(w^*w)}dw}. \]
Now similarly to Lemma \ref{intformula}, 
for any non-negative measurable function $f$ on $M(q,s;\BC)$, we have 
\[ \frac{1}{\pi^{qs}}\int_{\fp^+}f(w)dw
=\frac{1}{\Gamma_\Omega(q)}\int_{\substack{x\in\Omega,y\in (\fp^+_\rT)^\bot\\ k_1,k_2\in U(q)}}
f((k_1x^{\frac{1}{2}}k_2,k_1y))dk_1dk_2dxdy. \]
Using this and the $K_\rT$-invariance of $K_{\bm,\bn}(z,w)$ 
\begin{align*}
&K_{\bm,\bn}((k_1x^{\frac{1}{2}}k_2,k_1y),(k_1x^{\frac{1}{2}}k_2,k_1y))\\
=&\tau_\bk^{(s)}\begin{pmatrix}k_2^{-1}&0\\0&I_{s-q}\end{pmatrix}K_{\bm,\bn}((x^{\frac{1}{2}},y),(x^{\frac{1}{2}},y))\tau_\bk^{(s)}
\begin{pmatrix}k_2&0\\0&I_{s-q}\end{pmatrix} \\
&\hspace{100pt}(x\in\Omega,\;y\in(\fp^+_\rT)^\bot,\; k_1,k_2\in U(q)),
\end{align*}
we have 
\begin{align*}
R_{\bm,\bn}(\lambda)
=\frac{\begin{array}{r}\ds c_\lambda
\int_{\substack{x\in\Omega,y\in (\fp^+_\rT)^\bot\\ (x^{1/2},y)\in D}}
\Tr_{V_\bk^{(s)}}\!\left(\tau_{\bk}^{(s)}\!\left(I_s-\begin{pmatrix}x&x^{1/2}y\\y^*x^{1/2}&y^*y\end{pmatrix}\right)
\!K_{\bm,\bn}((x^{\frac{1}{2}},y),(x^{\frac{1}{2}},y))\right)\\
\ds \times\det\left(I_s-\begin{pmatrix}x&x^{1/2}y\\y^*x^{1/2}&y^*y\end{pmatrix}\right)^{\lambda-(q+s)}dxdy\end{array}}
{\ds \int_{x\in\Omega,y\in (\fp^+_\rT)^\bot}
\Tr_{V_\bk^{(s)}}(K_{\bm,\bn}((x^{\frac{1}{2}},y),(x^{\frac{1}{2}},y)))
e^{-\tr\left(\begin{smallmatrix}x&x^{1/2}y\\y^*x^{1/2}&y^*y\end{smallmatrix}\right)}dxdy}. 
\end{align*}
$K_{\bm,\bn}((x^{\frac{1}{2}},y),(x^{\frac{1}{2}},y))$ is transformed as below. 
\begin{align*}
K_{\bm,\bn}((x^{\frac{1}{2}},y),(x^{\frac{1}{2}},y))
=&K_{\bm,\bn}\left(x^{-\frac{1}{2}}(x,0)\begin{pmatrix}I_q&x^{-1/2}y\\0&I_{s-q}\end{pmatrix},
x^{\frac{1}{2}}(I_q,0)\begin{pmatrix}I_q&x^{-1/2}y\\0&I_{s-q}\end{pmatrix}\right)\\
=&\tau_\bk^{(s)}\begin{pmatrix}I_q&-x^{-1/2}y\\0&I_{s-q}\end{pmatrix}K_{\bm,\bn}((x,0),(I_q,0))
\tau_\bk^{(s)}\begin{pmatrix}I_q&0\\-y^*x^{-1/2}&I_{s-q}\end{pmatrix}. 
\end{align*}
Then $K_{\bm,\bn}((\cdot,0),(I_q,0))$ is $K_L=\mathrm{diag}(\{\pm 1\}\times SU(q))$-invariant under the representation 
$\tilde{\tau}$ of $K_\rT^\BC$ on $\cP(\fp^+_\rT,\mathrm{End}(V_\bk^{(s)}))=\cP(M(q,s),\mathrm{End}(V_\bk^{(s)}))$, 
\[ (\tilde{\tau}(l_1,l_2))F(x)
:=\tau_\bk^{(s)}\begin{pmatrix}l_2&0\\0&I_{s-q}\end{pmatrix}F(l_1^{-1}xl_2)\tau_\bk^{(s)}
\begin{pmatrix}l_1^{-1}&0\\0&I_{s-q}\end{pmatrix}. \]
That is, $K_{\bm,\bn}((\cdot,0),(I_q,0))\in\left(\left(V_\bm^{(q)\vee}\boxtimes V_{\bn'}^{(q)}\right)\otimes
\left(\left.V_\bk^{(s)\vee}\right|_{U(q)}\boxtimes\BC\right)\right)^{K_L}=V_{\bn'}^{(q)\vee}\boxtimes V_{\bn'}^{(q)}$. 
Therefore there exists an $F_{\bm,\bn}(x)\in\cP(\fp^+_\rT,\mathrm{End}(V_\bk^{(s)}))$ such that 
\begin{gather*}
\int_{U(q)}\tau_\bk^{(s)}\begin{pmatrix}k&0\\0&I_{s-q}\end{pmatrix}F_{\bm,\bn}(k^{-1}xk)
\tau_\bk^{(s)}\begin{pmatrix}k^{-1}&0\\0&I_{s-q}\end{pmatrix}dk
=K_{\bm,\bn}((x,0),(I_q,0)),\\
\tau_\bk^{(s)}\begin{pmatrix}l_2&0\\0&l_4\end{pmatrix}F_{\bm,\bn}({}^t\hspace{-1pt}l_1xl_2)\tau_\bk^{(s)}
\begin{pmatrix}{}^t\hspace{-1pt}l_1&0\\0&{}^t\hspace{-1pt}l_3\end{pmatrix}
=\Delta_{\bn'}({}^t\hspace{-1pt}l_1l_2)\Delta_{\bn''}({}^t\hspace{-1pt}l_3l_4)F_{\bm,\bn}(x)\\
\hspace{190pt}(x\in\fp^+_\rT,\; l_1,l_2\in Q_q,\; l_3,l_4\in Q_{s-q}). 
\end{gather*}
We define 
\[ \tilde{F}_{\bm,\bn}(x,y):=\tau_\bk^{(s)}\begin{pmatrix}I_q&-x^{-1/2}y\\0&I_{s-q}\end{pmatrix}F_{\bm,\bn}(x)
\tau_\bk^{(s)}\begin{pmatrix}I_q&0\\-y^*x^{-1/2}&I_{s-q}\end{pmatrix}. \]
Then we have 
\begin{align*}
R_{\bm,\bn}(\lambda)
=\frac{\begin{array}{r}\ds c_\lambda
\int_{\substack{x\in\Omega,y\in (\fp^+_\rT)^\bot\\ (x^{1/2},y)\in D}}
\Tr_{V_\bk^{(s)}}\left(\tau_{\bk}^{(s)}\left(I_s-\begin{pmatrix}x&x^{1/2}y\\y^*x^{1/2}&y^*y\end{pmatrix}\right)
\tilde{F}_{\bm,\bn}(x,y)\right)\hspace{40pt}\\
\ds \times\det\left(I_s-\begin{pmatrix}x&x^{1/2}y\\y^*x^{1/2}&y^*y\end{pmatrix}\right)^{\lambda-(q+s)}dxdy\end{array}}
{\ds \int_{x\in\Omega,y\in (\fp^+_\rT)^\bot}\Tr_{V_\bk^{(s)}}(\tilde{F}_{\bm,\bn}(x,y))
e^{-\tr\left(\begin{smallmatrix}x&x^{1/2}y\\y^*x^{1/2}&y^*y\end{smallmatrix}\right)}dxdy}. 
\end{align*}
We set 
\begin{align*}
B_{\bm,\bn}(\lambda):=&
\int_{\substack{x\in\Omega,y\in (\fp^+_\rT)^\bot\\ (x^{1/2},y)\in D}}
\Tr_{V_\bk^{(s)}}\left(\tau_{\bk}^{(s)}\left(I_s-\begin{pmatrix}x&x^{1/2}y\\y^*x^{1/2}&y^*y\end{pmatrix}\right)
\tilde{F}_{\bm,\bn}(x,y)\right)\\
&\hspace{150pt}\times\det\left(I_s-\begin{pmatrix}x&x^{1/2}y\\y^*x^{1/2}&y^*y\end{pmatrix}\right)^{\lambda-(q+s)}dxdy, \\
\Gamma_{\bm,\bn}:=&\int_{x\in\Omega,y\in (\fp^+_\rT)^\bot}\Tr_{V_\bk^{(s)}}(\tilde{F}_{\bm,\bn}(x,y))
e^{-\tr\left(\begin{smallmatrix}x&x^{1/2}y\\y^*x^{1/2}&y^*y\end{smallmatrix}\right)}dxdy, 
\end{align*}
so that $R_{\bm,\bn}(\lambda)=c_\lambda B_{\bm,\bn}(\lambda)/\Gamma_{\bm,\bn}$. 
We want to compute $B_{\bm,\bn}(\lambda)$ explicitly. To do this, similarly to (\ref{Ilambda}), 
for $z\in\tilde{\Omega}$ we define 
\begin{multline*}
J(z):=\int_{E{(z)}}
\Tr_{V_\bk^{(s)}}\left(\tau_{\bk}^{(s)}\left(z-\begin{pmatrix}x'&(x')^{1/2}y'\\(y')^*(x')^{1/2}&(y')^*y'\end{pmatrix}\right)
\tilde{F}_{\bm,\bn}(x',y')\right)\\
\times\det\left(z-\begin{pmatrix}x'&(x')^{1/2}y'\\(y')^*(x')^{1/2}&(y')^*y'\end{pmatrix}\right)^{\lambda-(q+s)}dx'dy', 
\end{multline*}
where 
\[ E(z):=\left\{ (x,y)\in\Omega\times(\fp^+_\rT)^\bot:
z-\begin{pmatrix}x&x^{1/2}y\\ y^*x^{1/2}&y^*y\end{pmatrix} \text{ is positive definite.} \right\}, \]
so that $E(I_s)$ coincides with the domain of integration of $B_{\bm,\bn}(\lambda)$, and $J(I_s)=B_{\bm,\bn}(\lambda)$ holds. 
To compute $E(z)$, we take $l_1\in Q_q$, $l_2\in M(q,s-q;\BC)$ and $l_3\in Q_{s-q}$ such that 
\[ z=\begin{pmatrix}l_1^*&0\\l_2^*&l_3^*\end{pmatrix}\begin{pmatrix}l_1&l_2\\0&l_3\end{pmatrix}, \]
and we change variables $x,y$ to 
\[ x'=l_1^*xl_1,\qquad y'=(l_1^*xl_1)^{-1/2}l_1^*x^{1/2}(yl_3+x^{1/2}l_2), \]
so that 
\begin{align*}
\begin{pmatrix}x'&(x')^{1/2}y'\\(y')^*(x')^{1/2}&(y')^*y'\end{pmatrix}
=&\begin{pmatrix}l_1^*xl_1&l_1^*x^{1/2}(yl_3+x^{1/2}l_2)\\
(l_3^*y^*+l_2^*x^{1/2})x^{1/2}l_1&(l_3^*y^*+l_2^*x^{1/2})(yl_3+x^{1/2}l_2)\end{pmatrix}\\
=&\begin{pmatrix}l_1^*&0\\l_2^*&l_3^*\end{pmatrix}\begin{pmatrix}x&x^{1/2}y\\y^*x^{1/2}&y^*y\end{pmatrix}
\begin{pmatrix}l_1&l_2\\0&l_3\end{pmatrix}. 
\end{align*}
Then under this change of variables, we have 
\begin{align*}
&\tau_\bk^{(s)}\begin{pmatrix}l_1&l_2\\0&l_3\end{pmatrix}\tilde{F}_{\bm,\bn}(x',y')
\tau_\bk^{(s)}\begin{pmatrix}l_1^*&0\\l_2^*&l_3^*\end{pmatrix}\\
=&\tau_\bk^{(s)}\begin{pmatrix}l_1&l_2\\0&l_3\end{pmatrix}\tau_\bk^{(s)}\begin{pmatrix}I_q&-(x')^{-1/2}y'\\0&I_{s-q}\end{pmatrix}
F_{\bm,\bn}(x')\tau_\bk^{(s)}\begin{pmatrix}I_q&0\\-(y')^*(x')^{-1/2}&I_{s-q}\end{pmatrix}
\tau_\bk^{(s)}\begin{pmatrix}l_1^*&0\\l_2^*&l_3^*\end{pmatrix}\\
=&\tau_\bk^{(s)}\begin{pmatrix}l_1&l_2\\0&l_3\end{pmatrix}
\tau_\bk^{(s)}\begin{pmatrix}I_q&-l_1^{-1}x^{-1/2}(yl_3+x^{1/2}l_2)\\0&I_{s-q}\end{pmatrix}F_{\bm,\bn}(l_1^*xl_1)\\
&\hspace{148pt}\times\tau_\bk^{(s)}\begin{pmatrix}I_q&0\\-(l_3^*y^*+l_2^*x^{1/2})x^{-1/2}l_1^{*-1}&I_{s-q}\end{pmatrix}
\tau_\bk^{(s)}\begin{pmatrix}l_1^*&0\\l_2^*&l_3^*\end{pmatrix}\\
=&\tau_\bk^{(s)}\begin{pmatrix}I_q&-x^{-1/2}y\\0&I_{s-q}\end{pmatrix}
\tau_\bk^{(s)}\begin{pmatrix}l_1&0\\0&l_3\end{pmatrix}F_{\bm,\bn}(l_1^*xl_1)
\tau_\bk^{(s)}\begin{pmatrix}l_1^*&0\\0&l_3^*\end{pmatrix}
\tau_\bk^{(s)}\begin{pmatrix}I_q&0\\-y^*x^{-1/2}&I_{s-q}\end{pmatrix}\\
=&\Delta_\bn\left(\begin{pmatrix}l_1^*&0\\l_2^*&l_3^*\end{pmatrix}\begin{pmatrix}l_1&l_2\\0&l_3\end{pmatrix}\right)\tilde{F}_{\bm,\bn}(x,y). 
\end{align*}
Thus we can compute $J(z)$ as 
\begin{align*}
J(z)=&\int_{E(I_s)}
\Tr_{V_\bk^{(s)}}\left(\tau_{\bk}^{(s)}\left(\begin{pmatrix}l_1^*&0\\l_2^*&l_3^*\end{pmatrix}
\left(I_s-\begin{pmatrix}x&x^{1/2}y\\y^*x^{1/2}&y^*y\end{pmatrix}\right)\begin{pmatrix}l_1&l_2\\0&l_3\end{pmatrix}\right)
\tilde{F}_{\bm,\bn}(x',y')\right)\\
&\hspace{90pt} \times\det\left(\begin{pmatrix}l_1^*&0\\l_2^*&l_3^*\end{pmatrix}
\left(I_s-\begin{pmatrix}x&x^{1/2}y\\y^*x^{1/2}&y^*y\end{pmatrix}\right)\begin{pmatrix}l_1&l_2\\0&l_3\end{pmatrix}\right)^{\lambda-(q+s)}\\
&\hspace{260pt} \times \det(l_1)^{2q}\det(l_3)^{2q}dxdy\\
=&\int_{E(I_s)}
\Tr_{V_\bk^{(s)}}\left(\tau_{\bk}^{(s)}\left(I_s-\begin{pmatrix}x&x^{1/2}y\\y^*x^{1/2}&y^*y\end{pmatrix}\right)
\tilde{F}_{\bm,\bn}(x,y)\right)\\
&\hspace{27pt} \times\det\left(I_s-\begin{pmatrix}x&x^{1/2}y\\y^*x^{1/2}&y^*y\end{pmatrix}\right)^{\lambda-(q+s)}
\Delta_{\lambda+\bn-s}\left(
\begin{pmatrix}l_1^*&0\\l_2^*&l_3^*\end{pmatrix}\begin{pmatrix}l_1&l_2\\0&l_3\end{pmatrix}\right)dxdy\\
=&B_{\bm,\bn}(\lambda)\Delta_{\lambda+\bn-s}(z). 
\end{align*}
Next we compute $\int_{\tilde{\Omega}}J(z)e^{-\tr(z)}dz$ in two ways. 
\begin{gather*}
\int_{\tilde{\Omega}}J(z)e^{-\tr(z)}dz=B_{\bm,\bn}(\lambda)\int_{\tilde{\Omega}}\Delta_{\lambda+\bn-s}(z)e^{-\tr(z)}
=B_{\bm,\bn}(\lambda)\Gamma_{\tilde{\Omega}}(\lambda+\bn), \\
\begin{split}
&\int_{\tilde{\Omega}}J(z)e^{-\tr(z)}dz\\
=&\iint_{E(z)}
\Tr_{V_\bk^{(s)}}\left(\tau_{\bk}^{(s)}\left(z-\begin{pmatrix}x'&(x')^{1/2}y'\\(y')^*(x')^{1/2}&(y')^*y'\end{pmatrix}\right)
\tilde{F}_{\bm,\bn}(x',y')\right)\\
&\hspace{125pt}\times\det\left(z-\begin{pmatrix}x'&(x')^{1/2}y'\\(y')^*(x')^{1/2}&(y')^*y'\end{pmatrix}\right)^{\lambda-(q+s)}
e^{-\tr(z)}dx'dy'dz\\
=&\iint\hspace{-10pt}\raisebox{-20pt}{}_{\substack{x'\in\Omega,y'\in (\fp^+_\rT)^\bot,\\z'\in\tilde{\Omega}}}\hspace{-42pt}
\Tr_{V_\bk^{(s)}}\left(\tau_{\bk}^{(s)}(z')\tilde{F}_{\bm,\bn}(x',y')\right)\det(z')^{\lambda-(q+s)}
e^{-\tr\left(z'+\left(\begin{smallmatrix}x'&(x')^{1/2}y'\\(y')^*(x')^{1/2}&(y')^*y'\end{smallmatrix}\right)\right)}dx'dy'dz'\\
=&\Tr_{V_\bk^{(s)}}\left(\int_{\tilde{\Omega}}\tau_{\bk}^{(s)}(z)\det(z)^{\lambda-(q+s)}e^{-\tr(z)}dz
\int_{\Omega\times(\fp^+_\rT)^\bot}\tilde{F}_{\bm,\bn}(x,y)
e^{-\tr\left(\begin{smallmatrix}x&x^{1/2}y\\y^*x^{1/2}&y^*y\end{smallmatrix}\right)}dxdy\right). 
\end{split}
\end{gather*}
Since $V_\bk^{(s)}$ is $U(s)$-invariant and $\int_{\tilde{\Omega}}\tau_{\bk}^{(s)}(z)\det(z)^{\lambda-(q+s)}e^{-\tr(z)}dz$ commutes with 
$U(s)$-action, this is proportional to the identity map. 
Also, similar to (\ref{gammaprime}), we can show 
\[ \int_{\tilde{\Omega}}\tau_{\bk}^{(s)}(z)\det(z)^{\lambda-(q+s)}e^{-\tr(z)}dz=\Gamma_{\tilde{\Omega}}(\lambda+\bk-q)I_{V_\bk^{(s)}} \]
when $\Re\lambda+k_s>q+s-1$. Therefore we have 
\begin{align*}
\int_{\tilde{\Omega}}J(z)e^{-\tr(z)}dz&=\Gamma_{\tilde{\Omega}}(\lambda+\bk-q)
\int_{\Omega\times(\fp^+_\rT)^\bot}\Tr_{V_\bk^{(s)}}(\tilde{F}_{\bm,\bn}(x,y))
e^{-\tr\left(\begin{smallmatrix}x&x^{1/2}y\\y^*x^{1/2}&y^*y\end{smallmatrix}\right)}dxdy\\
&=\Gamma_{\tilde{\Omega}}(\lambda+\bk-q)\Gamma_{\bm,\bn}, 
\end{align*}
and thus we get 
\begin{gather*}
B_{\bm,\bn}(\lambda)=\frac{\Gamma_{\tilde{\Omega}}(\lambda+\bk-q)}{\Gamma_{\tilde{\Omega}}(\lambda+\bn)}\Gamma_{\bm,\bn},\\
R_{\bm,\bn}(\lambda)=c_\lambda\frac{B_{\bm,\bn}(\lambda)}{\Gamma_{\bm,\bn}}
=c_\lambda\frac{\Gamma_{\tilde{\Omega}}(\lambda+\bk-q)}{\Gamma_{\tilde{\Omega}}(\lambda+\bn)}. 
\end{gather*}
Since the norm is normalized so that $R_{\bzero,\bk}(\lambda)=1$, we have 
\[ c_\lambda=\frac{\Gamma_{\tilde{\Omega}}(\lambda+\bk)}{\Gamma_{\tilde{\Omega}}(\lambda+\bk-q)}
=\prod_{j=1}^s(\lambda-(j-1)+k_j-q)_q, \]
and consequently we get 
\[ R_{\bm,\bn}(\lambda)=\frac{\Gamma_{\tilde{\Omega}}(\lambda+\bk)}{\Gamma_{\tilde{\Omega}}(\lambda+\bn)}
=\frac{\prod_{j=1}^s(\lambda-(j-1))_{k_j}}{\prod_{j=1}^s(\lambda-(j-1))_{n_j}}, \]
and we have completed the proof of Theorem \ref{suqsnontube}. \qed

\subsection{$SO^*(4r+2)$, $V=S^k(\BC^{2r+1})^\vee$}
In this subsection we set $G=SO^*(4r+2)$, which is realized explicitly as (\ref{sostarrealize}) with $s=2r+1$. 
Then we have 
\begin{gather*}
K\simeq U(2r+1),\quad \fp^\pm\simeq \mathrm{Skew}(2r+1,\BC),\\ G_\rT\simeq SO^*(4r),\quad  
L\simeq GL(r,\BH),\quad 
K_L\simeq Sp(r), \\ 
r=r,\quad n=r(2r+1),\quad d=4,\quad p=4r. 
\end{gather*}
We set $V=V_{(k,0,\ldots,0)}^{(2r+1)\vee}\simeq S^k(\BC^{2r+1})^\vee$. 
The goal of this subsection is to prove the following theorem. 
\begin{theorem}\label{sostarodd1}
When $G=SO^*(4r+2)$ and $(\tau,V)=(\tau_{(k,0,\ldots,0)}^{(2r+1)\vee},V_{(k,0,\ldots,0)}^{(2r+1)\vee})$, 
$\Vert\cdot\Vert_{\lambda,\tau}^2$ converges if $\Re\lambda>4r-1$, 
the normalizing constant $c_\lambda$ is given by 
\[ c_\lambda=(\lambda-(2r+1))(\lambda+k-2r)_{2r}\prod_{j=2}^r(\lambda-(2r+1)-2(j-1))_{2r+1}, \]
the $K$-type decomposition of $\cO(D,V)_K$ is given by 
\[ \cP(\fp^+)\otimes V_{(k,0,\ldots,0)}^{(2r+1)\vee}
=\bigoplus_{\bm\in\BZ^r_{++}}\bigoplus_{\substack{\bk\in(\BZ_{\ge 0})^{r+1};|\bk|=k\\ 0\le k_j\le m_{j-1}-m_j}}
V_{(m_1+k_1,m_1,m_2+k_2,m_2,\ldots,m_r+k_r,m_r,k_{r+1})}^{(2r+1)\vee}, \]
and for $f\in V_{(m_1+k_1,m_1,m_2+k_2,m_2,\ldots,m_r+k_r,m_r,k_{r+1})}^{(2r+1)\vee}$, the ratio of norms is given by 
\begin{align*}
\frac{\Vert f\Vert_{\lambda,\tau_{(k,0,\ldots,0)}^{(2r+1)\vee}}^2}{\Vert f\Vert_{F,\tau_{(k,0,\ldots,0)}^{(2r+1)\vee}}^2}
=&\frac{(\lambda)_k}{\prod_{j=1}^r(\lambda-2(j-1))_{m_j+k_j}(\lambda-2r)_{k_{r+1}}}\\
=&\frac{1}{(\lambda+k)_{m_1+k_1-k}\prod_{j=2}^r(\lambda-2(j-1))_{m_j+k_j}(\lambda-2r)_{k_{r+1}}}. 
\end{align*}
\end{theorem}
To begin with, we determine the normalizing constant $c_\lambda$. Since $V|_{K_\rT^\BC}$ is decomposed as 
\[ \left.V_{(k,0,\ldots,0)}^{(2r+1)\vee}\right|_{K_\rT^\BC}
=\bigoplus_{l=0}^k V_{(l,0,\ldots,0)}^{(2r)\vee}, \]
and $V_{(l,0,\ldots,0)}^{(2r)\vee}$ has the restricted lowest weight $-\left.\frac{l}{2}\gamma_1\right|_{\fa_\fl}$, 
and remains irreducible when restricted to $K_L=Sp(r)$, by Theorem \ref{keythm} 
$\Vert\cdot\Vert_{\lambda,\tau_{(k,0,\ldots,0)}^{(2r+1)\vee}}^2$ converges if $\Re\lambda>4r-1$, and we have 
\begin{align*}
c_\lambda^{-1}=&\frac{1}{\dim V_{(k,0,\ldots,0)}^{(2r+1)\vee}}\sum_{l=0}^k\left(\dim V_{(l,0,\ldots,0)}^{(2r)\vee}\right)
\frac{\Gamma_\Omega\left(\lambda+(l,0,\ldots,0)-(2r+1)\right)}{\Gamma_\Omega(\lambda+(l,0,\ldots,0))}\\
=&\frac{1}{\left(\begin{smallmatrix}2r+k\\k\end{smallmatrix}\right)}\sum_{l=0}^k
\frac{\left(\begin{smallmatrix}2r+l-1\\l\end{smallmatrix}\right)}{(\lambda+l-(2r+1))_{2r+1}}
\frac{1}{\prod_{j=2}^r(\lambda-(2r+1)-2(j-1))_{2r+1}}\\
=&\frac{1}{(\lambda-(2r+1))(\lambda+k-2r)_{2r}\prod_{j=2}^r(\lambda-(2r+1)-2(j-1))_{2r+1}}. 
\end{align*}

To compute the norm on each $K$-type, we consider $G':=SO^*(4r+4)$, which is realized explicitly as (\ref{sostarrealize}) with $s=2r+2$, 
and embed $G\hookrightarrow G'$ by 
\[ \begin{pmatrix}a&b\\c&d\end{pmatrix}\longmapsto
\begin{pmatrix}a&0&b&0\\0&1&0&0\\c&0&d&0\\0&0&0&1\end{pmatrix}
\qquad (a,b,c,d\in M(2r+1,\BC)). \]
We realize $(\tau_{(k,0,\ldots,0)}^{(2r+1)\vee},V_{(k,0,\ldots,0)}^{(2r+1)\vee})$ as 
\begin{gather*}
V_{(k,0,\ldots,0)}^{(2r+1)\vee}=\cP_k(\BC^{2r+1})=\{\text{Homogeneous holomorphic polynomials on $\BC^{2r+1}$ of degree $k$}\},\\
\tau_{(k,0,\ldots,0)}^{(2r+1)\vee}(l)p(v)=p(l^{-1}v)\qquad (l\in GL(2r+1,\BC),\; v\in\BC^{2r+1},\; p\in\cP_k(\BC^{2r+1})), 
\end{gather*}
with the inner product 
\[ (p_1,p_2)_{\tau_{(k,0,\ldots,0)}^{(2r+1)\vee}}:=\frac{1}{\pi^{2r+1}}\int_{\BC^{2r+1}}p_1(v)\overline{p_2(v)}e^{-|v|^2}dv
\qquad (p_1,p_2\in\cP_k(\BC^{2r+1})). \]
Then $\tilde{G}=\widetilde{SO^*}(4r+2)$ acts on $\cO(D,\cP_k(\BC^{2r+1}))$ by 
\begin{align*}
\tau_{\lambda}\left(\begin{pmatrix}a&b\\c&d\end{pmatrix}^{-1}\right)f(w,v)
:=\det(cw+d)^{-\lambda/2}f\left((aw+b)(cw+d)^{-1},{}^t\hspace{-1pt}(cw+d)^{-1}v\right)\\
(w\in D\subset\mathrm{Skew}(2r+1,\BC),\; v\in\BC^{2r+1}). 
\end{align*}
On the other hand, the scalar type representation of $\tilde{G}'=\widetilde{SO^*}(4r+4)$ on $\cO(D')$ 
($D'$ is realized as (\ref{sostarD}) with $s=2r+2$) is given by 
\begin{align*}
\tau_{\lambda}'\left(\begin{pmatrix}a&b\\c&d\end{pmatrix}^{-1}\right)f(w)
:=\det(cw+d)^{-\lambda/2}f\left((aw+b)(cw+d)^{-1}\right)\\
(w\in D'\subset\mathrm{Skew}(2r+2,\BC)). 
\end{align*}
If we restrict this representation to $\tilde{G}$, we have 
\begin{align*}
\tau_{\lambda}'\!\left(\!\!\begin{pmatrix}a\!&0\!&b\!&0\\0\!&1\!&0\!&0\\c\!&0\!&d\!&0\\0\!&0\!&0\!&1\end{pmatrix}
\hspace{-8pt}{\vphantom{\begin{pmatrix}0\\0\\0\\0\end{pmatrix}}}^{-1}\right)\!
f\!\begin{pmatrix}w\!&v\\-{}^t\hspace{-1pt}v\!&0\end{pmatrix}
=\det(cw+d)^{-\lambda}f\!\begin{pmatrix}(aw+b)(cw+d)^{-1}\!&{}^t\hspace{-1pt}(cw+d)^{-1}v\\-{}^t\hspace{-1pt}v(cw+d)^{-1}\!&0\end{pmatrix}\\
(w\in \mathrm{Skew}(2r+1,\BC),\; v\in\BC^{2r+1}). 
\end{align*}
Therefore if we define the embedding map $\iota:\cO(D,\cP_k(\BC^{2r+1}))\to\cO(D')$ by 
\[ (\iota(f))\begin{pmatrix}w&v\\-{}^t\hspace{-1pt}v&0\end{pmatrix}:=f(w,v)
\qquad (w\in\mathrm{Skew}(2r+1,\BC),\; v\in\BC^{2r+1}), \]
then $\iota$ intertwines two actions $\tau_{\lambda}$ and $\tau_{\lambda}'|_{\tilde{G}}$. 
Also, since Fischer inner products on $\cP(\fp^+,\cP_k(\BC^{2r+1}))$ and $\cP(\fp^{+\prime})$ 
($\fp^+=\mathrm{Skew}(2r+1,\BC)$, $\fp^{+\prime}=\mathrm{Skew}(2r+2,\BC)$) are given by 
\begin{gather*}
\langle f,g\rangle_{F,\tau_{(k,0,\ldots,0)}^{(2r+1)\vee}}
=\frac{1}{\pi^{(r+1)(2r+1)}}\int_{\mathrm{Skew}(2r+1,\BC)}\int_{\BC^{2r+1}}f(w,v)\overline{g(w,v)}e^{-\frac{1}{2}\tr(ww^*)}e^{-|v|^2}dvdw,\\
\langle f,g\rangle_{F,\bone^{(2r+2)}}
=\frac{1}{\pi^{(r+1)(2r+1)}}\int_{\mathrm{Skew}(2r+2,\BC)}f(w)\overline{g(w)}e^{-\frac{1}{2}\tr(ww^*)}dw,
\end{gather*}
$\iota$ is an isometry with respect to the Fischer inner product. 

Next, we compute the $K$-type decomposition of $\cO(D,\cP_k(\BC^{2r+1}))_K=\cP(\fp^+)\otimes\cP_k(\BC^{2r+1})$ and 
$\cO(D')_{K'}=\cP(\fp^{+\prime})$. 
\begin{align*}
\cP(\fp^+)\otimes \cP_k(\BC^{2r+1})
&=\bigoplus_{\bm\in\BZ^r_{++}}V_{(m_1,m_1,m_2,m_2,\ldots,m_r,m_r,0)}^{(2r+1)\vee}\otimes V_{(k,0,\ldots,0)}^{(2r+1)\vee}\\
&=\bigoplus_{\bm\in\BZ^r_{++}}\bigoplus_{\substack{\bk\in(\BZ_{\ge 0})^{r+1},\;|\bk|=k\\ 0\le k_j\le m_{j-1}-m_j}}
V_{(m_1+k_1,m_1,m_2+k_2,m_2,\ldots,m_r+k_r,m_r,k_{r+1})}^{(2r+1)\vee}, \\
\cP(\fp^{+\prime})&=\bigoplus_{\bn\in\BZ^{r+1}_{++}}V_{(n_1,n_1,n_2,n_2,\ldots,n_{r+1},n_{r+1})}^{(2r+2)\vee}. 
\end{align*}
Each $K^{\prime\BC}=GL(2r+2,\BC)$-module $V_{(n_1,n_1,n_2,n_2,\ldots,n_{r+1},n_{r+1})}^{(2r+2)\vee}$ 
is decomposed under $K^\BC=GL(2r+1,\BC)$ as 
\[ \left.V_{(n_1,n_1,n_2,n_2,\ldots,n_{r+1},n_{r+1})}^{(2r+2)\vee}\right|_{K^\BC}
=\bigoplus_{\substack{\bm\in\BZ^r_{++}\\ n_j\ge m_j\ge n_{j+1}}}V_{(n_1,m_1,n_2,m_2,\ldots,n_r,m_r,n_{r+1})}^{(2r+1)\vee}, \]
which follows from the following Lemma about the branching law of $GL(s,\BC)\downarrow GL(s-1,\BC)$. 
\begin{lemma}[{\cite[$\mathsection$66, Theorem 2]{Z}}]\label{gldecomp} For $\bm\in\BZ_+^s$, 
\[ \left.V_\bm^{(s)\vee}\right|_{GL(s-1,\BC)}=\bigoplus_{\substack{\bn\in\BZ_+^{s-1}\\m_j\ge n_j\ge m_{j+1}}}V_\bn^{(s-1)\vee}. \]
\end{lemma}
Therefore it follows that 
\begin{equation}\label{polyonskew}
\iota\left(V_{(m_1+k_1,m_1,\ldots,m_r+k_r,m_r,k_{r+1})}^{(2r+1)\vee}\right)
\subset V_{(m_1+k_1,m_1+k_1,\ldots,m_r+k_r,m_r+k_r,k_{r+1},k_{r+1})}^{(2r+2)\vee}. 
\end{equation}
Therefore, for any $f\in V_{(m_1+k_1,m_1,m_2+k_2,m_2,\ldots,m_r+k_r,m_r,k_{r+1})}^{(2r+1)\vee}$, the ratio of norm is given by 
\[ \frac{\Vert\iota(f)\Vert_{\lambda,\bone^{(2r+2)}}^2}{\Vert\iota(f)\Vert_{F,\bone^{(2r+2)}}^2}
=\frac{1}{\prod_{j=1}^r(\lambda-2(j-1))_{m_j+k_j}(\lambda-2r)_{k_{r+1}}}. \]
Since $\iota$ intertwines $\tilde{G}$-action, $\Vert\cdot\Vert_{\lambda,\tau_{(k,0,\ldots,0)}^{(2r+1)\vee}}$ is 
proportional to $\Vert\iota(\cdot)\Vert_{\lambda,\bone^{(2r+2)}}$. 
Also, since $\iota$ preserves the Fischer norm, and $\Vert\cdot\Vert_{\lambda,\tau_{(k,0,\ldots,0)}^{(2r+1)\vee}}$ is 
normalized such that it coincides with the Fischer norm on the minimal $K$-type, we have 
\[ \frac{\Vert f\Vert_{\lambda,\tau_{(k,0,\ldots,0)}^{(2r+1)\vee}}^2}{\Vert f\Vert_{F,\tau_{(k,0,\ldots,0)}^{(2r+1)\vee}}^2}
=\frac{(\lambda)_k}{\prod_{j=1}^r(\lambda-2(j-1))_{m_j+k_j}(\lambda-2r)_{k_{r+1}}}, \]
and we have proved Theorem \ref{sostarodd1}. \qed

\begin{remark}
We can also prove the former part of Theorem \ref{sostareven} ($G=SO^*(4r)$), 
or Theorem \ref{suqs}, \ref{suqsnontube} ($G=SU(q,s)$) by this method, by embedding 
\begin{align*}
SO^*(4r)\hookrightarrow SO^*(4r+2),&&&\cP(\mathrm{Skew}(2r,\BC),\cP_k(\BC^{2r}))\hookrightarrow \cP(\mathrm{Skew}(2r+1,\BC)),\\
U(p)\times U(q,s)\hookrightarrow U(p+q,s),&&&V_\bk^{(p)\vee}\boxtimes\cP(M(q,s,\BC),V_\bk^{(s)})\hookrightarrow\cP(M(p+q,s,\BC)),
\end{align*}
but we cannot determine the normalizing constant $c_\lambda$ in this way. 
\end{remark}

\subsection{$SO^*(4r+2)$, $V=S^k(\BC^{2r+1})\otimes\det^{-k/2}$}
In this subsection we continue to set $G=SO^*(4r+2)$, which is realized explicitly as (\ref{sostarrealize}). 
We set $V=V_{\left(\frac{k}{2},\ldots,\frac{k}{2},-\frac{k}{2}\right)}^{(2r+1)\vee}\simeq S^k(\BC^{2r+1})\otimes\det^{-k/2}$. 
The goal of this subsection is to prove the following theorem. 
\begin{theorem}\label{sostarodd2}
When $G=SO^*(4r+2)$ and $(\tau,V)=(\tau_{(k/2,\ldots,k/2,-k/2)}^{(2r+1)\vee},V_{(k/2,\ldots,k/2,-k/2)}^{(2r+1)\vee})$, 
$\Vert\cdot\Vert_{\lambda,\tau}^2$ converges if $\Re\lambda>4r-1$, 
the normalizing constant $c_\lambda$ is given by 
\[ c_\lambda=\prod_{j=1}^{r-1}(\lambda+k-(2r+1)-2(j-1))_{2r+1}(\lambda-4r+1)_{2r}(\lambda+k-2r+1), \]
the $K$-type decomposition of $\cO(D,V)_K$ is given by 
\[ \cP(\fp^+)\otimes V_{\left(\frac{k}{2},\ldots,\frac{k}{2},-\frac{k}{2}\right)}^{(2r+1)\vee}
=\!\!\bigoplus_{\bm\in\BZ^r_{++}}\bigoplus_{\substack{\bk\in(\BZ_{\ge 0})^{r+1};|\bk|=k\\ 0\le k_j\le m_j-m_{j+1}\\ 0\le k_r\le m_r}}
\!\!\!\!V_{(m_1,m_1-k_1,m_2,m_2-k_2,\ldots,m_r,m_r-k_r,-k_{r+1})+\left(\frac{k}{2},\ldots,\frac{k}{2}\right)}^{(2r+1)\vee}, \]
and for $f\in V_{(m_1,m_1-k_1,m_2,m_2-k_2,\ldots,m_r,m_r-k_r,-k_{r+1})+\left(\frac{k}{2},\ldots,\frac{k}{2}\right)}^{(2r+1)\vee}$, 
the ratio of norms is given by 
\begin{align*}
\frac{\Vert f\Vert_{\lambda,\tau_{(k/2,\ldots,k/2,-k/2)}^{(2r+1)\vee}}^2}
{\Vert f\Vert_{F,\tau_{(k/2,\ldots,k/2,-k/2)}^{(2r+1)\vee}}^2}
=&\frac{\prod_{j=1}^r\left(\lambda-2(j-1)\right)_k}{\prod_{j=1}^r\left(\lambda-2(j-1)\right)_{m_j-k_j+k}\left(\lambda-2r+1\right)_{k-k_{r+1}}}\\
=&\frac{1}{\prod_{j=1}^r\left(\lambda+k-2(j-1)\right)_{m_j-k_j}\left(\lambda-2r+1\right)_{k-k_{r+1}}}. 
\end{align*}
\end{theorem}

To begin with, we determine the normalizing constant $c_\lambda$. Since $V|_{K_\rT^\BC}$ is decomposed as 
\[ \left.V_{\left(\frac{k}{2},\ldots,\frac{k}{2},-\frac{k}{2}\right)}^{(2r+1)\vee}\right|_{K_\rT^\BC}
=\bigoplus_{l=0}^k V_{\left(\frac{k}{2},\ldots,\frac{k}{2},\frac{k}{2}-l\right)}^{(2r)\vee}, \]
and $V_{\left(\frac{k}{2},\ldots,\frac{k}{2},\frac{k}{2}-l\right)}^{(2r)\vee}$ has the restricted lowest weight 
$-\left.\left(\frac{k}{2}(\gamma_1+\cdots+\gamma_{r-1})+\frac{k-l}{2}\gamma_r\right)\right|_{\fa_\fl}$
and remains irreducible when restricted to $K_L=Sp(r)$, by Theorem \ref{keythm} 
$\Vert\cdot\Vert_{\lambda,\tau}^2$ converges if $\Re\lambda>4r-1$, and we have 
\begin{align*}
c_\lambda^{-1}=&\frac{1}{\dim V_{\left(\frac{k}{2},\ldots,\frac{k}{2},-\frac{k}{2}\right)}^{(2r+1)\vee}}
\sum_{l=0}^k\left(\dim V_{\left(\frac{k}{2},\ldots,\frac{k}{2},\frac{k}{2}-l\right)}^{(2r)\vee}\right)
\frac{\Gamma_\Omega\left(\lambda+(k,\ldots,k,k-l)-(2r+1)\right)}{\Gamma_\Omega(\lambda+(k,\ldots,k,k-l))}\\
=&\frac{1}{\left(\begin{smallmatrix}2r+k\\k\end{smallmatrix}\right)}
\frac{1}{\prod_{j=1}^{r-1}(\lambda+k-(2r+1)-2(j-1))_{2r+1}}
\sum_{l=0}^k\frac{\left(\begin{smallmatrix}2r+l-1\\l\end{smallmatrix}\right)}{(\lambda+k-l-(4r-1))_{2r+1}}\\
=&\frac{1}{\prod_{j=1}^{r-1}(\lambda+k-(2r+1)-2(j-1))_{2r+1}(\lambda-4r+1)_{2r}(\lambda+k-2r+1)}\\
=&\frac{(\lambda-2r+1)_k}{\prod_{j=1}^{r-1}(\lambda+k-(2r+1)-2(j-1))_{2r+1}(\lambda-4r+1)_{2r+1+k}}\\
=&\frac{\Gamma_\Omega(\lambda+(k,\ldots,k,0)-(2r+1))(\lambda-2r+1)_k}{\Gamma_\Omega(\lambda+(k,\ldots,k,k))}. 
\end{align*}
Next we compute the $K$-type decomposition of 
$\cO(D,V)_K=\cP(\fp^+)\otimes V_{\left(\frac{k}{2},\ldots,\frac{k}{2},-\frac{k}{2}\right)}^{(2r+1)\vee}$. 
\begin{align*}
\cP(\fp^+)\otimes V_{\left(\frac{k}{2},\ldots,\frac{k}{2},-\frac{k}{2}\right)}^{(2r+1)\vee}
&=\bigoplus_{\bm\in\BZ^r_{++}}V_{(m_1,m_1,m_2,m_2,\ldots,m_r,m_r,0)}^{(2r+1)\vee}
\otimes V_{\left(\frac{k}{2},\ldots,\frac{k}{2},-\frac{k}{2}\right)}^{(2r+1)\vee}\\
&=\!\!\bigoplus_{\bm\in\BZ^r_{++}}\bigoplus_{\substack{\bk\in(\BZ_{\ge 0})^{r+1},\;|\bk|=k\\ 0\le k_j\le m_j-m_{j+1}\\ 0\le k_r\le m_r}}
\!\!\!\!V_{(m_1,m_1-k_1,m_2,m_2-k_2,\ldots,m_r,m_r-k_r,-k_{r+1})+\left(\frac{k}{2},\ldots,\frac{k}{2}\right)}^{(2r+1)\vee}. 
\end{align*}
To apply Theorem \ref{keythm} for each $K$-type, we determine the image of each $K$-type under 
$\mathrm{rest}:\cP(\fp^+,V)\to \cP(\fp^+_\rT,V)$. Since we have 
\begin{align*}
&\mathrm{rest}\left(V_{(m_1,m_1,m_2,m_2,\ldots,m_r,m_r,0)}^{(2r+1)\vee}
\otimes V_{\left(\frac{k}{2},\ldots,\frac{k}{2},-\frac{k}{2}\right)}^{(2r+1)\vee}\right)\\
=&V_{(m_1,m_1,m_2,m_2,\ldots,m_r,m_r)}^{(2r)\vee}
\otimes\left. V_{\left(\frac{k}{2},\ldots,\frac{k}{2},-\frac{k}{2}\right)}^{(2r+1)\vee}\right|_{K_\rT^\BC}
=V_{(m_1,m_1,m_2,m_2,\ldots,m_r,m_r)}^{(2r)\vee}\otimes
\bigoplus_{l=0}^k V_{\left(\frac{k}{2},\ldots,\frac{k}{2},\frac{k}{2}-l\right)}^{(2r)\vee}\\
=&\bigoplus_{l=0}^k\bigoplus_{\substack{\bl\in(\BZ_{\ge 0})^{r},\;|\bl|=l\\ 0\le l_j\le m_j-m_{j+1}}}
V_{(m_1,m_1-l_1,m_2,m_2-l_2,\ldots,m_r,m_r-l_r)+\left(\frac{k}{2},\ldots,\frac{k}{2}\right)}^{(2r)\vee}, 
\end{align*}
and the abstract decomposition of $K^\BC$-modules under $K^\BC_\rT$ is given by Lemma \ref{gldecomp}, we have 
\begin{align*}
&\mathrm{rest}\left(V_{(m_1,m_1-k_1,m_2,m_2-k_2\ldots,m_r,m_r-k_r,-k_{r+1})+\left(\frac{k}{2},\ldots,\frac{k}{2}\right)}^{(2r+1)\vee}\right)\\
\subset&\bigoplus_{l=k-k_{r+1}}^k\bigoplus_{\substack{\bl\in(\BZ_{\ge 0})^r,\ |\bl|=l\\ k_j\le l_j\le m_j-m_{j+1}}}
V_{(m_1,m_1-l_1,m_2,m_2-l_2,\ldots,m_r,m_r-l_r)+\left(\frac{k}{2},\ldots,\frac{k}{2}\right)}^{(2r)\vee}. 
\end{align*}
Then, the only $K_L=Sp(r)$-spherical submodule in 
\begin{align*}
&V_{(m_1,m_1-l_1,m_2,m_2-l_2,\ldots,m_r,m_r-l_r)+\left(\frac{k}{2},\ldots,\frac{k}{2}\right)}^{(2r)\vee}
\otimes \overline{V_{\left(\frac{k}{2},\ldots,\frac{k}{2},\frac{k}{2}-l\right)}^{(2r)\vee}}\\
\simeq&V_{(m_1,m_1-l_1,m_2,m_2-l_2,\ldots,m_r,m_r-l_r)+\left(\frac{k}{2},\ldots,\frac{k}{2}\right)}^{(2r)\vee}
\otimes V_{\left(\frac{k}{2},\ldots,\frac{k}{2},\frac{k}{2}-l\right)}^{(2r)\vee}
\end{align*}
is $V_{(m_1-l_1,m_1-l_1,m_2-l_2,m_2-l_2,\ldots,m_r-l_r,m_r-l_r)+(k,\ldots,k)}^{(2r)\vee}$, 
which has the lowest weight $-((m_1-l_1+k)\gamma_1+\cdots+(m_r-l_r+k)\gamma_r)$. 
Therefore by Theorem \ref{keythm}, there exist non-negative numbers $a_{\bm,\bk,\bl}$ such that 
for $f\in V_{(m_1,m_1-k_1,\ldots,m_r,m_r-k_r,-k_{r+1})+\left(\frac{k}{2},\ldots,\frac{k}{2}\right)}$, 
the ratio of norms is given by 
\begin{align*}
\frac{\Vert f\Vert_{\lambda,\tau}^2}{\Vert f\Vert_{F,\tau}^2}
&=\frac{c_\lambda}{\sum_\bl a_{\bm,\bk,\bl}}
\sum_{l=k-k_{r+1}}^k\sum_{\substack{\bl\in(\BZ_{\ge 0})^r,\ |\bl|=l\\ k_j\le l_j\le m_{j+1}-m_j}}a_{\bm,\bk,\bl}
\frac{\Gamma_\Omega\left(\lambda+(k,\ldots,k,k-l)-(2r+1)\right)}{\Gamma_\Omega(\lambda+\bm-\bl+(k,\ldots,k))}\\
&=\frac{1}{\sum_\bl a_{\bm,\bk,\bl}}
\sum_{l=k-k_{r+1}}^k\sum_{\substack{\bl\in(\BZ_{\ge 0})^r,\ |\bl|=l\\ k_j\le l_j\le m_{j+1}-m_j}}
\frac{a_{\bm,\bk,\bl}(\lambda-4r+1)_{k-l}}{\prod_{j=1}^r(\lambda+k-2(j-1))_{m_j-l_j}(\lambda-2r+1)_k}. 
\end{align*}
It is difficult to know the exact values of $a_{\bm,\bk,\bl}$, but at least we have proved 
\begin{lemma}\label{sostaroddlemma}
For $f\in V_{(m_1,m_1-k_1,\ldots,m_r,m_r-k_r,-k_{r+1})+\left(\frac{k}{2},\ldots,\frac{k}{2}\right)}^{(2r+1)\vee}$, 
the ratio of norms is 
\[ \frac{\Vert f\Vert_{\lambda,\tau_{(k/2,\ldots,k/2,-k/2)}^{(2r+1)\vee}}^2}
{\Vert f\Vert_{F,\tau_{(k/2,\ldots,k/2,-k/2)}^{(2r+1)\vee}}^2}
=\frac{(\text{monic polynomial of degree $k_{r+1}$})}{\prod_{j=1}^r(\lambda+k-2(j-1))_{m_j-k_j}(\lambda-2r+1)_k}. \]
\end{lemma}

Next we consider $G_\rA:=SU(2r,1)$, which is realized as (\ref{suqsrealize}), and embed $G_\rA\hookrightarrow G$ as 
\[ \begin{pmatrix}a&b\\c&d\end{pmatrix}\longmapsto
\begin{pmatrix}a&0&0&b\\0&\bar{d}&-\bar{c}&0\\0&-\bar{b}&\bar{a}&0\\c&0&0&d\end{pmatrix}\qquad
\left(\begin{array}{c}a\in M(2r,\BC),\; b\in M(2r,1;\BC),\\ c\in M(1,2r;\BC),\; d\in \BC\end{array}\right). \]
Then the positive root system $\Delta_+(\fg_\rA^\BC,(\fh\cap\fg_\rA)^\BC)$ of $\fg_\rA$, 
induced from $\Delta_+(\fg^\BC,\fh^\BC)$, has the simple system 
\[ \{\varepsilon_j-\varepsilon_{j+1}:j=1,2,\ldots,2r-1\}\cup\{\varepsilon_{2r}+\varepsilon_{2r+1}\}. \]
Each representation of $K_\rA^\BC=S(GL(2r,\BC)\times GL(1,\BC))$ is of the form 
$(\tau_\bm^{(2r)\vee}\boxtimes\tau_{m_0}^{(1)\vee},V_\bm^{(2r)\vee}\otimes V_{m_0}^{(1)\vee})$, 
and we sometimes abbreviate this to $(\tau_{(\bm;m_0)}^{(2r,1)\vee},V_{(\bm;m_0)}^{(2r,1)\vee})$. 
Clearly $V_{(\bm+(c,\ldots,c);m_0-c)}^{(2r,1)\vee}\simeq V_{(\bm;m_0)}^{(2r,1)\vee}$ holds as $K_\rA^\BC$-modules for any $c$. 
The representation $\tau_\lambda$ of $\tilde{G}$ on $\cO(D,V)$ is given by (\ref{sostarrepn}), 
and if we restrict this representation to $\tilde{G}_\rA$, we have 
\begin{align*}
&\tau_\lambda\left(\begin{pmatrix}a&0&0&b\\0&\bar{d}&-\bar{c}&0\\0&-\bar{b}&\bar{a}&0\\c&0&0&d\end{pmatrix}^{-1}\right)
f\begin{pmatrix}w&v\\-{}^t\hspace{-1pt}v&0\end{pmatrix}\\
=&\det(a^*+vb^*)^{-\lambda/2}\det(cv+d)^{-\lambda/2}
\tau_{\left(\frac{k}{2},\ldots,\frac{k}{2},-\frac{k}{2}\right)}^{(2r+1)\vee}
\begin{pmatrix}a^*+vb^*&-w{}^t\hspace{-1pt}c\\0&{}^t\hspace{-1pt}(cv+d)\end{pmatrix}\\
&\hspace{140pt}\times
f\begin{pmatrix}(a^*+vb^*)^{-1}w{}^t\hspace{-1pt}(a^*+vb^*)^{-1}&(av+b)(cv+d)^{-1}\\-{}^t\hspace{-1pt}((av+b)(cv+d)^{-1})&0\end{pmatrix}\\
=&\det(cv+d)^{-\lambda}\tau_{\left(\frac{k}{2},\ldots,\frac{k}{2},-\frac{k}{2}\right)}^{(2r+1)\vee}
\begin{pmatrix}a^*+vb^*&-w{}^t\hspace{-1pt}c\\0&{}^t\hspace{-1pt}(cv+d)\end{pmatrix}\\
&\hspace{140pt}\times
f\begin{pmatrix}(a^*+vb^*)^{-1}w{}^t\hspace{-1pt}(a^*+vb^*)^{-1}&(av+b)(cv+d)^{-1}\\-{}^t\hspace{-1pt}((av+b)(cv+d)^{-1})&0\end{pmatrix}\\
&\hspace{275pt}(w\in\mathrm{Skew}(2r,\BC),\; v\in\BC^{2r}). 
\end{align*}
For $N\in\BN$, let $\cP_{\le N}(\mathrm{Skew}(2r,\BC))$ be the space of polynomials on $\mathrm{Skew}(2r,\BC)$ 
whose degree is smaller than or equal to $N$, and let $D_\rA\subset\BC^{2r}$ be the unit disk. 
Also, let $\mathrm{incl}:V_{(k,\ldots,k,0;0)}^{(2r,1)\vee}=V_{\left(\frac{k}{2},\ldots,\frac{k}{2},-\frac{k}{2};\frac{k}{2}\right)}^{(2r,1)\vee}
\hookrightarrow V_{\left(\frac{k}{2},\ldots,\frac{k}{2},-\frac{k}{2}\right)}^{(2r+1)\vee}$ be the $K_\rA$-equivariant inclusion. 
Then by the above computation, the map 
\begin{gather*}
\iota:\cO(D_\rA,(\cP_{\le N}(\mathrm{Skew}(2r,\BC))\boxtimes\BC)\otimes V_{(k,\ldots,k,0;0)}^{(2r,1)\vee})
\to \cO(D,V_{\left(\frac{k}{2},\ldots,\frac{k}{2},-\frac{k}{2}\right)}^{(2r+1)\vee}),\\
\iota(f)\begin{pmatrix}w&v\\-{}^t\hspace{-1pt}v&0\end{pmatrix}:=\mathrm{incl}(f(v,w))
\end{gather*}
intertwines the $G_\rA$ action, and we can also prove that $\iota$ preserves the Fischer norm. Thus we study the space 
\begin{align*}
&\cO(D_\rA,(\cP_{\le N}(\mathrm{Skew}(2r,\BC))\boxtimes\BC)\otimes V_{(k,\ldots,k,0;0)}^{(2r,1)\vee})_{K_\rA}\\
=&\cP(\BC^{2r})\otimes (\cP_{\le N}(\mathrm{Skew}(2r,\BC))\boxtimes\BC)\otimes V_{(k,\ldots,k,0;0)}^{(2r,1)\vee}\\
\simeq&\bigoplus_{m_0=0}^\infty V_{(m_0,0,\ldots,0;m_0)}^{(2r,1)\vee}\otimes 
\bigoplus_{\substack{\bm\in\BZ_{++}^r\\|\bm|\le N}}V_{(m_1,m_1,m_2,m_2,\ldots,m_r,m_r;0)}^{(2r,1)\vee}\otimes
V_{(k,\ldots,k,0;0)}^{(2r,1)\vee}. 
\end{align*}
This space is not irreducible under $G_\rA$. For $\bm\in\BZ_{++}^r$ and $\bl\in\BZ_{\ge 0}^r$ we define 
\begin{align*}
F_{\bm,\bl}:=&V_{(m_1,m_1-l_1,m_2,m_2-l_2,\ldots,m_r,m_r-l_r;0)+(k,\ldots,k;0)}^{(2r,1)\vee}\\
\subset& V_{(m_1,m_1,m_2,m_2,\ldots,m_r,m_r;0)}^{(2r,1)\vee}\otimes V_{(k,\ldots,k,0;0)}^{(2r,1)\vee}\\
\subset& (\cP_{\le N}(\mathrm{Skew}(2r,\BC))\boxtimes\BC)\otimes V_{(k,\ldots,k,0;0)}^{(2r,1)\vee}, 
\end{align*}
so that 
\begin{align*}
(\cP_{\le N}(\mathrm{Skew}(2r,\BC))\boxtimes\BC)\otimes V_{(k,\ldots,k,0;0)}^{(2r,1)\vee}
=&\bigoplus_{\substack{\bm\in\BZ_{++}^r\\|\bm|\le N}}
\bigoplus_{\substack{\bl\in\BZ_{\ge 0}^r,\; |\bl|=k\\ 0\le l_j\le m_j-m_{j+1}}} F_{\bm,\bl},\\
\cO(D_\rA,(\cP_{\le N}(\mathrm{Skew}(2r,\BC))\boxtimes\BC)\otimes V_{(k,\ldots,k,0;0)}^{(2r,1)\vee})
=&\bigoplus_{\substack{\bm\in\BZ_{++}^r\\|\bm|\le N}}
\bigoplus_{\substack{\bl\in\BZ_{\ge 0}^r,\; |\bl|=k\\ 0\le l_j\le m_j-m_{j+1}}} \cO(D_\rA,F_{\bm,\bl}). 
\end{align*}
Also, for $\bm\in\BZ_{++}^r$ and $\bk\in\BZ_{\ge 0}^{r+1}$ we set 
\begin{align*}
W_{\bm,\bk}:=&V_{(m_1-k_1,m_2,m_2-k_2,m_3,\ldots,m_{r-1}-k_{r-1},m_r,m_r-k_r,-k_{r+1};m_1)+(k,\ldots,k;0)}^{(2r,1)\vee}\\
\subset&V_{(m_1,m_2,m_2,m_3,\ldots,m_{r-1},m_r,m_r,0;m_1)}^{(2r,1)\vee}\otimes V_{(k,\ldots,k,0;0)}^{(2r,1)\vee}\\
\subset&V_{(m_1,0,\ldots,0;m_1)}^{(2r,1)\vee}\otimes V_{(m_2,m_2,m_3,m_3,\ldots,m_r,m_r,0,0;0)}^{(2r,1)\vee}
\otimes V_{(k,\ldots,k,0;0)}^{(2r,1)\vee}\\
\subset&\cP(\BC^{2r})\otimes (\cP_{\le N}(\mathrm{Skew}(2r,\BC))\boxtimes\BC)\otimes V_{(k,\ldots,k,0;0)}^{(2r,1)\vee}.
\end{align*}
Then we have the following. 
\begin{lemma}\label{propertyFW}
\begin{enumerate}
\item $\ds \iota(W_{\bm,\bk})\subset 
V_{(m_1,m_1-k_1,m_2,m_2-k_2,\ldots,m_r,m_r-k_r,-k_{r+1})+\left(\frac{k}{2},\ldots,\frac{k}{2}\right)}^{(2r+1)\vee}$. 
\item $\ds W_{\bm,\bk}\subset\bigoplus_{\substack{\bl\in(\BZ_{\ge 0})^r,\; |\bl|=k\\ l_j\le k_{j+1},\; l_r\ge k_{r+1}}}
\cO(D_\rA, F_{(m_2,\ldots,m_r,0),\bl})$. 
\item $\ds \iota(F_{\bm,\bl})\subset
\bigoplus_{\substack{\bn\in(\BZ_{\ge 0})^{r+1},\; |\bn|=k\\ n_j\le l_j,\; n_{r+1}\ge l_r-m_r}}
V_{(m_1,m_1-n_1,m_2,m_2-n_2,\ldots,m_r,m_r-n_r,-n_{r+1})+\left(\frac{k}{2},\ldots,\frac{k}{2}\right)}^{(2r+1)\vee}$. 
\end{enumerate}
\end{lemma}
\begin{proof}
(1) The polynomial space $\cP(\BC^{2r})\otimes (\cP(\mathrm{Skew}(2r,\BC))\boxtimes\BC)$ is decomposed as 
\begin{align*}
\cP(\BC^{2r})\otimes (\cP(\mathrm{Skew}(2r,\BC))\boxtimes\BC)&=\bigoplus_{m_0=0}^\infty V_{(m_0,0,\ldots,0;m_0)}^{(2r,1)\vee}
\otimes \bigoplus_{\bm\in\BZ_{++}^r}V_{(m_1,m_1,m_2,m_2,\ldots,m_r,m_r;0)}^{(2r,1)\vee}\\
&=\bigoplus_{\bm\in\BZ_{++}^r}\bigoplus_{\substack{\bl\in(\BZ_{\ge 0})^r,\; |\bl|=m_0\\ 0\le l_j\le m_{j-1}-m_j}}
V_{(m_1+l_1,m_1,m_2+l_2,m_2,\ldots,m_r+l_r,m_r;m_0)}^{(2r,1)\vee}, 
\end{align*}
and similarly to (\ref{polyonskew}), we have 
\[ V_{(m_1+l_1,m_1,m_2+l_2,m_2,\ldots,m_r+l_r,m_r;m_0)}^{(2r,1)\vee}\subset
V_{(m_1+l_1,m_1+l_1,m_2+l_2,m_2+l_2,\ldots,m_r+l_r,m_r+l_r)}^{(2r+1)\vee}. \]
Therefore we have 
\begin{align}\label{propertyiota}
&\iota\left(V_{(m_1+l_1,m_1,m_2+l_2,m_2,\ldots,m_r+l_r,m_r;m_0)}^{(2r,1)\vee}\otimes V_{(k,\ldots,k,0;0)}^{(2r,1)\vee}\right)\notag \\
\subset&V_{(m_1+l_1,m_1+l_1,m_2+l_2,m_2+l_2,\ldots,m_r+l_r,m_r+l_r,0)}^{(2r+1)\vee}\otimes 
\mathrm{incl}\left(V_{(k,\ldots,k,0;0)}^{(2r,1)\vee}\right)\notag \\
\subset&V_{(m_1+l_1,m_1+l_1,m_2+l_2,m_2+l_2,\ldots,m_r+l_r,m_r+l_r,0)}^{(2r+1)\vee}\otimes 
V_{\left(\frac{k}{2},\ldots,\frac{k}{2},-\frac{k}{2}\right)}^{(2r+1)\vee}. 
\end{align}
Especially, by putting $\bl=\bzero$ we have 
\begin{align*}
W_{\bm,\bk}\subset &V_{(m_1,m_1,m_2,m_2,\ldots,m_r,m_r,0)}^{(2r+1)\vee}\otimes\mathrm{incl}\left(V_{(k,\ldots,k,0;0)}^{(2r,1)\vee}\right)\\
\subset&V_{(m_1,m_1,m_2,m_2,\ldots,m_r,m_r,0)}^{(2r+1)\vee}\otimes V_{\left(\frac{k}{2},\ldots,\frac{k}{2},-\frac{k}{2}\right)}^{(2r+1)\vee}
\end{align*}
Let $v\in W_{\bm,\bk}$ be the highest weight vector. Then 
\begin{align*}
\iota(v)=\sum_iv_{1,i}\otimes v_{2,i}
\in&V_{(m_1,m_1,m_2,m_2,\ldots,m_r,m_r,0)}^{(2r+1)\vee}\otimes\mathrm{incl}\left(V_{(k,\ldots,k,0;0)}^{(2r,1)\vee}\right)\\
\subset&V_{(m_1,m_1,m_2,m_2,\ldots,m_r,m_r,0)}^{(2r+1)\vee}\otimes V_{\left(\frac{k}{2},\ldots,\frac{k}{2},-\frac{k}{2}\right)}^{(2r+1)\vee}
\end{align*}
has the weight $-(-k_{r+1},m_r-k_r,m_r,\ldots,m_2-k_2,m_2,m_1-k_1,m_1)-\left(\frac{k}{2},\ldots,\frac{k}{2}\right)$, 
vanishes under root vectors $x\in\fk^\BC_{\varepsilon_j-\varepsilon_{j+1}}$ ($j=1,\ldots,2r-1$) 
since $v$ is the highest under $K_\rA^\BC$, and also vanishes under root vectors $x\in\fk^\BC_{\varepsilon_{2r}-\varepsilon_{2r+1}}$ 
since each $v_{1,i},v_{2,i}$ has the weight $(*,\ldots,*,-m_1)$ and $(*,\ldots,*,0)-\left(\frac{k}{2},\ldots,\frac{k}{2}\right)$ 
respectively, where $*$ are some integers. Thus $\iota(v)$ becomes a highest weight vector of 
$V_{(m_1,m_1-k_1,m_2,m_2-k_2,\ldots,m_r,m_r-k_r,-k_{r+1})+\left(\frac{k}{2},\ldots,\frac{k}{2}\right)}^{(2r+1)\vee}$. 

(2) We have 
\begin{align*}
W_{\bm,\bl}\subset&V_{(m_1,\ldots,0;m_1)}^\vee\otimes
V_{(m_2,m_2,m_3,m_3,\ldots,m_r,m_r,0,0;0)}^{(2r,1)\vee}\otimes V_{(k,\ldots,k,0;0)}^{(2r,1)\vee}\\
=&\bigoplus_{\substack{\bl\in\BZ_{\ge 0}^r,\; |\bl|=k\\ 0\le l_j\le m_{j+1}-m_{j+2}}}V_{(m_1,\ldots,0;m_1)}^\vee\otimes
V_{(m_2,m_2-l_1,m_3,m_3-l_2,\ldots,m_r,m_r-l_{r-1},0,-l_r;0)+(k,\ldots,k;0)}^{(2r,1)\vee}\\
=&\bigoplus_{\substack{\bl\in\BZ_{\ge 0}^r,\; |\bl|=k\\ 0\le l_j\le m_{j+1}-m_{j+2}}}V_{(m_1,\ldots,0;m_1)}^\vee\otimes
F_{(m_2,\ldots,m_r,0),\bl},
\end{align*}
and abstractly 
\begin{align*}
W_{\bm,\bl}\simeq&V_{(m_1-k_1,m_2,m_2-k_2,m_3,\ldots,m_{r-1}-k_{r-1},m_r,m_r-k_r,-k_{r+1};m_1)+(k,\ldots,k;0)}^{(2r,1)\vee}\\
\subset&V_{(m_1,0,\ldots,0;m_1)}^{(2r,1)\vee}\otimes
V_{(m_2,m_2-l_1,m_3,m_3-l_2,\ldots,m_r,m_r-l_{r-1},0,-l_r;0)+(k,\ldots,k;0)}^{(2r,1)\vee}
\end{align*}
holds only if $l_j\le k_{j+1}$, $l_r\ge k_{r+1}$ holds.  

(3) By (\ref{propertyiota}) with $\bl=\bzero$ we have 
\begin{align*}
\iota(F_{\bm,\bl})\subset& V_{(m_1,m_1,m_2,m_2,\ldots,m_r,m_r)}^{(2r+1)\vee}\otimes 
V_{\left(\frac{k}{2},\ldots,\frac{k}{2},-\frac{k}{2}\right)}^{(2r+1)\vee} \\
=&\bigoplus_{\substack{\bn\in(\BZ_{\ge 0})^{r+1},\; |\bn|=k\\ n_j\le m_j-m_{j+1}}}
V_{(m_1,m_1-n_1,m_2,m_2-n_2,\ldots,m_r,m_r-n_r,-n_{r+1})+\left(\frac{k}{2},\ldots,\frac{k}{2}\right)}^{(2r+1)\vee}. 
\end{align*}
Combining with the abstract branching rule under $K^\BC\supset K_\rA^\BC$ (Lemma \ref{gldecomp}), we get the desired formula. 
\end{proof}

Now we want to show that, 
on $V_{(m_1,m_1-k_1,\ldots,m_r,m_r-k_r,-k_{r+1})+\left(\frac{k}{2},\ldots,\frac{k}{2}\right)}^{(2r+1)\vee}$ the ratio is given by 
\begin{equation}\label{indhyp}
\frac{\Vert f\Vert_{\lambda,\tau_{(k/2,\ldots,k/2,-k/2)}^{(2r+1)\vee}}^2}
{\Vert f\Vert_{F,\tau_{(k/2,\ldots,k/2,-k/2)}^{(2r+1)\vee}}^2}
=\frac{1}{\prod_{j=1}^r\left(\lambda+k-2(j-1)\right)_{m_j-k_j}\left(\lambda-2r+1\right)_{k-k_{r+1}}}
\end{equation}
by induction on $\min\{j:m_j=0\}$. 

First, when $\bm=\bzero$ i.e. on $V_{(0,\ldots,0,-k)+\frac{k}{2}}^\vee$, (\ref{indhyp}) clearly holds by the normalization assumption.  
Second, we assume (\ref{indhyp}) holds when $m_j=0$, 
and prove this also holds on 
$V_{(m_1,m_1-k_1,\ldots,m_r,m_r-k_r,-k_{r+1})+\left(\frac{k}{2},\ldots,\frac{k}{2}\right)}^{(2r+1)\vee}$ 
when $m_{j+1}=0$. 

By Lemma \ref{propertyFW} (1), it suffices to compute $\Vert\iota(f)\Vert_{\lambda,\tau}^2/\Vert\iota(f)\Vert_{F,\tau}^2$ for $f\in W_{\bm,\bk}$. 
For any $\bl$, let $f_{\bl}$ be the orthogonal of $f$ onto $\cO(D_\rA,F_{\bm',\bl})$, where $\bm':=(m_2,\ldots,m_r,0)$. 
Then by Lemma \ref{propertyFW} (2), we have 
\[ f=\sum_{\substack{\bl\in(\BZ_{\ge 0})^r,\; |\bl|=k\\ l_j\le k_{j+1},\; l_r\ge k_{r+1}}}f_\bl, \]
and there exist $b_\bl\ge 0$ such that $\Vert\iota(f_\bl)\Vert_F^2=b_\bl\Vert\iota(f)\Vert_F^2$ holds. 
Next, by Theorem \ref{suqsnontube}, we have 
\begin{align*}
&\frac{\Vert\iota(f_\bl)\Vert_{\lambda,\tau}}{\Vert\iota(f_\bl)\Vert_{F,\tau}}
\times\frac{\Vert\iota(v_\bl)\Vert_{F,\tau}}{\Vert\iota(v_\bl)\Vert_{\lambda,\tau}}\\
=&\frac{\prod_{j=1}^{r-1}((\lambda-(2j-2))_{m_{j+1}+k}(\lambda-(2j-1))_{m_{j+1}-l_j+k})(\lambda-(2r-1))_{-l_r+k}}
{\begin{array}{l}\prod_{j=1}^{r-1}((\lambda-(2j-2))_{m_j-k_j+k}(\lambda-(2j-1))_{m_{j+1}+k})\\
\hspace{180pt}\times(\lambda-(2r-2))_{m_r-k_r+k}(\lambda-(2r-1))_{-k_{r+1}+k}\end{array}}\\
=&\frac{\prod_{j=1}^{r-1}(\lambda+k-2(j-1))_{m_{j+1}}\prod_{j=2}^r(\lambda+k-(2j-3))_{m_j-l_{j-1}}(\lambda-2r+1)_{k-l_r}}
{\prod_{j=1}^r(\lambda+k-2(j-1))_{m_j-k_j}\prod_{j=2}^r(\lambda+k-(2j-3))_{m_j}(\lambda-2r+1)_{k-k_{r+1}}}, 
\end{align*}
where $v_\bl$ is any non-zero element in the minimal $K_\rA$-type $F_{\fm',\fl}$. 
Next, let $v_{\bl,\bn}$ be the orthogonal projection of $\iota(v_\bl)$ onto 
$V_{(m_2,m_2-n_1,m_3,m_3-n_2,\ldots,m_r,m_r-n_{r-1},0,0,-n_r)+\left(\frac{k}{2},\ldots,\frac{k}{2}\right)}^{(2r+1)\vee}$, so that 
\[ \iota(v_\bl)=\sum_{\substack{\bn\in(\BZ_{\ge 0})^r,\; |\bn|=k\\ n_j\le l_j,\; n_r\ge l_r}} v_{\bl,\bn} \]
by Lemma \ref{propertyFW} (3). Then there exist $c_{\bl,\bn}\ge 0$ such that 
$\Vert v_{\bl,\bn}\Vert_{F,\tau}^2=c_{\bl,\bn}\Vert\iota(v_\bl)\Vert_{F,\tau}^2$ holds. 
Next, by the induction hypothesis (\ref{indhyp}), for each $\bn$ we have 
\[ \frac{\Vert v_{\bl,\bn}\Vert_{\lambda,\tau}^2}{\Vert v_{\bl,\bn}\Vert_{F,\tau}^2}
=\frac{1}{\prod_{j=1}^{r-1}\left(\lambda+k-2(j-1)\right)_{m_{j+1}-n_j}\left(\lambda-2r+1\right)_{k-n_r}}. \]
Thus for each $\bl$ we get 
\begin{align*}
\frac{\Vert\iota(v_\bl)\Vert_{\lambda,\tau}^2}{\Vert\iota(v_\bl)\Vert_{F,\tau}^2}
=&\sum_{\substack{\bn\in(\BZ_{\ge 0})^r,\; |\bn|=k\\ n_j\le l_j,\; n_r\ge l_r}}
c_{\bl,\bn}\frac{\Vert v_{\bl,\bn}\Vert_{\lambda,\tau}^2}{\Vert v_{\bl,\bn}\Vert_{F,\tau}^2}\\
=&\sum_{\substack{\bn\in(\BZ_{\ge 0})^r,\; |\bn|=k\\ n_j\le l_j,\; n_r\ge l_r}}
\frac{c_{\bl,\bn}}{\prod_{j=1}^{r-1}\left(\lambda+k-2(j-1)\right)_{m_{j+1}-n_j}\left(\lambda-2r+1\right)_{k-n_r}}\\
=&\frac{(\text{monic polynomial of degree $k-l_r$})}{\prod_{j=1}^{r-1}\left(\lambda+k-2(j-1)\right)_{m_{j+1}}\left(\lambda-2r+1\right)_{k-l_r}}, 
\end{align*}
and therefore we get 
\begin{align*}
&\frac{\Vert\iota(f)\Vert_{\lambda,\tau}^2}{\Vert\iota(f)\Vert_{F,\tau}^2}
=\sum_{\substack{\bl\in(\BZ_{\ge 0})^r,\; |\bl|=k\\ l_j\le k_{j+1},\; l_r\ge k_{r+1}}}
b_\bl\frac{\Vert f_\bl\Vert_{\lambda,\tau}^2}{\Vert f_\bl\Vert_{F,\tau}^2}\\
=&\sum_{\substack{\bl\in(\BZ_{\ge 0})^r,\; |\bl|=k\\ l_j\le k_{j+1},\; l_r\ge k_{r+1}}}b_\bl\Biggl(
\frac{(\text{monic polynomial of degree $k-l_r$})}{\prod_{j=1}^{r-1}\left(\lambda+k-2(j-1)\right)_{m_{j+1}}\left(\lambda-2r+1\right)_{k-l_r}}\\
&\hspace{15pt}\times \frac{\prod_{j=1}^{r-1}(\lambda+k-2(j-1))_{m_{j+1}}\prod_{j=2}^r(\lambda+k-2(j-1)+1)_{m_j-l_{j-1}}(\lambda-2r+1)_{k-l_r}}
{\prod_{j=1}^r(\lambda+k-2(j-1))_{m_j-k_j}\prod_{j=2}^r(\lambda+k-(2j-3))_{m_j}(\lambda-2r+1)_{k-k_{r+1}}}\Biggr)\\
=&\frac{(\text{monic polynomial of degree $k_2+\cdots+k_r$})}
{\prod_{j=1}^r(\lambda+k-2(j-1))_{m_j-k_j}\prod_{j=2}^r(\lambda+k+m_j-k_j-(2j-3))_{k_j}(\lambda-2r+1)_{k-k_{r+1}}}.
\end{align*}
On the other hand, by Lemma \ref{sostaroddlemma} we have 
\[ \frac{\Vert\iota(f)\Vert_{\lambda,\tau}^2}{\Vert\iota(f)\Vert_{F,\tau}^2}
=\frac{(\text{monic polynomial of degree $k_{r+1}$})}{\prod_{j=1}^r(\lambda+k-2(j-1))_{m_j-k_j}(\lambda-2r+1)_k}, \]
so combining these two formulas, we get 
\[ \frac{\Vert\iota(f)\Vert_{\lambda,\tau}^2}{\Vert\iota(f)\Vert_{F,\tau}^2}
=\frac{1}{\prod_{j=1}^r(\lambda+k-2(j-1))_{m_j-k_j}(\lambda-2r+1)_{k-k_{r+1}}}, \]
and the induction continues. Thus we have proved (\ref{indhyp}) for any $\bm$, and proved Theorem \ref{sostarodd2}. \qed

\subsection{Conjecture on $E_{6(-14)}$}\label{secte6}
In this subsection we set $G=E_{6(-14)}$. Then we have 
\begin{gather*}
\fk\simeq \mathfrak{so}(2)\oplus\mathfrak{so}(10),\quad \fp^\pm\simeq M(2,1;\BO_\BC),\quad
\fg_\rT\simeq \mathfrak{so}(2,8),\quad \fl\simeq \BR\oplus\mathfrak{so}(1,7),\quad \fk_\fl\simeq \mathfrak{so}(7),\\
r=2,\quad n=16,\quad d=6,\quad p=12. 
\end{gather*}

We take a Cartan subalgebra $\fh\subset\fk$. Then we can take a basis $\{t_0,t_1,\ldots,t_5\}\subset\sqrt{-1}\fh$ 
and $\{\varepsilon_0,\varepsilon_1,\ldots,\varepsilon_5\}\subset(\sqrt{-1}\fh)^\vee$, such that 
\[ \varepsilon_0(t_j)=\frac{4}{3}\delta_{0,j},\quad \varepsilon_i(t_j)=\delta_{i,j}
\quad (i=1,\ldots,5,\; j=0,1,\ldots,5), \]
and the simple system of positive roots $\Delta_+(\fg^\BC,\fh^\BC)$ is given by 
\[ \left\{\varepsilon_1-\varepsilon_2,\;\varepsilon_2-\varepsilon_3,\;\varepsilon_3-\varepsilon_4,\;
\varepsilon_4-\varepsilon_5,\;\varepsilon_4+\varepsilon_5,\;
\frac{3}{4}\varepsilon_0+\frac{1}{2}(-\varepsilon_1-\varepsilon_2-\varepsilon_3-\varepsilon_4+\varepsilon_5)\right\}, \]
where $\frac{3}{4}\varepsilon_0+\frac{1}{2}(-\varepsilon_1-\varepsilon_2-\varepsilon_3-\varepsilon_4+\varepsilon_5)$ 
is the unique non-compact simple root, and the central character of $\fk^\BC$ is given by $d\chi=\varepsilon_0$. 
The set of strongly orthogonal roots $\{\gamma_1,\gamma_2\}\subset\Delta_{\fp^+}$ is given by 
\[ \gamma_1=\frac{3}{4}\varepsilon_0+\frac{1}{2}(\varepsilon_1+\varepsilon_2+\varepsilon_3+\varepsilon_4+\varepsilon_5),\quad
\gamma_2=\frac{3}{4}\varepsilon_0+\frac{1}{2}(\varepsilon_1-\varepsilon_2-\varepsilon_3-\varepsilon_4-\varepsilon_5), \]
and $\fh_\rT:=\fh\cap\fg_\rT$, $\fa_\fl$ is given by 
\[ \sqrt{-1}\fh_\rT=\mathrm{span}\left\{ \frac{3}{4}t_0+\frac{1}{2}t_1,\; t_2,\; t_3,\; t_4,\; t_5\right\}, \;
\fa_\fl=\mathrm{span}\left\{\frac{3}{4}t_0+\frac{1}{2}t_1,\;\frac{1}{2}(t_2+t_3+t_4+t_5)\right\}. \]
We denote the restriction of $\varepsilon_j$ to $\sqrt{-1}\fh_\rT$ by the same symbol $\varepsilon_j$ ($j=2,3,4,5$), 
and define $\varepsilon_1'\in(\sqrt{-1}\fh_\rT)^\vee$ by 
\[ \varepsilon_1'\left(\frac{3}{4}t_0+\frac{1}{2}t_1\right)=1,\quad \varepsilon_1'(t_j)=0\quad (j=2,3,4,5), \]
so that $(m_0\varepsilon_0+m_1\varepsilon_1)|_{\sqrt{-1}\fh_\rT}=\left(m_0+\frac{1}{2}m_1\right)\varepsilon_1'$ holds. 
Also, we define $\varepsilon_2^\omega,\varepsilon_3^\omega,\varepsilon_4^\omega,\varepsilon_5^\omega\in(\sqrt{-1}\fh_\rT)^\vee$ 
such that they satisfy the relations 
\begin{align*}
\varepsilon_2^\omega&=\frac{1}{2}(\varepsilon_2+\varepsilon_3+\varepsilon_4+\varepsilon_5),&
\frac{1}{2}(\varepsilon_2^\omega+\varepsilon_3^\omega+\varepsilon_4^\omega+\varepsilon_5^\omega)&=\varepsilon_2,\\
\varepsilon_2^\omega+\varepsilon_3^\omega&=\varepsilon_2+\varepsilon_3,&
\frac{1}{2}(\varepsilon_2^\omega+\varepsilon_3^\omega+\varepsilon_4^\omega-\varepsilon_5^\omega)
&=\frac{1}{2}(\varepsilon_2+\varepsilon_3+\varepsilon_4-\varepsilon_5), 
\end{align*}
so that $\gamma_1|_{\sqrt{-1}\fh_\rT}=\varepsilon_1'+\varepsilon_2^\omega$, $\gamma_2|_{\sqrt{-1}\fh_\rT}=\varepsilon_1'-\varepsilon_2^\omega$ 
holds. 

For $(m_0;\bm)\in\BC\times\left(\BZ^5\cup\left(\BZ+\frac{1}{2}\right)^5\right)$ with $m_1\ge\cdots\ge m_4\ge |m_5|$, 
let $(\tau_{(m_0;\bm)}^{[2,10]},V_{(m_0;\bm)}^{[2,10]})=(\chi^{m_0}\boxtimes \tau_\bm^{[10]},\BC_{m_0}\otimes V_\bm^{[10]})$ 
be the irreducible $\fk^\BC$-module with highest weight $m_0\varepsilon_0+m_1\varepsilon_1+\cdots+m_5\varepsilon_5$. 
Also, for $(m_0;m_1;m_2,\ldots,m_5)\in\BC\times\BC\times\left(\BZ^4\cup\left(\BZ+\frac{1}{2}\right)^4\right)$ 
with $m_2\ge m_3\ge m_4\ge |m_5|$, let $(\tau_{(m_0;m_1;m_2,\ldots,m_5)}^{[2,2,8]},V_{(m_0;m_1;m_2,\ldots,m_5)}^{[2,2,8]})$, 
$(\tau_{(m_1;m_2,\ldots,m_5)}^{[2,8]},V_{(m_1;m_2,\ldots,m_5)}^{[2,8]})$ and 
$(\tau_{(m_1;m_2,\ldots,m_5)}^{[2,8]\omega},V_{(m_1;m_2,\ldots,m_5)}^{[2,8]\omega})$ 
be the irreducible $\fk_\rT^\BC$-module with highest weight $m_0\varepsilon_0+m_1\varepsilon_1+m_2\varepsilon_2+\ldots+m_5\varepsilon_5$, 
$m_1\varepsilon_1'+m_2\varepsilon_2+\ldots+m_5\varepsilon_5$, and 
$m_1\varepsilon_1'+m_2\varepsilon_2^\omega+\ldots+m_5\varepsilon_5^\omega$ respectively. 
Then as in Section \ref{explicitroots}, we can show 
\[ (\overline{\tau_{(m_1;m_2,m_3,m_4,m_5)}^{[2,8]\omega}},\overline{V_{(m_1;m_2,m_3,m_4,m_5)}^{[2,8]\omega}})
\simeq (\tau_{(m_1;m_2,m_3,m_4,-m_5)}^{[2,8]\omega},V_{(m_1;m_2,m_3,m_4,-m_5)}^{[2,8]\omega}). \]

We set $V=V_{\left(-\frac{k}{2};k,0,0,0,0\right)}^{[2,10]}$. The goal of this subsection is to prove the following proposition. 
\begin{proposition}\label{e6(-14)}
When $G=E_{6(-14)}$ and $(\tau,V)=(\chi_{-k/2}\boxtimes\tau_{(k,0,0,0,0)}^{[10]},\BC_{-k/2}\otimes V_{(k,0,0,0,0)}^{[10]})$ 
$(k\in\BZ_{\ge 0})$, $\Vert\cdot\Vert_{\lambda,\tau}^2$ converges if $\Re\lambda>11$, the normalizing constant $c_\lambda$ is given by 
\[ c_\lambda=(\lambda-7+k)_7(\lambda-8)(\lambda-11)_7(\lambda-4+k), \]
the $K$-type decomposition of $\cO(D,V)_K$ is given by 
\begin{align*}
&\cP(\fp^+)\otimes \left(\BC_{-k/2}\boxtimes V_{(k,0,0,0,0)}^{[10]}\right)\\
=&\bigoplus_{\bm\in\BZ_{++}^2}
\bigoplus_{\substack{\bk\in(\BZ_{\ge 0})^4,\; |\bk|=k\\ k_2+k_4\le m_2\\ k_3\le m_1-m_2}}
\BC_{-\frac{3}{4}(m_1+m_2)-\frac{k}{2}}\boxtimes 
V_{\left(\frac{m_1+m_2}{2}+k_1-k_4,\frac{m_1-m_2}{2}+k_2,\frac{m_1-m_2}{2},\frac{m_1-m_2}{2},-\frac{m_1-m_2}{2}+k_3\right)}^{[10]},
\end{align*}
and for $f\in \BC_{-\frac{3}{4}(m_1+m_2)-\frac{k}{2}}\boxtimes
V_{\left(\frac{m_1+m_2}{2}+k_1-k_4,\frac{m_1-m_2}{2}+k_2,\frac{m_1-m_2}{2},\frac{m_1-m_2}{2},-\frac{m_1-m_2}{2}+k_3\right)}^{[10]}$, 
the ratio of norms is of the form 
\begin{align*}
\frac{\Vert f\Vert_{\lambda,\chi_{-k/2}\boxtimes\tau_{(k,0,0,0,0)}^{[10]}}^2}
{\Vert f\Vert_{F,\chi_{-k/2}\boxtimes\tau_{(k,0,0,0,0)}^{[10]}}^2}
=&\frac{(\lambda)_k(\lambda-3)_k(\text{monic polynomial of degree $2k_1+k_2+k_3$})}
{(\lambda)_{m_1+k_1+k_2}(\lambda-3)_{m_2+k_1+k_3}(\lambda-4)_k(\lambda-7)_k}\\
=&\frac{(\text{monic polynomial of degree $2k_1+k_2+k_3$})}
{(\lambda+k)_{m_1+k_1+k_2-k}(\lambda+k-3)_{m_2+k_1+k_3-k}(\lambda-4)_k(\lambda-7)_k}. 
\end{align*}
\end{proposition}

Before starting the proof, we quote the following lemma about the restriction of the representation $V^{[2s]}$ of $\mathfrak{so}(2s+2)$ to 
$\mathfrak{so}(2)\oplus\mathfrak{so}(2s)$. 
\begin{lemma}[{\cite[Theorem 1.1]{T}}]\label{sodecomp}
\[ \left.V_{(m_0,m_1,\ldots,m_s)}^{[2s+2]}\right|_{\mathfrak{so}(2)\oplus\mathfrak{so}(2s)}
\simeq\bigoplus_{\substack{m_{i-1}\ge n_i\ge |m_{i+1}|\\m_{s-1}\ge |n_s|}}\bigoplus_{n_0}
c^{(m_0,m_1,\ldots,m_s)}_{(n_1,\ldots,n_s)}(n_0) V_{(n_0;n_1,\ldots,n_s)}^{[2,2s]}, \]
where $c^{(m_0,m_1,\ldots,m_s)}_{(n_1,\ldots,n_s)}(n_0)\in\BZ_{\ge 0}$ is the coefficient of $X^{n_0}$ of the polynomial 
\[ X^{a_s}\prod_{j=0}^{s-1}\frac{X^{a_j+1}-X^{-a_j-1}}{X-X^{-1}}, \]
where 
\begin{align*}
a_0&=m_0-\max\{m_1,n_1\},&\\
a_j&=\min\{m_j,n_j\}-\max\{|m_{j+1}|,|n_{j+1}|\} &(j=1,\ldots,s-1),\\
a_s&=(\sgn m_s)(\sgn n_s)\min\{|m_s|,|n_s|\}.& 
\end{align*}
\end{lemma}
From this lemma we can easily deduce the following. 
\begin{lemma}
\[ \left.V_{(k,0,\ldots,0)}^{[2s+2]}\right|_{\mathfrak{so}(2)\oplus\mathfrak{so}(2s)}
=\bigoplus_{l_1=0}^k\;\bigoplus_{\substack{l_0\in\BZ,\; |l_0|\le k-l_1\\ k-l_0-l_1\in 2\BZ}}
V_{(l_0;l_1,0,\ldots,0)}^{[2,2s]}. \]
\end{lemma}

Now we start the proof. To begin with, we determine the normalizing constant $c_\lambda$. 
Since $V_{\left(-\frac{k}{2};k,0,0,0,0\right)}^{[2,10]}$ is decomposed under $\fk_\rT$ as 
\begin{align*}
&\left.V_{\left(-\frac{k}{2};k,0,0,0,0\right)}^{[2,10]}\right|_{\fk_\rT}
=\bigoplus_{l_1=0}^k\;\bigoplus_{\substack{l_0\in\BZ,\; |l_0|\le k-l_1\\ k-l_0-l_1\in 2\BZ}}
V_{\left(-\frac{k}{2};l_0;l_1,0,0,0\right)}^{[2,2,8]}
=\bigoplus_{l_1=0}^k\;\bigoplus_{\substack{l_0\in\BZ,\; |l_0|\le k-l_1\\ k-l_0-l_1\in 2\BZ}}
V_{\left(\frac{-k+l_0}{2};l_2,0,0,0\right)}^{[2,8]}\\
=&\bigoplus_{\substack{k_1,k_2\in\BZ_{\ge 0}\\k\ge k_1\ge k_2\ge 0}}
V_{\left(-\frac{k_1+k_2}{2};k_1-k_2,0,0,0\right)}^{[2,8]}
=\bigoplus_{\substack{k_1,k_2\in\BZ_{\ge 0}\\k\ge k_1\ge k_2\ge 0}}
V_{\left(-\frac{k_1+k_2}{2};\frac{k_1-k_2}{2},\frac{k_1-k_2}{2},\frac{k_1-k_2}{2},\frac{k_1-k_2}{2}\right)}^{[2,8]\omega}, 
\end{align*}
each $V_{\left(-\frac{k_1+k_2}{2};\frac{k_1-k_2}{2},\frac{k_1-k_2}{2},\frac{k_1-k_2}{2},\frac{k_1-k_2}{2}\right)}^{[2,8]\omega}$ 
remains irreducible under $\fk_\fl=\mathfrak{so}(7)$, and has the restricted lowest weight 
$-\left.\frac{1}{2}(k_1\gamma_1+k_2\gamma_2)\right|_{\fa_\fl}$, by Theorem \ref{keythm}, 
$\Vert\cdot\Vert_{\lambda,\tau}^2$ converges if $\Re\lambda>11$, and $c_\lambda$ is given by 
\begin{align*}
c_\lambda^{-1}=&\frac{1}{\dim V_{\left(-\frac{k}{2};k,0,0,0,0\right)}^{[2,10]}}\sum_{\substack{k_1,k_2\in\BZ_{\ge 0}\\k\ge k_1\ge k_2\ge 0}}
\left(\dim V_{\left(-\frac{k_1+k_2}{2};\frac{k_1-k_2}{2},\ldots,\frac{k_1-k_2}{2}\right)}^{[2,8]\omega}\right)
\frac{\Gamma_\Omega(\lambda+(k_1,k_2)-8)}{\Gamma_\Omega(\lambda+(k_1,k_2))}\\
=&\frac{1}{\left(\begin{smallmatrix}k+9\\9\end{smallmatrix}\right)-\left(\begin{smallmatrix}k+7\\9\end{smallmatrix}\right)}
\sum_{\substack{k_1,k_2\in\BZ_{\ge 0}\\k\ge k_1\ge k_2\ge 0}}
\frac{\left(\begin{smallmatrix}k_1-k_2+7\\7\end{smallmatrix}\right)-\left(\begin{smallmatrix}k_1-k_2+5\\7\end{smallmatrix}\right)}
{(\lambda+k_1-8)_8(\lambda+k_2-11)_8}. 
\end{align*}
For $l\in\BZ_{\ge 0}$, we define 
\[ F(\lambda,l):=\sum_{\substack{k_1,k_2\in\BZ_{\ge 0}\\l\ge k_1\ge k_2\ge 0}}
\frac{\left(\begin{smallmatrix}k_1-k_2+7\\7\end{smallmatrix}\right)-\left(\begin{smallmatrix}k_1-k_2+5\\7\end{smallmatrix}\right)}
{(\lambda+k_1-8)_8(\lambda+k_2-11)_8}. \]
Then it satisfies 
\begin{align*}
&F(\lambda,l+1)\\
=&\left(\sum_{l\ge k_1\ge k_2\ge 0}+\sum_{l+1\ge k_1\ge k_2\ge 1}-\sum_{l\ge k_1\ge k_2\ge 1}+\sum_{(k_1,k_2)=(l+1,0)}\right)
\frac{\left(\begin{smallmatrix}k_1-k_2+7\\7\end{smallmatrix}\right)-\left(\begin{smallmatrix}k_1-k_2+5\\7\end{smallmatrix}\right)}
{(\lambda+k_1-8)_8(\lambda+k_2-11)_8}\\
=&F(\lambda,l)+F(\lambda+1,l)-F(\lambda+1,l-1)
+\frac{\left(\begin{smallmatrix}l+8\\7\end{smallmatrix}\right)-\left(\begin{smallmatrix}l+6\\7\end{smallmatrix}\right)}
{(\lambda+l-7)_8(\lambda-11)_8}.
\end{align*}
Solving this recurrence relation, we get 
\[ F(\lambda,l)=\frac{\left(\begin{smallmatrix}l+9\\9\end{smallmatrix}\right)-\left(\begin{smallmatrix}l+7\\9\end{smallmatrix}\right)}
{(\lambda-7+l)_7(\lambda-8)(\lambda-11)_7(\lambda-4+l)}, \]
and thus we have 
\begin{align*}
c_\lambda=&(\lambda-7+k)_7(\lambda-8)(\lambda-11)_7(\lambda-4+k)
=\frac{(\lambda-8)_{k+8}(\lambda-11)_{k+8}}{(\lambda-7)_k(\lambda-4)_k}\\
=&\frac{\Gamma_\Omega(\lambda+k)}{\Gamma_\Omega(\lambda-8)(\lambda-4)_k(\lambda-7)_k}.
\end{align*}

Next we compute the $K$-type decomposition of $\cO(D,V)_K=\cP(\fp^+)\otimes V_{\left(-\frac{k}{2};k,0,0,0,0\right)}^{[2,10]}$. 
By Theorem \ref{HKS} and the ``multi-minuscule rule'' \cite[Corollary 2.16]{S}, we have 
\begin{align*}
&\cP(\fp^+)\otimes V_{\left(-\frac{k}{2};k,0,0,0,0\right)}^{[2,10]}\\
=&\bigoplus_{\bm\in\BZ_{++}^2}
V_{\left(-\frac{3}{4}(m_1+m_2);\frac{m_1+m_2}{2},\frac{m_1-m_2}{2},\frac{m_1-m_2}{2},\frac{m_1-m_2}{2},-\frac{m_1-m_2}{2}\right)}^{[2,10]}
\otimes V_{\left(-\frac{k}{2};k,0,0,0,0\right)}^{[2,10]}\\
=&\bigoplus_{\bm\in\BZ_{++}^2}
\bigoplus_{\substack{\bk\in(\BZ_{\ge 0})^4,\; |\bk|=k\\ k_2+k_4\le m_2\\ k_3\le m_1-m_2}}
V_{\left(-\frac{3}{4}(m_1+m_2)-\frac{k}{2};\frac{m_1+m_2}{2}+k_1-k_4,\frac{m_1-m_2}{2}+k_2,\frac{m_1-m_2}{2},\frac{m_1-m_2}{2},
-\frac{m_1-m_2}{2}+k_3\right)}^{[2,10]}.
\end{align*}
In order to apply Theorem \ref{keythm}, we observe the image of each $K$-type under $\mathrm{rest}:\cP(\fp^+,V)\to\cP(\fp^+_\rT,V)$. 
For each $\bm\in\BZ_{++}^2$, we have 
\begin{align*}
&\mathrm{rest}\left(V_{\left(-\frac{3}{4}(m_1+m_2);\frac{m_1+m_2}{2},\frac{m_1-m_2}{2},\frac{m_1-m_2}{2},\frac{m_1-m_2}{2},
-\frac{m_1-m_2}{2}\right)}^{[2,10]}
\otimes V_{\left(-\frac{k}{2};k,0,0,0,0\right)}^{[2,10]}\right)\\
=&V_{\left(-(m_1+m_2);\frac{m_1-m_2}{2},\frac{m_1-m_2}{2},\frac{m_1-m_2}{2},\frac{m_1-m_2}{2}\right)}^{[2,8]}
\otimes \bigoplus_{\substack{k_1',k_2'\in\BZ_{\ge 0}\\k\ge k_1'\ge k_2'\ge 0}}
V_{\left(-\frac{k_1'+k_2'}{2};k_1'-k_2',0,0,0\right)}^{[2,8]}\\
=&\bigoplus_{\substack{k_1',k_2'\in\BZ_{\ge 0}\\k\ge k_1'\ge k_2'\ge 0}}\;
\bigoplus_{\substack{l_1,l_2\in\BZ_{\ge 0}\\ l_2\le m_1-m_2\\ l_1+l_2=k_1'-k_2'}}
V_{\left(-\left(m_1+m_2+\frac{k_1'+k_2'}{2}\right);
\frac{m_1-m_2}{2}+l_1,\frac{m_1-m_2}{2},\frac{m_1-m_2}{2},\frac{m_1-m_2}{2}-l_2\right)}^{[2,8]}. 
\end{align*}
We write $k_1'+k_2'=:l_0$, so that $k_1'=\frac{1}{2}(l_0+l_1+l_2)$, $k_2'=\frac{1}{2}(l_0-l_1-l_2)$. 
By Lemma \ref{sodecomp}, 
\begin{multline*}
\mathrm{rest}\left(V_{\left(-\frac{3}{4}(m_1+m_2)-\frac{k}{2};\frac{m_1+m_2}{2}+k_1-k_4,\frac{m_1-m_2}{2}+k_2,\frac{m_1-m_2}{2},\frac{m_1-m_2}{2},
-\frac{m_1-m_2}{2}+k_3\right)}^{[2,10]}\right)\\
\cap V_{\left(-\left(m_1+m_2+\frac{l_0}{2}\right);\frac{m_1-m_2}{2}+l_1,\frac{m_1-m_2}{2},\frac{m_1-m_2}{2},\frac{m_1-m_2}{2}-l_2\right)}^{[2,8]}
\neq \{0\} 
\end{multline*}
implies 
\[ 0\le l_1\le m_2+k_1-k_4,\qquad 0\le l_2\le m_1-m_2, \]
and the coefficient of $X^{2\left(-\left(m_1+m_2+\frac{l_0}{2}\right)+\left(\frac{3}{4}(m_1+m_2)+\frac{k}{2}\right)\right)}
=X^{-\frac{m_1+m_2}{2}-l_0+k}$ of the polynomial 
\[ X^{a_4}\frac{X^{a_0+1}-X^{-a_0-1}}{X-X^{-1}}\frac{X^{a_1+1}-X^{-a_1-1}}{X-X^{-1}}\frac{X^{a_3+1}-X^{-a_3-1}}{X-X^{-1}}, \]
does not vanish, where 
\begin{align*}
a_0=&\frac{m_1+m_2}{2}+k_1-k_4-\max\left\{\frac{m_1-m_2}{2}+k_2,\frac{m_1-m_2}{2}+l_1\right\}\\
=&m_2+k_1-k_4-\max\{k_2,l_1\}, \\
a_1=&\min\left\{\frac{m_1-m_2}{2}+k_2,\frac{m_1-m_2}{2}+l_1\right\}-\frac{m_1-m_2}{2}\\
=&\min\{k_2,l_1\}, \\
a_3=&\frac{m_1-m_2}{2}-\max\left\{\left|\frac{m_1-m_2}{2}-k_3\right|, \left|\frac{m_1-m_2}{2}-l_2\right|\right\}, \\
a_4=&\sgn\!\left(-\frac{m_1-m_2}{2}+k_3\right)\sgn\!\left(\frac{m_1-m_2}{2}-l_2\right)
\min\!\left\{\left|\frac{m_1-m_2}{2}-k_3\right|, \left|\frac{m_1-m_2}{2}-l_2\right|\right\}. 
\end{align*}
This condition is satisfied only if 
\begin{align*}
-\frac{m_1+m_2}{2}-l_0+k\ge& -a_0-a_1-a_3+a_4\\
=& -\frac{m_1+m_2}{2}-k_1+k_4+|k_2-l_1|+|k_3-l_2|\\
\therefore l_0\le& k+k_1-k_4-|k_2-l_1|-|k_3-l_2|\\
=& 2k_1+k_2+k_3-|k_2-l_1|-|k_3-l_2|. 
\end{align*}
Thus we get 
\begin{align*}
&\mathrm{rest}\left(V_{\left(-\frac{3}{4}(m_1+m_2)-\frac{k}{2};\frac{m_1+m_2}{2}+k_1-k_4,\frac{m_1-m_2}{2}+k_2,\frac{m_1-m_2}{2},
\frac{m_1-m_2}{2},-\frac{m_1-m_2}{2}+k_3\right)}^{[2,10]}\right)\\
\subset&\bigoplus_{\substack{l_0,l_1,l_2\in\BZ_{\ge 0},\; l_0-l_1-l_2\in 2\BZ_{\ge 0}\\ l_1\le m_2+k_1-k_4,\; l_2\le m_1-m_2\\
l_0\le 2k_1+k_2+k_3-|k_2-l_1|-|k_3-l_2|}}
V_{\left(-\left(m_1+m_2+\frac{l_0}{2}\right);\frac{m_1-m_2}{2}+l_1,\frac{m_1-m_2}{2},\frac{m_1-m_2}{2},\frac{m_1-m_2}{2}-l_2\right)}^{[2,8]}. 
\end{align*}
For each $m_1,m_2,l_0,l_1,l_2$, we have 
\[ V_{\left(-\left(m_1+m_2+\frac{l_0}{2}\right);\frac{m_1-m_2}{2}+l_1,\frac{m_1-m_2}{2},\frac{m_1-m_2}{2},\frac{m_1-m_2}{2}-l_2\right)}^{[2,8]}\!
=\!V_{\left(-\left(m_1+m_2+\frac{l_0}{2}\right);
m_1-m_2+\frac{l_1-l_2}{2},\frac{l_1+l_2}{2},\frac{l_1+l_2}{2},\frac{l_1-l_2}{2}\right)}^{[2,8]\omega}, \]
and as in Section \ref{sectso2n}, $\fk_\fl=\mathfrak{so}(7)$-spherical irreducible submodules in 
\begin{align*}
&V_{\left(-\left(m_1+m_2+\frac{l_0}{2}\right);
m_1-m_2+\frac{l_1-l_2}{2},\frac{l_1+l_2}{2},\frac{l_1+l_2}{2},\frac{l_1-l_2}{2}\right)}^{[2,8]\omega}
\otimes \overline{V_{\left(-\frac{l_0}{2};\frac{l_1+l_2}{2},\frac{l_1+l_2}{2},\frac{l_1+l_2}{2},\frac{l_1+l_2}{2}\right)}^{[2,8]\omega}}\\
\simeq &V_{\left(-\left(m_1+m_2+\frac{l_0}{2}\right);
m_1-m_2+\frac{l_1-l_2}{2},\frac{l_1+l_2}{2},\frac{l_1+l_2}{2},\frac{l_1-l_2}{2}\right)}^{[2,8]\omega}
\otimes V_{\left(-\frac{l_0}{2};\frac{l_1+l_2}{2},\frac{l_1+l_2}{2},\frac{l_1+l_2}{2},-\frac{l_1+l_2}{2}\right)}^{[2,8]\omega}
\end{align*}
are isomorphic to $V_{(-(m_1+m_2+l_0);m_1-m_2+l_1-l_2,0,0,0)}^{[2,8]\omega}$, which has the lowest weight 
\[ -\left(m_1+\frac{l_0+l_1-l_2}{2}\right)\gamma_1-\left(m_2+\frac{l_0-l_1+l_2}{2}\right)\gamma_2. \]
Therefore for $f\in V_{\left(-\frac{3}{4}(m_1+m_2)-\frac{k}{2};\frac{m_1+m_2}{2}+k_1-k_4,\frac{m_1-m_2}{2}+k_2,\frac{m_1-m_2}{2},
\frac{m_1-m_2}{2},-\frac{m_1-m_2}{2}+k_3\right)}^{[2,10]}$, by Theorem \ref{keythm}, the ratio of norms is given by 
\begin{align*}
&\frac{\Vert f\Vert_{\lambda,\tau}}{\Vert f\Vert_{F,\tau}}
=\frac{c_\lambda}{\sum_{\bl}a_{\bm,\bk,\bl}}
\sum_{\substack{l_0,l_1,l_2\in\BZ_{\ge 0},\; l_0-l_1-l_2\in 2\BZ_{\ge 0}\\ l_1\le m_2+k_1-k_4,\; l_2\le m_1-m_2\\
l_0\le 2k_1+k_2+k_3-|k_2-l_1|-|k_3-l_2|}}
\frac{a_{\bm,\bk,\bl}\Gamma_\Omega\left(\lambda+\left(\frac{l_0+l_1+l_2}{2},\frac{l_0-l_1-l_2}{2}\right)-8\right)}
{\Gamma_\Omega\left(\lambda+\left(m_1+\frac{l_0+l_1-l_2}{2},m_2+\frac{l_0-l_1+l_2}{2}\right)\right)}\\
=&\frac{1}{\sum_{\bl}a_{\bm,\bk,\bl}}
\sum_{\substack{l_0,l_1,l_2\in\BZ_{\ge 0},\; l_0-l_1-l_2\in 2\BZ_{\ge 0}\\ l_1\le m_2+k_1-k_4,\; l_2\le m_1-m_2\\
l_0\le 2k_1+k_2+k_3-|k_2-l_1|-|k_3-l_2|}}
\frac{a_{\bm,\bk,\bl}(\lambda)_k(\lambda-3)_k(\lambda-8)_{\frac{l_0+l_1+l_2}{2}}(\lambda-11)_{\frac{l_0-l_1-l_2}{2}}}
{(\lambda)_{m_1+\frac{l_0+l_1-l_2}{2}}(\lambda-3)_{m_2+\frac{l_0-l_1+l_2}{2}}(\lambda-4)_k(\lambda-7)_k}, 
\end{align*}
using some non-negative numbers $a_{\bm,\bk,\bl}$. Now, since 
\begin{align*}
l_0+l_1-l_2\le& 2k_1+k_2+k_3-|k_2-l_1|-|k_3-l_2|+l_1-l_2\\
\le& 2k_1+2k_2-(k_2-l_1)-|k_2-l_1|+(k_3-l_2)-|k_3-l_2|\le 2(k_1+k_2),\\
l_0-l_1+l_2\le& 2k_1+k_2+k_3-|k_2-l_1|-|k_3-l_2|-l_1+l_2\\
\le& 2k_1+2k_3+(k_2-l_1)-|k_2-l_1|-(k_3-l_2)-|k_3-l_2|\le 2(k_1+k_3), 
\end{align*}
we have 
\[ \frac{\Vert f\Vert_{\lambda,\tau}}{\Vert f\Vert_{F,\tau}}
=\frac{(\lambda)_k(\lambda-3)_k(\text{monic polynomial of degree $2k_1+k_2+k_3$})}
{(\lambda)_{m_1+k_1+k_2}(\lambda-3)_{m_2+k_1+k_3}(\lambda-4)_k(\lambda-7)_k}, \]
and we have proved Proposition \ref{e6(-14)}. \qed

By $k_2+k_4\le m_2$ and $k_3\le m_1-m_2$, we have the inequality 
\[ m_1+k_1+k_2\ge m_2+k_1+k_3\ge k_2+k_3+k_4\ge k_4. \]
Thus the author conjectures the following. 
\begin{conjecture}\label{e6(-14)conj}
For $f\in \BC_{-\frac{3}{4}(m_1+m_2)-\frac{k}{2}}\boxtimes
V_{\left(\frac{m_1+m_2}{2}+k_1-k_4,\frac{m_1-m_2}{2}+k_2,\frac{m_1-m_2}{2},\frac{m_1-m_2}{2},-\frac{m_1-m_2}{2}+k_3\right)}^{[10]}$, 
the ratio of norms is given by 
\begin{align*}
\frac{\Vert f\Vert_{\lambda,\chi_{-k/2}\boxtimes\tau_{(k,0,0,0,0)}^{[10]}}^2}
{\Vert f\Vert_{F,\chi_{-k/2}\boxtimes\tau_{(k,0,0,0,0)}^{[10]}}^2}
=&\frac{(\lambda)_k(\lambda-3)_k}
{(\lambda)_{m_1+k_1+k_2}(\lambda-3)_{m_2+k_1+k_3}(\lambda-4)_{k_2+k_3+k_4}(\lambda-7)_{k_4}}\\
=&\frac{1}{(\lambda+k)_{m_1+k_1+k_2-k}(\lambda+k-3)_{m_2+k_1+k_3-k}(\lambda-4)_{k_2+k_3+k_4}(\lambda-7)_{k_4}}. 
\end{align*}
\end{conjecture}

\section{Analytic continuation of holomorphic discrete series}
In the previous sections, we calculated the norms of the holomorphic discrete series representations. 
Using this, we see how the highest weight modules behave as the parameter $\lambda$ goes small, 
following the arguments in \cite{FK0} and \cite{O}. 

For example, when $G=Sp(r,\BR)$ and $V=V_{\varepsilon_1+\cdots+\varepsilon_k}^\vee$ with $k=0,1,\ldots,r-1$, 
by Theorem \ref{sprr}, the norm $\Vert\cdot\Vert_{\lambda,\tau_{\varepsilon_1+\cdots+\varepsilon_k}^\vee}$ is written as 
\[ \Vert f\Vert_{\lambda,\tau_{\varepsilon_1+\cdots+\varepsilon_k}^\vee}^2
=\sum_{\bm\in\BZ^r_{++}}\sum_{\substack{\bk\in\{0,1\}^r,\,|\bk|=k\\ \bm+\bk\in\BZ^r_+}}
\frac{\prod_{j=1}^k\left(\lambda-\frac{1}{2}(j-1)\right)}{\prod_{j=1}^r\left(\lambda-\frac{1}{2}(j-1)\right)_{m_j+k_j}}
\Vert f_{\bm,\bk}\Vert_{F,\tau_{\varepsilon_1+\cdots+\varepsilon_k}^\vee}^2 \]
for $\lambda>r$, where $f_{\bm,\bk}$ is the orthogonal projection of $f$ onto $V_{2\bm+\bk}^\vee$. 
Then as in \cite[Theorem XIII.2.4]{FK}, the reproducing kernel $K_{\lambda,\tau_{\varepsilon_1+\cdots+\varepsilon_k}^\vee}$ 
is written by the converging sum 
\begin{align*}
K_{\lambda,\tau_{\varepsilon_1+\cdots+\varepsilon_k}^\vee}(z,w)
&=\sum_{\bm\in\BZ^r_{++}}\sum_{\substack{\bk\in\{0,1\}^r,\,|\bk|=k\\ \bm+\bk\in\BZ^r_+}}
\frac{\prod_{j=1}^r\left(\lambda-\frac{1}{2}(j-1)\right)_{m_j+k_j}}{\prod_{j=1}^k\left(\lambda-\frac{1}{2}(j-1)\right)}
K_{\bm,\bk}(z,w)
\end{align*}
where $K_{\bm,\bk}(z,w)$ is the reproducing kernel of $V_{2\bm+\bk}^\vee$ 
with respect to the Fischer norm $\Vert\cdot\Vert_{F,\tau_{\varepsilon_1+\cdots+\varepsilon_k}^\vee}^2$. 
This is continued analytically for smaller $\lambda$, and by \cite[Lemma XIII.2.6]{FK}, this is positive definite 
if and only if each coefficient is positive, that is, 
\[ \lambda\in\left\{\frac{k}{2},\frac{k+1}{2},\ldots,\frac{r-1}{2}\right\}\cup
\left(\frac{r-1}{2},\infty\right). \]
The positive definite function automatically becomes a reproducing kernel of some Hilbert space $\cH_\lambda(D,V)$, 
and this $\cH_\lambda(D,V)$ gives the unitary representation of $\tilde{G}$. 
Conversely, if there exists a unitary subrepresentation $\cH_\lambda(D,V)\subset\cO(D,V)$ for some $\lambda\in\BR$, 
then its reproducing kernel is 
automatically proportional to $K_{\lambda,\tau_{\varepsilon_1+\cdots+\varepsilon_k}^\vee}(z,w)$ by the arguments in Section \ref{HDS}, 
and thus the above condition on $\lambda$ is precisely the necessary and sufficient condition for unitarizability. 
Using this idea, we get the following result. 

\begin{theorem}
\begin{enumerate}
\item When $G=Sp(r,\BR)$ and $V=V_{\varepsilon_1+\cdots+\varepsilon_k}^\vee$ with $k=0,1,\ldots,r-1$, 
$(\tau_\lambda,\cO(D,V))$, originally unitarizable when $\lambda>r$, 
contains a non-zero unitary submodule $\cH_\lambda(D,V)$ if and only if 
\[ \lambda\in\left\{\frac{k}{2},\frac{k+1}{2},\ldots,\frac{r-1}{2}\right\}\cup
\left(\frac{r-1}{2},\infty\right). \]
\item When $G=SU(q,s)$ and $V=\BC\boxtimes V_\bk^{(s)}$ with $\bk\in\BZ^r_{++}$ 
$(k_l\ne 0,\; k_{l+1}=0,\; l=0,\ldots,s-1)$, 
$(\tau_\lambda,\cO(D,V))$, originally unitarizable when $\lambda>q+s-1$, 
contains a non-zero unitary submodule $\cH_\lambda(D,V)$ if and only if 
\[ \lambda\in\bigl\{l,l+1,\ldots,\min\{q+l,s\}-1\bigr\}\cup\bigl(\min\{q+l,s\}-1,\infty\bigr). \]
\item When $G=SO^*(2s)$ and $V=V_{(k,0,\ldots,0)}^\vee$ with $k=\BZ_{\ge 0}$, 
$(\tau_\lambda,\cO(D,V))$, originally unitarizable when $\lambda>2s-3$, 
contains a non-zero unitary submodule $\cH_\lambda(D,V)$ if and only if 
\[ \lambda\in\begin{cases}\left\{0,2,4,\ldots,2\left(\left\lfloor\frac{s}{2}\right\rfloor-1\right)\right\}
\cup\left(2\left(\left\lfloor\frac{s}{2}\right\rfloor-1\right),\infty\right)&(k=0),\\
\phantom{0,}\left\{2,4,\ldots,2\left(\left\lceil\frac{s}{2}\right\rceil-1\right)\right\}
\cup\left(2\left(\left\lceil\frac{s}{2}\right\rceil-1\right),\infty\right)&(k\ge 1).\end{cases} \]
\item When $G=SO^*(2s)$ and $V=V_{(k/2,\ldots,k/2,-k/2)}^\vee$ with $k=\BZ_{> 0}$, 
$(\tau_\lambda,\cO(D,V))$, originally unitarizable when $\lambda>2s-3$, 
contains a non-zero unitary submodule $\cH_\lambda(D,V)$ if and only if 
\[ \lambda\in\{s-2\}\cup(s-2,\infty). \]
\item When $G=Spin_0(2,n)$ and 
\[ V=\left\{\begin{array}{lll} \BC_k\boxtimes V_{(k,\ldots,k,\pm k)} & (k=\frac{1}{2}\BZ_{\ge 0}) & (n:\text{even}),\\
\BC_k\boxtimes V_{(k,\ldots,k,k)} & (k=0,\frac{1}{2}) & (n:\text{odd}),\end{array}\right. \]
$(\tau_\lambda,\cO(D,V))$, originally unitarizable when $\lambda>n-1$, 
contains a non-zero unitary submodule $\cH_\lambda(D,V)$ if and only if 
\[ \lambda\in\begin{cases} \left\{0,\frac{n-2}{2}\right\}\cup\left(\frac{n-2}{2},\infty\right) & (k=0),\\
\phantom{0,}\left\{\frac{n-2}{2}\right\}\cup\left(\frac{n-2}{2},\infty\right) & (k\ge \frac{1}{2}).\end{cases} \]
\end{enumerate}
\end{theorem}

From the explicit norm computation, we can also determine completely when the representation is reducible, 
and get some informations on the composition series, as in \cite{FK0}, \cite{O}. 
We denote the $K$-type decomposition of $\cO(D,V)_K=\cP(\fp^+,V)$ by 
\[ \cP(\fp^+,V)=\bigoplus_m W_m, \]
and for $f\in W_m$ we denote the ratio of norms by $\Vert f\Vert_{\lambda,\tau}^2/\Vert f\Vert_{F,\tau}^2=:R_m(\lambda)$, so that 
\[ \langle f,g\rangle_{\lambda,\tau}=\sum_m R_m(\lambda)\langle f_m,g_m\rangle_{F,\tau}. \]
If $\lambda$ is not a pole for all $R_m(\lambda)$, then the above sesquilinear form is well-defined, 
and non-degenerate for our cases because the numerator of each $R_m(\lambda)$ is one. 
From this we can show $(d\tau_\lambda,\cP(\fp^+,V))$ is irreducible, because 
if $\cP(\fp^+,V)$ has a proper submodule $M$, then its orthogonal complement $M^\bot$ also becomes a submodule, 
and both $M$ and $M^\bot$ contain a $\fp^+$-invariant vector i.e. contain the minimal $K$-type $V$, which is a contradiction. 
We note that in our cases the sesquilinear form is always definite on each $K$-isotypic component, 
and thus $M^\bot$ is precisely a complement vector space. 

On the other hand, if $\lambda$ is a pole for some $R_m(\lambda)$, then $(d\tau_\lambda,\cP(\fp^+,V))$ is reducible. 
In fact, for $j\in\BN$ and $\lambda\in\BR$ we define $\tilde{M}_j(\lambda)$ as the direct sum of $W_m$'s such that 
$R_m(\lambda)$ has a pole of order at most $j$ at $\lambda$. 
Then the sesquilinear form 
\begin{equation}\label{modifiedform}
\lim_{\lambda'\to\lambda}(\lambda'-\lambda)^j\langle f,g\rangle_{\lambda',\tau}
\end{equation}
is $(\fg,K)$-invariant under the representation $d\tau_\lambda$ on $\tilde{M}_j(\lambda)$, which vanishes on $\tilde{M}_{j-1}(\lambda)$. 
Thus $M_j$ is a $(\fg,K)$-submodule of $\cP(\fp^+,V)$. 
Clearly $\tilde{M}_j(\lambda)/\tilde{M}_{j-1}(\lambda)$ is infinitesimally unitary if 
the sesquilinear form (\ref{modifiedform}) is definite. This gives the following theorem. 

\begin{theorem}
\begin{enumerate}
\item When $G=Sp(r,\BR)$ and $V=V_{\varepsilon_1+\cdots+\varepsilon_k}^\vee$ with $k=0,1,\ldots,r-1$, 
for $\lambda\in\BR$ and $j=1,2,\ldots,r$, we define 
\[ M_j(\lambda):=\bigoplus_{m_j+k_j<\frac{j}{2}-\lambda+\frac{1}{2}}V_{2\bm+\bk}^\vee\subset\cP(\fp^+,V). \]
Then $(d\tau_\lambda,\cP(\fp^+,V))$ is reducible if and only if $\lambda\le\frac{r-1}{2}$ and $\lambda\in\frac{1}{2}\BZ$. 
In this case we have the sequence of submodules 
\[ \{0\}\subset M_a(\lambda)\subset M_{a+2}(\lambda)\subset\cdots\subset M_b(\lambda)\subset\cP(\fp^+,V), \]
where 
\[ a=\begin{cases} 2\lambda+1& (\frac{k}{2}\le \lambda\le \frac{r-1}{2}),\\
2\lambda+3& (0\le \lambda\le \frac{k-1}{2}),\\
1& (\lambda\le -\frac{1}{2},\; \lambda\in\BZ),\\
2& (\lambda\le -\frac{1}{2},\; \lambda\in\BZ+\frac{1}{2}), \end{cases} \qquad
b=\begin{cases} r-1& (2\lambda\equiv r \mod 2),\\
r& (2\lambda\not\equiv r \mod 2). \end{cases} \]
$M_{2\lambda+1}(\lambda)$ $(\lambda=\frac{k}{2},\frac{k+1}{2},\ldots,\frac{r-1}{2})$ and 
$\cP(\fp^+,V)/M_r(\lambda)$ $(\lambda\le \frac{r-1}{2}$, $2\lambda\not\equiv r \mod 2)$ are infinitesimally unitary. 
%
\item When $G=SU(q,s)$ and $V=\BC\boxtimes V_\bk^{(s)}$ with $\bk\in\BZ^r_{++}$, 
for $\lambda\in\BR$ and $j=1,2,\ldots,s$, we define 
\[ M_j(\lambda):=\bigoplus_{n_j<j-\lambda}c^\bn_{\bk,\bm}V_\bm^{(q)\vee}\boxtimes V_\bn^{(s)}\subset\cP(\fp^+,V). \]
Then $(d\tau_\lambda,\cP(\fp^+,V))$ is reducible if and only if $\lambda\le\min\{q+l,s\}-1$, $\lambda\in\BZ$ 
and there is no $j=q+1,\ldots,s$ such that $\lambda=j-k_j=j-k_{j-q+1}$ holds. 
In this case we have the sequence of submodules 
\[ \{0\}\subset M_a(\lambda)\subset M_{a+1}(\lambda)\subset\cdots\subset M_b(\lambda)\subset\cP(\fp^+,V), \]
where 
\[ a=\begin{cases} j+1& (j-k_j\le \lambda\le j-k_{j+1})\quad (1\le j\le \min\{q+l,s\}-1),\\
1& (\lambda\le -k_1), \end{cases} \]
and $b=s$ if $q\ge s$, 
\[ b=\begin{cases} \min\{q+l,s\}& (\min\{q+l,s\}-k_{\min\{l,s-q\}}\le \lambda\le \min\{q+l,s\}-1),\\
j& (j-k_{j-q}\le \lambda\le j-k_{j-q+1})\quad (q+1\le j\le \min\{q+l,s\}-1),\\
q& (\lambda\le q-k_1) \end{cases} \]
if $q< s$. 

If $q\ge s$ or $\bk=\bzero$, then $M_{\lambda+1}(\lambda)$ $(\lambda=l,l+1,\ldots,\min\{q,s\}-1)$ and 
$\cP(\fp^+,V)/M_{\min\{q,s\}}(\lambda)$ $(\lambda\le \min\{q,s\}-1,\; \lambda\in\BZ)$ are infinitesimally unitary. 

If $q<s$ and $\bk\ne\bzero$, then $M_{\lambda+1}(\lambda)$ $(\lambda=l,l+1,\ldots,\min\{q+l,s\}-1)$ and 
$\cP(\fp^+,V)/M_{\min\{q+l,s\}}(\lambda)$ $(\min\{q+l,s\}-k_{\min\{l,s-q\}}\le \lambda\le \min\{q+l,s\}-1,\; \lambda\in\BZ)$ 
are infinitesimally unitary. 
%
\item When $G=SO^*(4r)$ and $V=V_{(k,0,\ldots,0)}^\vee$ with $k=\BZ_{\ge 0}$, 
for $\lambda\in\BR$ and $j=1,2,\ldots,r$, we define 
\[ M_j(\lambda):=\bigoplus_{m_j+k_j<2j-\lambda-1}V_{(m_1+k_1,m_1,\ldots,m_r+k_r,m_r)}^\vee\subset\cP(\fp^+,V). \]
Then $(d\tau_\lambda,\cP(\fp^+,V))$ is reducible if and only if $\lambda\le 2r-2$ and $\lambda\in\BZ$. 
In this case we have the sequence of submodules 
\[ \{0\}\subset M_a(\lambda)\subset M_{a+1}(\lambda)\subset\cdots\subset M_r(\lambda)\subset\cP(\fp^+,V), \]
where 
\[ a=\begin{cases} \left\lceil\frac{\lambda}{2}\right\rceil+1& (3\le \lambda\le 2r-2),\\
2& (-k+1\le \lambda\le 2),\\
1& (\lambda\le -k). \end{cases} \]
$M_{\frac{\lambda}{2}+1}(\lambda)$ $(\lambda=2,4,\ldots,2r-2$ if $k\ge 1$, $\lambda=0,2,\ldots,2r-2$ if $k=0)$ and 
$\cP(\fp^+,V)/M_r(\lambda)$ $(\lambda\le 2r-2,\; \lambda\in\BZ)$ are infinitesimally unitary. 
%
\item When $G=SO^*(4r)$ and $V=V_{(k/2,\ldots,k/2,-k/2)}^\vee$ with $k=\BZ_{> 0}$, 
for $\lambda\in\BR$ and $j=1,2,\ldots,r$, we define 
\[ M_j(\lambda):=\bigoplus_{m_j-k_j+k<2j-\lambda-1}V_{(m_1,m_1-k_1,\ldots,m_r,m_r-k_r)+(k/2,\ldots,k/2)}^\vee\subset\cP(\fp^+,V). \]
Then $(d\tau_\lambda,\cP(\fp^+,V))$ is reducible if and only if $\lambda\le 2r-2$ and $\lambda\in\BZ$. 
In this case we have the sequence of submodules 
\[ \{0\}\subset M_a(\lambda)\subset M_{a+1}(\lambda)\subset\cdots\subset M_r(\lambda)\subset\cP(\fp^+,V), \]
where 
\[ a=\begin{cases} r& (2r-3-k\le \lambda\le 2r-2),\\
\left\lceil\frac{\lambda+k}{2}\right\rceil+1& (-k+1\le\lambda\le 2r-4-k),\\
1& (\lambda\le -k). \end{cases} \]
$M_r(2r-2)$ and $\cP(\fp^+,V)/M_r(\lambda)$ $(\lambda\le 2r-2,\; \lambda\in\BZ)$ are infinitesimally unitary. 
%
\item When $G=SO^*(4r+2)$ and $V=V_{(k,0,\ldots,0)}^\vee$ with $k=\BZ_{\ge 0}$, 
for $\lambda\in\BR$ and $j=1,2,\ldots,r+1$, we define 
\begin{align*}
M_j(\lambda):=&\bigoplus_{m_j+k_j<2j-\lambda-1}V_{(m_1+k_1,m_1,\ldots,m_r+k_r,m_r)}^\vee\subset\cP(\fp^+,V) &(j=1,\ldots,r),\\
M_{r+1}(\lambda):=&\bigoplus_{k_{r+1}<2r-\lambda+1}V_{(m_1+k_1,m_1,\ldots,m_r+k_r,m_r)}^\vee\subset\cP(\fp^+,V). &
\end{align*}
Then $(d\tau_\lambda,\cP(\fp^+,V))$ is reducible if and only if $\lambda\le\begin{cases}2r&(k\ge 1)\\2r-2&(k=0)\end{cases}$, $\lambda\in\BZ$ 
and $(r,\lambda)\ne (1,-k+1)$. 
In this case we have the sequence of submodules 
\[ \{0\}\subset M_a(\lambda)\subset M_{a+1}(\lambda)\subset\cdots\subset M_b(\lambda)\subset\cP(\fp^+,V), \]
where 
\[ a=\begin{cases} \left\lceil\frac{\lambda}{2}\right\rceil+1& (3\le \lambda\le 2r),\\
2& (-k+1\le \lambda\le 2),\\
1& (\lambda\le -k), \end{cases}\quad
b=\begin{cases} r+1& (2r+1-k\le \lambda\le 2r),\\
r& (\lambda\le 2r-k). \end{cases} \]
If $k=0$, then $M_{\frac{\lambda}{2}+1}(\lambda)$ $(\lambda=0,2,\ldots,2r-2)$ and 
$\cP(\fp^+,V)/M_r(\lambda)$ $(\lambda\le 2r-2,\; \lambda\in\BZ)$ are infinitesimally unitary. 

If $k\ge 1$, then $M_{\frac{\lambda}{2}+1}(\lambda)$ $(\lambda=2,4,\ldots,2r)$ and 
$\cP(\fp^+,V)/M_{r+1}(\lambda)$ $(2r+1-k\le \lambda\le 2r,\;\lambda\in\BZ)$ are infinitesimally unitary. 
%
\item When $G=SO^*(4r+2)$ and $V=V_{(k/2,\ldots,k/2,-k/2)}^\vee$ with $k=\BZ_{> 0}$, 
for $\lambda\in\BR$ and $j=1,2,\ldots,r+1$, we define 
\begin{align*}
M_j(\lambda):=&\hspace{-10pt}\bigoplus_{m_j-k_j+k<2j-\lambda-1}\hspace{-10pt}
V_{(m_1,m_1-k_1,\ldots,m_r,m_r-k_r)+(k/2,\ldots,k/2)}^\vee\subset\cP(\fp^+,V)
\;\;(j=1,\ldots,r),\\
M_{r+1}(\lambda):=&\bigoplus_{k-k_{r+1}<2r-\lambda}V_{(m_1,m_1-k_1,\ldots,m_r,m_r-k_r)+(k/2,\ldots,k/2)}^\vee\subset\cP(\fp^+,V). &
\end{align*}
Then $(d\tau_\lambda,\cP(\fp^+,V))$ is reducible if and only if $\lambda\le 2r-1$, $\lambda\in\BZ$ and $\lambda\ne 2r-k-1$. 
In this case we have the sequence of submodules 
\[ \{0\}\subset M_a(\lambda)\subset M_{a+1}(\lambda)\subset\cdots\subset M_b(\lambda)\subset\cP(\fp^+,V), \]
where 
\[ (a,b)=\begin{cases} (r+1,r+1)& (2r-k\le \lambda\le 2r-1),\\
(\left\lceil\frac{\lambda+k}{2}\right\rceil+1,r)& (-k+1\le\lambda\le 2r-2-k),\\
(1,r)& (\lambda\le -k). \end{cases} \]
$M_{r+1}(2r-1)$ and $\cP(\fp^+,V)/M_{r+1}(\lambda)$ $(2r-k\le \lambda\le 2r-1,\; \lambda\in\BZ)$ are infinitesimally unitary. 
%
\item When $G=Spin_0(2,2s)$ and 
$V=\BC_k\boxtimes V_{(k,\ldots,k,\pm k)}$ with $k=\frac{1}{2}\BZ_{\ge 0}$, 
for $\lambda\in\BR$ and $j=1,2$, we define 
\begin{align*}
M_1(\lambda)&:=\bigoplus_{m_1+k+l<1-\lambda}\BC_{m_1+m_2+k}\boxtimes V_{(m_1-m_2+l,k,\ldots,k,\pm l)}\subset\cP(\fp^+,V),\\
M_2(\lambda)&:=\bigoplus_{m_2+k-l<\frac{n}{2}-\lambda}\BC_{m_1+m_2+k}\boxtimes V_{(m_1-m_2+l,k,\ldots,k,\pm l)}\subset\cP(\fp^+,V).
\end{align*}
Then $(d\tau_\lambda,\cP(\fp^+,V))$ is reducible if and only if $\lambda\le s-1$ and $\lambda\in\BZ$. 
In this case we have the sequence of submodules 
\begin{align*}
\{0\}\subset M_2(\lambda)\subset\cP(\fp^+,V)&& &(1-2k\le \lambda\le s-1),\\
\{0\}\subset M_1(\lambda)\subset M_2(\lambda)\subset\cP(\fp^+,V)&& &(\lambda\le -2k). 
\end{align*}
$M_2(s-1)$, $M_1(0)$ (only when $k=0$), and 
$\cP(\fp^+,V)/M_2(\lambda)$ $(\lambda\le s-1,\; \lambda\in\BZ)$ are infinitesimally unitary. 
%
\item When $G=Spin_0(2,2s+1)$ and 
$V=\BC_k\boxtimes V_{(k,\ldots,k)}$ with $k=0,\frac{1}{2}$, 
for $\lambda\in\BR$ and $j=1,2$, we define 
\begin{align*}
M_1(\lambda)&:=\bigoplus_{m_1+k+l<1-\lambda}\BC_{m_1+m_2+k}\boxtimes V_{(m_1-m_2+l,k,\ldots,k,|l|)}\subset\cP(\fp^+,V),\\
M_2(\lambda)&:=\bigoplus_{m_2+k-l<\frac{n}{2}-\lambda}\BC_{m_1+m_2+k}\boxtimes V_{(m_1-m_2+l,k,\ldots,k,|l|)}\subset\cP(\fp^+,V).
\end{align*}
Then $(d\tau_\lambda,\cP(\fp^+,V))$ is reducible if and only if $\lambda\le s-\frac{1}{2}$ and $\lambda\in\BZ+\frac{1}{2}$, 
or $\lambda\le -2k$ and $\lambda\in\BZ$. 
In this case we have the sequence of submodules 
\begin{align*}
\{0\}\subset M_2(\lambda)\subset\cP(\fp^+,V)&& &(\lambda\le s-\frac{1}{2},\; \lambda\in\BZ+\frac{1}{2}),\\
\{0\}\subset M_1(\lambda)\subset\cP(\fp^+,V)&& &(\lambda\le -2k,\; \lambda\in\BZ). 
\end{align*}
$M_2(s-\frac{1}{2})$, $M_1(0)$ (only when $k=0$), and 
$\cP(\fp^+,V)/M_2(\lambda)$ $(\lambda\le s-\frac{1}{2},\; \lambda\in\BZ+\frac{1}{2})$ are infinitesimally unitary. 
\end{enumerate}
\end{theorem}

By \cite[Lemma 4.8]{KO}, we can determine the associated variety of each subquotient module by comparing the asymptotic $K$-support 
of each subquotient module and (\ref{orbitpoly}). In fact, we have 
\begin{align*}
\cV_\fg(M_{l+1}(\lambda)/M_{l\;(\mathrm{or}\; l-1)}(\lambda))&=
\begin{cases}\overline{\cO_l} & (l=0,1,\ldots,r-1),\\ \overline{\cO_r}=\fp^+ & (l\ge r), \end{cases}\\
\cV_\fg(\cP(\fp^+,V)/M_{b\;(\mathrm{or}\; r)}(\lambda))&=\overline{\cO_r}=\fp^+, 
\end{align*}
where we set $M_0(\lambda)=M_{-1}(\lambda)=\{0\}$, $\cO_l$ are defined in (\ref{orbit}), and $r=\rank_\BR G$. 
These and (\ref{orbitdim}) give the Gelfand-Kirillov dimension of each subquotient module. 
\begin{align*}
\operatorname{DIM}(M_{l+1}(\lambda)/M_{l\;(\mathrm{or}\; l-1)}(\lambda))&=
\begin{cases}l+\frac{1}{2}l(2r-l-1)d+lb & (l=0,1,\ldots,r-1),\\ r+\frac{1}{2}r(r-1)d+rb=n & (l\ge r), \end{cases}\\
\operatorname{DIM}(\cP(\fp^+,V)/M_{b\;(\mathrm{or}\; r)}(\lambda))&=r+\frac{1}{2}r(r-1)d+rb=n. 
\end{align*}

Also, we can show that the smallest submodule $M_a(\lambda)$ is irreducible in any case, by the same argument for the irreducibility of 
$\cP(\fp^+,V)$ for $\lambda$ generic case. However, we cannot determine whether the other subquotient modules are irreducible or not, 
by the norm computation, and we need some other techniques to determine the full composition series, 
such as the techniques used in e.g. \cite{MS}, \cite{OZ3}, \cite{Sa}, or \cite{BSS}.

\section*{Acknowledgments}
The author would like to thank his supervisor T. Kobayashi, and professor B. \O rsted for a lot of helpful advice on this paper. 
He also thanks his colleagues, especially M. Kitagawa for a lot of helpful discussion.

\end{document}